\documentclass[12pt,a4paper]{article} 
\usepackage[colorlinks]{hyperref}
\usepackage{amsmath}
\usepackage{amsthm}
\usepackage{amsfonts}
\usepackage{amssymb}
\usepackage{mathabx}
\usepackage{bussproofs}
\usepackage{lineno}

\usepackage{graphicx}
\usepackage{amsmath}
\usepackage{amssymb}
\usepackage{relsize}

\usepackage{url}
\usepackage{enumitem}
\usepackage{xcolor}
\usepackage{colortbl}
\usepackage{colonequals}
\usepackage[all]{xy}
\usepackage{tikz}
\usetikzlibrary{positioning,arrows,calc}

\usepackage{graphics} 
\usepackage{mathrsfs}
\usepackage{xcolor}
\theoremstyle{plain} 
\newtheorem{thm}{Theorem}[section]

\usepackage{booktabs,array}
\newcolumntype{P}[1]{>{\raggedright\arraybackslash}p{\dimexpr#1\linewidth-2\tabcolsep}}

\newtheorem{lem}[thm]{Lemma}

\newtheorem{cor}[thm]{Corollary}
\theoremstyle{definition}
\newtheorem{dfn}[thm]{Definition}
\newtheorem{exam}[thm]{Example}
\newtheorem{rem}[thm]{Remark}

\def\L{\mathcal{L}}
\def\LJ{\mathbf{LJ}}

\newcommand{\calF}{\mathcal{F}}
\newcommand{\calM}{\mathcal{M}}

\def\phi{\varphi}
\newcommand{\G}{G}

\begin{document}

\title{Universal Proof Theory: Feasible Admissibility in Intuitionistic Modal Logics} 
\author{
  Amirhossein Akbar Tabatabai\footnote{Support by the FWF project P 33548 is gratefully acknowledged.}\\
  \small{University of Groningen}
  \and
  Raheleh Jalali\footnote{Support by the Netherlands Organisation for Scientific Research under grant 639.073.807
is gratefully acknowledged.}\\
  \small{Utrecht University}
}

\date{\today}
\maketitle

\begin{abstract}
In this paper, we introduce a general family of sequent-style calculi over the modal language and its fragments to capture the essence of all constructively acceptable systems. Calling these calculi \emph{constructive}, we show that any strong enough constructive sequent calculus, satisfying a mild technical condition, feasibly admits all Visser's rules, i.e., there is a polynomial time algorithm that reads a proof of the premise of a Visser's rule and provides a proof for its conclusion. As a positive application, we show the feasible admissibility of Visser's rules in several sequent calculi for intuitionistic modal logics, including $\mathsf{CK}$, $\mathsf{IK}$ and their extensions by the modal axioms $T$, $B$, $4$, $5$, the modal axioms of bounded width and depth and the propositional lax logic. On the negative side, we show that if a strong enough intuitionistic modal logic (satisfying a mild technical condition) does not admit at least one of Visser's rules, then it cannot have a constructive sequent calculus. Consequently, no intermediate logic other than $\mathsf{IPC}$ has a constructive sequent calculus.\\

\end{abstract}

\newpage

\tableofcontents

\newpage

\section{Introduction} 
Universal proof theory \cite{Tab1,Tab} is a recent research project whose aim is to investigate the \emph{generic} behavior of proof systems, considered as independent mathematical objects standing on their own feet. It opposes the usual attitude in proof theory that sees a proof system as an auxiliary object by which we investigate the logical system it captures. Similar to any other field in mathematics (e.g. group theory), the goal in universal proof theory is to provide a \emph{classification} of proof systems of a given form up to a given equivalence. For that purpose, one must first address the following two problems: first, the \emph{existence problem} that investigates the existence of the proof systems of a given form and second, the \emph{equivalence problem}, focusing on the natural equivalence relations between the systems. 
So far, the main focus of the project has been on its first and most accessible problem, i.e., the existence problem. To attack this problem, as it is usual in mathematics, one must invoke the \emph{method of invariants}. Roughly speaking, the main idea is to show that the special form of a proof system for a logic $L$ implies a pure logical property for $L$, depending only on $L$ and not its proof systems. Therefore, the lack of the property can be used as a method to show the non-existence of the proof system.
As the first instance of implementing this method, Iemhoff  \cite{Iemhoff,Iemhoff1} and then the authors \cite{Tab1,Tab} studied the relationship between a certain general form for the rules in a sequent calculus and different flavours of the interpolation property for the corresponding logic.
In this paper, we continue this type of study by employing the \emph{admissibility of Visser's rules} as the logical property. To explain how, 
let us first provide a historical context for the admissible rules in general, Visser's rules in particular, and the corresponding complexity-theoretic issues.


A rule is called \emph{admissible} in a logic $L$ if the set of theorems of $L$ is closed under that rule. Admissible rules have been studied from various perspectives including decidability, decision and proof complexity, and their explicit bases. As an early instance of such an interest, the decidability problem for the admissible rules of the intuitionistic propositional logic, $\mathsf{IPC}$, was posed by Friedman in \cite{Friedman} and answered positively in a series of works by Rybakov gathered in \cite{Rybakov}. He showed that admissibility is decidable in many intermediate logics and various modal logics extending $\mathsf{K4}$. Later, Ghilardi \cite{Ghil1, Ghil2} provided a characterization for the admissible rules in some modal and intermediate logics using projective formulas. Building on his results, Iemhoff \cite{IemhoffAd} gave an explicit basis for $\mathsf{IPC}$ and some other intermediate logics \cite{ IemhoffAd1, IemhoffAd2}. The basis consists of what is called Visser's rules. Later,
Je\v{r}\'{a}bek employed the techniques used for $\mathsf{IPC}$ to provide a basis for a series of normal modal logics extending $\mathsf{K4}$ \cite{EmilAd}, and investigated the admissible rules of Lukasiewicz logic \cite{EmilLukasad,EmilLukasad2}. Recently, van der Giessen 
studied the admissible rules of some intuitionistic modal logics and provided a basis for them \cite{Iris}. 

Admissible rules are also interesting from a complexity-theoretic lens. Two instances of such interests are the decision and the proof complexity of the admissible rules. As for the decision complexity, Je\v{r}\'{a}ebk \cite{EmilComplexity} showed that even in basic transitive modal logics such as $\mathsf{K4, S4, GL}$, and $\mathsf{Grz}$, the decision procedure is coNEXP-complete and for Lukasiewicz logic, it is PSPACE-complete \cite{EmilLukas}. 
As for the proof complexity of admissible rules, some investigations have been done on the special case of the disjunction property. 
Buss and Mints \cite{BussMints} and later Buss and Pudl\'{a}k \cite{BussPudlak} showed that the natural deduction system and the sequent calculus for $\mathsf{IPC}$, \emph{feasibly} (i.e., in polynomial time) admit the disjunction property, respectively. More precisely, they showed that given the proof systems for $\mathsf{IPC}$, there exists a polynomial time algorithm that reads a proof of the formula $A \vee B$ and outputs a proof of either $A$ or $B$. In both papers, a form of normalization or cut elimination is required.
Later, Ferrari et al. \cite{Ferrari02,Ferrari05} provided a uniform framework to study the proof complexity of the disjunction property in intuitionistic logic and some modal and intuitionistic modal logics. The method they used is based on a calculus called the extraction calculus. The benefit of their method, compared to the previous ones, is that they do not take the structural properties of the system in use into account and hence there is no need for normalization/cut elimination. Moreover, there is a weaker version of the disjunction property suitable for the classical modal logics. For this property, the feasibility has been shown by B\'{i}lkov\'{a} \cite{Marta} and for Frege systems for any extensible modal logic, by Je\v{r}\'{a}bek \cite{Emil}. For more on the disjunction property of intermediate logics, see \cite{DPintermediate}.

\subsubsection*{Our contribution}
In this paper, we first identify a class of rules called the \emph{constructive} rules to provide a general enough family of constructively acceptable rules for intuitionistic modal logics over the language $\mathcal{L}=\{\wedge, \vee, \to, \top, \bot, \Box, \Diamond\}$ and its fragments. To find such a family, we will use the well-known heuristic that the constructive rules, whatever they are, must be careful in introducing any disjunction-like connectives. Restricting all the direct and indirect ways to produce such a connective, we will provide a tight and robust family of rules to use.
Then, as mentioned before, we set the admissibility of Visser's rules as the logical invariant property. In fact, we use a stronger property called the \emph{feasible Visser-Harrop property}. A calculus $G$ is called to have the feasible Visser-Harrop property, if
there is a polynomial time algorithm $f$ such that for any proof $\pi$ of $\Gamma, \{A_i \to B_i\}_{i\in I} \Rightarrow C \vee D$ in $G$, $f(\pi)$ is a $G$-proof of either
$\Gamma, \{A_i \to B_i\}_{i\in I} \Rightarrow C$, or $\Gamma, \{A_i \to B_i\}_{i\in I} \Rightarrow D$, or $\Gamma, \{A_i \to B_i\}_{i\in I} \Rightarrow A_i$ for some $i \in I$, where $\Gamma$ consists of a modal version of the Harrop formulas. Note that the feasible Visser-Harrop property is a property of the calculus and not the logic. In this sense it deviates from the basic setting of the invariant technique we discussed. However, forgetting the feasibility condition, it is easy to rephrase the property as a property of the logic which is a modest generalization of the admissibility of Visser's rules.

Having the two ingredients of form and property settled, we show that over $\mathcal{L}$ and its fragments, any strong enough sequent calculus, 
consisting only of constructive rules and some basic modal rules, and satisfying a mild technical condition has the feasible Visser-Harrop property and hence, its logic admits Visser's rules.
As an application, on the positive side, we show that the sequent calculi for several intuitionistic modal logics have the feasible Visser-Harrop property and consequently enjoy the feasible disjunction property. It includes the usual sequent calculi for the constructive modal logic $\mathsf{CK}$, intuitionistic modal logic $\mathsf{IK}$, their extensions by the 
modal axioms $T$, $B$, $4$, $5$ or more generally the modal axioms of bounded width and depth $ga_{klmn}: \Diamond^k \Box^l p \to \Box^m \Diamond^n p$, unless $k=m=0$, and many others. We also prove the same result for the fragments of the language, where the system for the propositional lax logic, $\mathsf{PLL}$, is also covered.
On the negative side, we show that if an intuitionistic modal logic extending $\mathsf{CK}$ does not admit Visser's rules, then it cannot have a sequent calculus as explained above. As there are many such intermediate modal logics, we provide an interesting family of non-existence results. One interesting example lives over the propositional fragment where we prove that the only intermediate logic with a calculus consisting of constructive rules is $\mathsf{IPC}$.

The method we use to prove the main result is also interesting by its own. It is inspired by the technique that Hrube\v{s} used in \cite{Hrubes} to prove an exponential lower bound on the lengths of proofs in the intuitionistic Frege system. Usually, to prove the admissibility of an admissible rule in a logic, one needs to design a sequent calculus for the logic and then eliminate the cut rule to make the proof combinatorially controllable. Such a drastic change in the proof structure is not possible in many cases if the logic has no cut-free sequent calculi. Even in the cases that it does, the process of cut elimination or the like is usually extremely costly. However, our method in the present paper is based on using translations and as translations commute with the cut rule, there is no need for any sort of cut elimination. As a consequence of this liberal attitude towards cut, the extraction method becomes feasible.

The paper is organized as follows.
In Preliminaries, Section \ref{SectionPrem}, we recall several well-known constructive and intuitionistic modal logics and their sequent calculi.
In Section \ref{SectionAlmostnegative}, we introduce the constructive rules and as a witness to their generality, we present a wide range of its instances. In Section \ref{SectionMain}, we present our main result, which is proving the feasible Visser-Harrop property for any strong enough sequent calculus consisting of constructive rules and some basic modal rules (up to a mild technical property). Finally, the analogue of the main result is provided in Subsections \ref{SubsectionDiamond},  \ref{SubsectionBox}, and  \ref{SubsectionPropositional}, for the $\Diamond$-free, $\Box$-free, and propositional fragments, respectively.\\

\noindent \textbf{Acknowledgements}
\noindent We are thankful to Rosalie Iemhoff for the interesting discussions we had on the subject.

\section{Preliminaries}
\label{SectionPrem}

Let $\mathcal{L}=\{\wedge, \vee, \to, \top, \bot, \Box, \Diamond\}$ be the language of intuitionistic modal logic.
In this paper, we mainly work over the language $\mathcal{L}$. However, we are also interested in its fragments $\mathcal{L}_{\Box}=\mathcal{L} \setminus \{\Diamond\}$, $\mathcal{L}_{\Diamond}=\mathcal{L} \setminus \{\Box\}$ and $\mathcal{L}_{p}=\mathcal{L} \setminus \{\Box, \Diamond\}$, addressed in Subsections \ref{SubsectionDiamond}, \ref{SubsectionBox}, and \ref{SubsectionPropositional}, respectively. To refer to any of these languages, we use the variable $\mathfrak{L}$ with the condition $\mathfrak{L} \in \{\mathcal{L}, \mathcal{L}_{\Box}, \mathcal{L}_{\Diamond}, \mathcal{L}_p\}$. Fixing $\mathfrak{L}$, by \emph{$\mathfrak{L}$-formulas}, we mean the formulas over the language $\mathfrak{L}$ in its usual sense. We use small Roman letters $p, q, \dots$ for \emph{atomic formulas (atoms)}. Small Greek letters $\phi, \psi, \dots$ and capital Roman letters $A, B, \dots$ denote  $\mathfrak{L}$-formulas. The bar notation, as in $\overline{\phi}, \overline{\psi}, \overline{A}, \overline{B}, \ldots$, is reserved for finite multisets of $\mathfrak{L}$-formulas. Capital Greek letters $\Gamma, \Delta, \dots$ denote finite multisets of $\mathfrak{L}$-formulas as well as multiset variables (also called contexts). The latter are the variables that can be substituted by finite multisets as it is usual in the sequent-style rules. It will always be clear from the text which one we are using. The set of
\emph{immediate subformulas} of a formula $A$ is defined recursively as follows. The immediate subformula of an atom is itself; the immediate subformulas of $A \circ B$ are $A$ and $B$, for $\circ \in \{\wedge , \vee, \to\}$ and the immediate subformula of $\bigcirc A$ is $A$, for $\bigcirc \in \{\Box, \Diamond\}$. For a formula $A$, define $\neg A$ as $A \to \bot$ and $\bigcirc^n A$ recursively by $\bigcirc^0 A=A$ and $\bigcirc^{n+1} A= \bigcirc \bigcirc^n A$, where $\bigcirc \in \{\Box, \Diamond\}$. For a multiset $\Gamma$, by $\bigcirc \Gamma$, we mean $\{\bigcirc \gamma \mid \gamma \in \Gamma\}$. And, for $\Gamma=\{\gamma_1, \dots , \gamma_n\}$, by $\bigwedge \Gamma$ (resp. $\bigvee \Gamma$), we mean $\gamma_1 \wedge \ldots \wedge \gamma_n$ (resp. $\gamma_1 \vee \ldots \vee \gamma_n$). We use the convention $\bigwedge \varnothing =\top$ and $\bigvee \varnothing =\bot$. By $\mathsf{IPC}$, we mean the intuitionistic propositional logic and by $\mathsf{CPC}$, the classical propositional logic.

\subsection{Intuitionistic Modal Logics} \label{SubsectionPreliminariesIntuitionisticModal}
We begin with a definition of logics over a given language $\mathfrak{L} \in \{\mathcal{L}, \mathcal{L}_\Box, \mathcal{L}_\Diamond, \mathcal{L}_p\}$.
\begin{dfn}\label{DefLogic}
Let $\mathfrak{L} \in \{\mathcal{L}, \mathcal{L}_\Box, \mathcal{L}_\Diamond, \mathcal{L}_p\}$ be a language. A \emph{logic $L$ over $\mathfrak{L}$ }is a set of $\mathfrak{L}$-formulas closed under the substitution, modus ponens and
\begin{itemize}
\item
the necessitation rule, if $\mathfrak{L}=\mathcal{L}$ or $\mathfrak{L}=\mathcal{L}_{\Box}$,
\item
no other rule, if $\mathfrak{L}=\mathcal{L}_\Diamond$ or $\mathfrak{L}=\mathcal{L}_p$, where 
\end{itemize}

\begin{itemize}
\item[$\triangleright$]
the \emph{modus ponens} rule means that if $\phi, \phi \to \psi \in L$ then $\psi \in L$,
\item[$\triangleright$]
the \emph{necessitation} rule means that if $\phi \in L$ then $\Box \phi \in L$.
\end{itemize}
Let $L_1$ and $L_2$ be logics over the language $\mathfrak{L}$. Then, $L_2$ \emph{extends} $L_1$, if $L_1 \subseteq L_2$. Given a logic $L$ defined over $\mathfrak{L}$ and a set of $\mathfrak{L}$-formulas $\{A_i\}_{i \in I}$, by $L +\{A_i\}_{i\in I}$, we mean the smallest logic over $\mathfrak{L}$, extending $L$ and containing all the formulas in $\{A_i\}_{i\in I}$.
\end{dfn}
To introduce intuitionistic modal logics, notice that over the language $\mathcal{L}$, the modalities $\Box$ and $\Diamond$ are not supposed to be the dual of each other, i.e, $\Diamond$ is not equivalent to $\neg \Box \neg$. Consequently, there are many possibilities to define the intuitionistic counterpart of a given classical modal logic (e.g., see \cite{Wolter}). 
In the following, we introduce some intuitionistic modal logics starting from two well-known intuitionistic counterparts of the classical modal logic $\mathsf{K}$. First, consider the family of axioms in Table \ref{tableAxiom}.
\begin{table}[!ht]
\footnotesize{\begin{tabular}{ |c| c| c| c| }
\hline
name & axiom & name & axiom\\
\hline
$K_a$ & $\Box(p \to q) \to (\Box p \to \Box q)$ & $K_b$ & $\Box(p \to q) \to(\Diamond p \to \Diamond q)$\\
\hline
$\Diamond \bot$ & $\neg \Diamond \bot$ & $\Diamond \vee$ & $\Diamond (p \vee q) \to \Diamond p \vee \Diamond q$\\
\hline
$\Box \to $ & $(\Diamond p \to \Box q)\to \Box (p \to q)$ & $D$&$\Box p \to \Diamond p$\\
\hline
$D_a$ & $\Box \bot \to \bot$ &  $D_b$ & $\top \to \Diamond \top$ \\
\hline
$T_a$ & $\Box p \to p$ &  $T_b$ & $p \to \Diamond p$ \\
\hline
$4_a$ &
$\Box p \to \Box \Box p$ & $4_b$ &$\Diamond \Diamond p \to \Diamond p$ \\
\hline
$B_a$ & $\Diamond \Box p \to p$ & $B_b$ &$p \to \Box \Diamond p$\\
\hline
$5_a$ & $\Diamond \Box p \to \Box p$ & $5_b$ & $\Diamond p \to \Box \Diamond p$\\
\hline
$c_a$ & $p \to \Box p$ & $c_b$ & $\Diamond p \to p$\\
\hline
$den_{n,a}$ & $\Box^{n+1} p \to \Box^n p$ & $den_{n,b}$ & $\Diamond^n p \to \Diamond^{n+1} p$\\
\hline
$4_{n,m,a}$ & $\Box^{n} p \to \Box^m p$ & $4_{n,m,b}$ & $\Diamond^m p \to \Diamond^{n} p$ \\
\hline
$tra_{n,a}$ & $\bigwedge_{i=1}^n \Box^{i} p \to \Box^{n+1} p$ & $tra_{n,b}$ & $\Diamond^{n+1} p \to \bigvee_{i=0}^n \Diamond^{i} p$ \\
\hline
$ga$ & $\Diamond \Box p \to \Box \Diamond p$ & $ga_{klmn}$ & $\Diamond^k \Box^l p \to \Box^m \Diamond^n p$\\
\hline
$.2$ & $\Diamond (p \wedge \Box q) \to \Box (p \vee \Diamond q)$ & $d_1$ & $\neg \Diamond p \to \Box \neg p$ \\
\hline
$d_2$ & $\Box \neg p \to \neg \Diamond p$ & $d_3$ & $\Diamond \neg p \to \neg \Box p$ \\
\hline
$bw_r$ & $\bigwedge_{i=0}^r \Diamond p_i \to \bigvee_{0 \leq i \neq j}^r \Diamond (p_i \wedge (p_j \vee \Diamond p_j))$ & $H$ & $p \to \Box (\Diamond p \to p)$\\
\hline
$bd_n$ & $\Diamond (\Box p_n \wedge bd_n) \to p_n$  & $dir$ & $\Diamond (\Box p \wedge q) \to \Box (\Diamond p \vee q)$\\
\hline
$M^{\to}_{\Diamond}$ & $(p \to q) \to (\Diamond p \to \Diamond q)$  &   &  \\
\hline
\end{tabular}}
\caption{\small{Some modal axioms. Everywhere, we assume $n \geq 0$ and $r\geq 1$. In $4_{n, m, a}$ and $4_{n, m, b}$ we assume $0 \leq n < m$ and in $ga_{klmn}$ we have $k,l,m,n \geq 0$. The formula $bd_{n}$ is defined recursively: $bd_0:=\bot$, and $bd_{n+1}:=p_n \vee \Box (\Diamond \neg p_n \vee bd_n)$.}}\label{tableAxiom}
\end{table}
Most of them 
are well-known for the properties they impose on the classical Kripke frames that validate them: the axiom $(dir)$, when added to the logic $\mathsf{K}$, is sound and complete with respect to directed Kripke frames. Or, a Kripke frame $\mathcal{F}$ validates ($T_a$) if and only if $\mathcal{F}$ is reflexive. Note that beside their classical importance, the axioms $ga_{klmn}$, for $0 \leq k,l,m,n$, have been studied both syntactically and semantically in the realm of intuitionistic modal logics (e.g., see \cite{Wolter}).

Over the language $\mathcal{L}$, there are two basic modal logics considered as the intuitionistic counterparts of the logic $\mathsf{K}$. The first, is the \emph{constructive} $\mathsf{K}$, denoted by $\mathsf{CK}$, studied for instance in \cite{bierman}, and the second is the \emph{intuitionistic} $\mathsf{K}$, denoted by $\mathsf{IK}$, introduced in \cite{Fischer1,Fischer2} and studied in \cite{Simpson} in detail:
 \begin{center}
$\mathsf{CK}:=  \mathsf{IPC} + \{K_a, K_b\} \quad \quad \quad \mathsf{IK}:= \mathsf{IPC} +\{ K_a , K_b , \Diamond \vee , \Box \to , \Diamond \bot\}$
\end{center}
\normalsize
$\mathsf{CK}$ is the weakest intuitionistic modal logic considered in this paper. Using the axioms in Table \ref{tableAxiom}, we can design  constructive versions of well-known classical modal logics. To name a few, let $X \subseteq \{T, B, 4, 5\}$. By $\mathsf{CK}X$, we mean the smallest logic over $\mathcal{L}$ extending $\mathsf{CK}$ and all the axioms in $X$ in both $(a)$ and $(b)$ versions.
An interesting logic is the constructive counterpart of $\mathsf{S4}$, introduced in \cite{bierman} and defined as $\mathsf{CK}T4$, which is $\mathsf{CK}+\{T_a, T_b, 4_a, 4_b\}$.

Although $\mathsf{CK}$ is natural and interesting in some settings, including the type theoretical approach to modality, the system $\mathsf{IK}$ is the one that is claimed to be the \emph{true} intuitionistic analogue of the classical modal logic $\mathsf{K}$. In \cite{Fischer1,Fischer2}, Fischer-Servi provided two types of evidence to support this claim. First, she mapped $\mathsf{IK}$ to an extension of the fusion of $\mathsf{K}$ and $\mathsf{S4}$, using a natural generalization of G\"{o}del's translation. Second, using the standard translation of modal formulas to first order formulas, in the same manner that $\mathsf{K}$ is mapped into the classical first order logic, $\mathsf{IK}$ is mapped into the intuitionistic first order logic. Several extensions of $\mathsf{IK}$ have been introduced and studied (e.g., see \cite{amati}), where the intuitionistic version of several modal logics such as $\mathsf{KD, KT, KB, S4,}$ and $\mathsf{S5}$ are investigated. To be more precise, for any $X \subseteq \{T, B, 4, 5\}$, define $\mathsf{IK}X$ as the smallest logic over $\mathcal{L}$ extending $\mathsf{IK}$ and all the axioms in $X$ in both $(a)$ and $(b)$ versions.
An important logic here is the logic $\mathsf{MIPC}$ defined as
    $\mathsf{MIPC}:= \mathsf{IK}T45.$
    
It is possible to introduce some intuitionistic modal logics over the fragments of $\mathcal{L}$. Over the language $\mathcal{L}_{\Box}$, the basic system we are interested in is
$\mathsf{CK}_{\Box}:= \mathsf{IPC} + \{K_a\}$.
In \cite{Wolter}, $\mathsf{CK}_{\Box}$ is introduced by the name $\mathsf{IntK_{\Box}}$ and $\Diamond \phi$ is defined as $\neg \Box \neg \phi$. Consequently, $\Diamond$ is not distributed over disjunction. This relation between $\Box$ and $\Diamond$ is too binding from the intuitionistic perspective.
Again, we can extend $\mathsf{CK}_{\Box}$ by the $\Diamond$-free axioms in Table  \ref{tableAxiom}. More precisely, let $X \subseteq \{T, 4\}$. By $\mathsf{CK}_{\Box}X$, we mean the smallest logic over $\mathcal{L}_{\Box}$ extending $\mathsf{CK}_{\Box}$ and all the axioms in $X$ in their $(a)$ version.
These logics, specially with the $\mathsf{S4}$ and $\mathsf{S5}$ flavours, are studied in \cite{Ono} by model theoretical methods.

Finally, the most basic logic in this paper over $\mathcal{L}_{\Diamond}$ is the \emph{basic lax logic}, $\mathsf{BLL}:=\mathsf{IPC}+\{M^{\to}_{\Diamond}\}$, (see Table \ref{tableAxiom}). If we add the axioms $(T_b)$ and $(4_b)$ to the logic $\mathsf{BLL}$, we get the \emph{propositional lax logic}, $\mathsf{PLL}:=\mathsf{IPC} + \{M^{\to}_{\Diamond}, T_b, 4_b\}$, see \cite{Fairtlough}. For a nice survey on intuitionistic modal logics, see \cite{Simpson}.

So far,\! we have seen two approaches to design modal logics over the intuitionistic base. We can also lift these approaches to the intermediate base.
\begin{dfn}\label{DefIntermediateModalLogics}
Let $L$ be an intermediate logic, i.e., a logic over $\mathcal{L}_p$ such that $\mathsf{IPC} \subseteq L \subseteq \mathsf{CPC}$. Then, by $L\mathsf{CK}$ (resp. $L\mathsf{IK}$), we mean the smallest logic over $\mathcal{L}$, extending $L \cup \mathsf{CK}$ (resp. $L \cup \mathsf{IK}$). Similarly, we define $L\mathsf{CK}_{\Box}$ as the smallest logic over $\mathcal{L}_{\Box}$, extending $L \cup \mathsf{CK}_{\Box}$ and we set $L\mathsf{BLL}$ as the smallest logic over $\mathcal{L}_{\Diamond}$, extending $L \cup \mathsf{BLL}$.
\end{dfn}

Moving to Kripke frames, there are also different proposals for their intuitionistic counterparts. 
Here, we define the most general one that captures our base logic $\mathsf{CK}$. However, as we will only use them when they have one node, the choice of the model is immaterial in this paper. 

\begin{dfn}\label{def: Kripke Frames and Models}
\cite{Kojima} A \emph{constructive modal Kripke frame} is a quadruple $\calF=(W, \leq , R, F)$, where $W$ is a non-empty set, $\leq$ is a preorder (a reflexive and transitive binary relation) on $W$, $R$ is a binary relation on $W$ and $F \subseteq W$, a set of fallible worlds, such that:

   
 
\begin{itemize}
    \item
  $F$ is closed under $\leq$: if $v \in F$ and $v \leq w$ then $w \in F$, 
    \item
    $F$ is closed under $R$: if $(v, w) \in R$ and $v \in F$ then $w \in F$, 
    \item
    $R$ is serial on $F$: if $v \in F$ then there exists $w \in F$ such that $(v, w) \in R$. 
\end{itemize}
A \emph{constructive modal Kripke model} based on the frame $\mathcal{F}=(W, \leq, R, F)$ is a tuple $\calM=(W, \leq, R, F, V)$, where $V$ is a valuation function mapping each world in $W$ to a set of atomic formulas in $\mathcal{L}$ such that for every $v,w \in W$ if $v \leq w$ then $V(v) \subseteq V(w)$, and if $w \in F$ then $V(w)$ has all the atomic formulas. Define a formula $\phi$ to be \emph{true at the world $w$ in $\calM$}, denoted by $\calM, w \vDash \phi$ (or $w \vDash \phi$, for short), by recursion on the construction of $\phi$:

\vspace{5pt}
\begin{tabular}{lll}
$\calM, w \vDash \top$ & & \\
$\calM, w \vDash \bot$ & if{f} & $w \in F$;\\
$\calM, w \vDash p$ & if{f} & $p \in V(w)$, for an atomic formula $p$;\\
$\calM, w \vDash \varphi \wedge \psi$ & if{f} & $\calM, w \vDash \varphi$ and $\calM, w \vDash \psi$;\\
$\calM, w \vDash \varphi \vee \psi$ & if{f} & $\calM, w \vDash \varphi$ or $\calM, w \vDash \psi$;\\
$\calM, w \vDash \varphi \to \psi$ & if{f} & $\forall v \geq w$, if $\calM, v \vDash \varphi$ then $\calM, v \vDash \psi$;\\
$\calM, w \vDash \Box \varphi$ & if{f} &
$\forall u \geq w, \forall v \in W$ \big(if $(u, v) \in R$, then $\calM, v \vDash \phi\big)$;\\
$\calM, w \vDash \Diamond \varphi$ & if{f} & 
$\forall u \geq w$, $\exists v \in W$ $\big((u, v) \in R$ and $\calM, v \vDash \phi$\big).
\end{tabular}
\linebreak

\noindent A formula $\phi$ is \emph{valid in $\calM$}, denoted by $\calM \vDash \phi$, when $\calM, w \vDash \phi$ for all $w \in W$, and it is \emph{valid in~$\calF$} when it is valid in all models based on $\calF$. A logic $L$ is \emph{valid in} $\calF$ if $\mathcal{F} \vDash \phi$, for any $\phi \in L$. 
By the \emph{reflexive node frame}, we mean the frame $\mathcal{K}_r=(\{w\}, =, \{(w, w)\}, \varnothing)$ and by the \emph{irreflexive node frame}, we mean $\mathcal{K}_i=(\{w\}, =, \varnothing, \varnothing)$. 
\end{dfn}



\subsection{Sequent Calculi}

A \emph{sequent} $S$ over the language $\mathfrak{L}$ is an expression of the form $\Sigma \Rightarrow \Lambda$, where $\Sigma$ and $\Lambda$ are finite multisets of $\mathfrak{L}$-formulas. By the formula interpretation of the sequent $S=\Sigma \Rightarrow \Lambda$, we mean $I(S)=\bigwedge \Sigma \to \bigvee \Lambda$. The multiset $\Sigma$ is called the \emph{antecedent} and $\Lambda$ the \emph{succedent} of the sequent $S$. If the succedent of a sequent has at most one formula, the sequent is called \emph{single-conclusion}. A \emph{meta-sequent} $\mathcal{S}$ is defined in the same way as sequents, except that here we are also allowed to use multiset variables both in the antecedents and in the 
succedents and we can also use the boxes of the multiset variables in the antecedents. More precisely, a meta-sequent is in the form 
\begin{center}
    $\{\Box\Gamma_i\}_{i \in I}, \{\Pi_j\}_{j \in J}, \overline{A} \Rightarrow \Delta \quad$ or $\quad \{\Box\Gamma_i\}_{i \in I}, \{\Pi_j\}_{j \in J}, \overline{A} \Rightarrow \overline{B}$
\end{center}
where $\Gamma_i, \Pi_j,$ and $\Delta$ are multiset variables and $I$ is a (possibly empty) set of indices, and $\overline{A}$ and $\overline{B}$ are (possibly empty) multisets of $\mathfrak{L}$-formulas.
Notice that the meta-sequents, as defined, are not in their most general form. For instance, it is possible to allow more complex expressions such as $\Diamond \Box \Gamma$ in the meta-sequents, where $\Gamma$ is a multiset variable. However, we use this restricted form as it is the only form we need in the present paper. 

A \emph{rule} is an expression of the form 
\small{\begin{tabular}{c}
 \AxiomC{$\mathcal{S}_1, \;\; \ldots, \;\; \mathcal{S}_n$}
 \UnaryInfC{$\mathcal{S}$}
 \DisplayProof
\end{tabular}}
\normalsize where $\mathcal{S}$, called the \emph{conclusion}, and $\mathcal{S}_i$'s, called the \emph{premises}, are meta-sequents. 
A rule with no premise is called an \emph{axiom}. By an \emph{instance} of a rule or an axiom over $\mathfrak{L}$, we mean the result of substituting the multiset variables by multisets of $\mathfrak{L}$-formulas and substituting the atomic formulas by  $\mathfrak{L}$-formulas.
A \emph{sequent calculus} (\emph{calculus}, for short) over $\mathfrak{L}$ is defined as a finite set of rules over $\mathfrak{L}$.
A sequent calculus $H$ \emph{extends} the sequent calculus $G$ when they are over the same language and $G \subseteq H$. For a set of rules $\mathcal{R}$, by $G+ \mathcal{R}$, we mean the calculus obtained by adding every rule in $\mathcal{R}$ to $G$. For a set of formulas $\mathcal{A}$, by $G+ \mathcal{A}$, we mean the calculus obtained by adding the meta-sequent $(\, \Rightarrow A)$, for every $A \in \mathcal{A}$ as an axiom to $G$. A \emph{proof} $\pi$ in the calculus $G$ for a sequent $T$ from a finite set of sequents $\{T_i\}_{i \in I}$ is a sequence of sequents $\{S_j\}_{j=1}^m$ such that $S_m=T$ and each sequent $S_j$ is either one of the sequents $T_i$, or an instance of an axiom in $G$, or derived from an instance of a rule in $G$ from some $S_{j_1}, \dots, S_{j_k}$, where $j_1, \dots, j_k <j$. When $\pi$ is a proof of $T$ from $\{T_i\}_{i \in I}$ in $G$, we write $\{T_i\}_{i \in I} \vdash^{\pi}_G T$, and we call it a \emph{$G$-proof} of $T$ from the set of assumptions $\{T_i\}_{i \in I}$. If $\pi$ is a proof of the sequent $S$ in $G$ from an empty set of assumptions, we write $G \vdash^{\pi} S$ and we call $\pi$ a \emph{$G$-proof} of $S$. We say the formulas $A$ and $B$ are $G$-\emph{equivalent} or \emph{equivalent in $G$}, denoted by $G \vdash A \Leftrightarrow B$, when $G \vdash A \Rightarrow B$ and $G \vdash B \Rightarrow A$. 
A rule $R$ is called \emph{admissible} in $G$, if for any instance \small{\begin{tabular}{c}
 \AxiomC{$S_1, \;\; \ldots, \;\; S_n$}
 \UnaryInfC{$S$}
 \DisplayProof
\end{tabular}}
\normalsize of $R$, the provability of all the premises in $G$ implies the provability of the conclusion in $G$. Similarly, $R$ is \emph{admissible} in a logic $L$, when for any instance \small{\begin{tabular}{c}
 \AxiomC{$S_1, \;\; \ldots, \;\; S_n$}
 \UnaryInfC{$S$}
 \DisplayProof
\end{tabular}}
\normalsize of $R$, if $I(S_1), \dots , I(S_n) \in L$ then $I(S)\in L$. The rule $R$ is called \emph{provable} in $G$, if for any instance \small{\begin{tabular}{c}
 \AxiomC{$S_1, \;\; \ldots, \;\; S_n$}
 \UnaryInfC{$S$}
 \DisplayProof
\end{tabular}}
\normalsize we have $S_1, \ldots, S_n \vdash_G S$. We say $L$ is the \emph{logic of the sequent calculus} $G$ or equivalently $G$ is the \emph{sequent calculus of the logic} $L$, when $\Sigma \Rightarrow \Lambda$ is provable in $G$ if and only if $\bigwedge \Sigma \to \bigvee \Lambda \in L$, for any sequent $\Sigma \Rightarrow \Lambda$. 
Usually, we prefer to use the same name for a logic and its canonical sequent calculus. However, we use different fonts to distinguish them. We use the boldface letters, as in $\mathbf{CK}$, for a sequent calculus while the sans-serif font, as in $\mathsf{CK}$, is reserved for the logic. Finally, if $\calM$ is a constructive modal Kripke model and $S=(\Sigma \Rightarrow \Lambda)$ a sequent, then \emph{$S$ is valid in $\calM$} if $I(S)=\bigwedge \Sigma \to \bigvee \Lambda$ is valid in $\calM$. A similar definition for frames is in place. Moreover,
a sequent calculus $G$ is \emph{valid in the frame} $\calF$ (model $\mathcal{M}$) if every $G$-provable sequent is valid in $\calF$ (model $\mathcal{M}$).

For $\mathsf{IPC}$ we use the single-conclusion Gentzen-style sequent calculus, $\mathbf{LJ}$\footnote{The calculus mentioned here is almost identical to Gentzen's original calculus $\mathbf{LJ}$, where $\Gamma$ and $\Delta$ are sequences of formulas, as opposed to finite multisets, and the axioms are slightly different. However, it is easy to see that these two systems are equivalent.}, presented in Table \ref{tableLJ}.
\begin{table}[!ht]
  \centering
\resizebox{\columnwidth}{!}{%
 \begin{tabular}{ p{3cm} p{3cm} p{3cm} p{3cm} }
   \multicolumn{2}{c}{ \footnotesize  $\Gamma, p \Rightarrow p \;\;$ \scriptsize$(id)$}  & \hspace{-3em} \footnotesize $\Gamma, \bot \Rightarrow \Delta \;\;$ \scriptsize$(L \bot)$ & \hspace{-3em} \footnotesize $\Gamma \Rightarrow \top \;\;$ \scriptsize$(R \top)$\\
   \midrule    
   
 \footnotesize  \AxiomC{$\Gamma \Rightarrow \Delta$}
 \RightLabel{\scriptsize{$(L w)$} }
\UnaryInfC{$\Gamma, p \Rightarrow \Delta$}
 \DisplayProof
& \footnotesize  \AxiomC{$\Gamma \Rightarrow $}
 \RightLabel{\scriptsize{$(R w)$} }
\UnaryInfC{$\Gamma \Rightarrow p$}
 \hspace*{4em}\DisplayProof
 & 
\multicolumn{2}{c}{ \footnotesize  
 \AxiomC{$\Gamma, p, p \Rightarrow \Delta$}
 \RightLabel{\scriptsize{$(L c)$} }
 \UnaryInfC{$\Gamma, p \Rightarrow \Delta$}
  \hspace*{2em}\DisplayProof}\\
 \midrule 
   
\multicolumn{4}{c}{
 \footnotesize 
 \AxiomC{$\Gamma \Rightarrow p$}
 \AxiomC{$\Gamma, p \Rightarrow \Delta$}
 \RightLabel{\scriptsize{$(cut)$} }
 \BinaryInfC{$\Gamma \Rightarrow \Delta$}
 \DisplayProof}\\
 \midrule
     
  \footnotesize  
 \AxiomC{$\Gamma, p \Rightarrow \Delta$}
 \RightLabel{\scriptsize{$(L \wedge_1)$} }
 \UnaryInfC{$\Gamma, p \wedge q \Rightarrow \Delta$}
 \DisplayProof
& \footnotesize 
\AxiomC{$\Gamma, q \Rightarrow \Delta$}
 \RightLabel{\scriptsize{$(L \wedge_2)$} }
 \UnaryInfC{$\Gamma, p \wedge q \Rightarrow \Delta$}
 \hspace*{2em}\DisplayProof
 & 
\multicolumn{2}{c}{ \footnotesize   \AxiomC{$\Gamma \Rightarrow p$}
 \AxiomC{$\Gamma \Rightarrow q$}
 \RightLabel{\scriptsize{$(R \wedge)$} }
 \BinaryInfC{$\Gamma \Rightarrow p \wedge q$}
  \hspace*{2em}\DisplayProof}\\
 \midrule
 
 \footnotesize  
  \AxiomC{$\Gamma, p \Rightarrow \Delta$}
 \AxiomC{$\Gamma, q \Rightarrow \Delta$}
 \RightLabel{\scriptsize{$(L \vee)$} }
 \BinaryInfC{$\Gamma, p \vee q \Rightarrow \Delta$}
 \DisplayProof
& 
\multicolumn{2}{c}{\footnotesize 
 \AxiomC{$\Gamma \Rightarrow p$}
 \RightLabel{\scriptsize{$(R \vee_1)$} }
 \UnaryInfC{$\Gamma \Rightarrow p \vee q$}
 \hspace*{2em}\DisplayProof}
 & 
 \footnotesize   \AxiomC{$\Gamma \Rightarrow q$}
 \RightLabel{\scriptsize{$(R \vee_2)$} }
 \UnaryInfC{$\Gamma \Rightarrow p \vee q$}
 \hspace*{-3em}\DisplayProof\\
 \midrule

    \multicolumn{2}{c}{
\footnotesize   \AxiomC{$\Gamma \Rightarrow p$}
 \AxiomC{$\Gamma, q \Rightarrow \Delta$}
 \RightLabel{\scriptsize{$(L \to)$} }
 \BinaryInfC{$\Gamma,  p \to q \Rightarrow \Delta$}
 \hspace*{2em}\DisplayProof}
&
\multicolumn{2}{c}{
\footnotesize  \AxiomC{$\Gamma, p \Rightarrow q$}
 \RightLabel{\scriptsize{$(R \to)$} }
 \UnaryInfC{$\Gamma \Rightarrow p \to q$}
 \hspace*{-1em}\DisplayProof}
 \end{tabular} }
\caption{Single-conclusion sequent calculus $\mathbf{LJ}$ \label{tableLJ}}
\end{table} 
Note that being single-conclusion means that in each rule, $\Delta$ has at most one formula. In Table \ref{tableAxiom}, several modal axioms are introduced. Adding them to $\LJ$ will result in various sequent calculi for intuitionistic modal logics. To present sequent calculi for the logics introduced in Subsection \ref{SubsectionPreliminariesIntuitionisticModal}, consider the following modal rules:
\begin{center}
\begin{tabular}{c c c}
 \AxiomC{$\Gamma \Rightarrow p$}
 \RightLabel{\small$(K_{\Box})$}
 \UnaryInfC{$\Box \Gamma \Rightarrow \Box p$}
 \DisplayProof
& 
\AxiomC{$\Gamma, p \Rightarrow q$}
 \RightLabel{\small$(K_{\Diamond})$}
 \UnaryInfC{$\Box \Gamma, \Diamond p \Rightarrow \Diamond q$}
 \DisplayProof 
 & 
\AxiomC{$\Gamma, p \Rightarrow q$}
 \RightLabel{\small$(\Diamond L)$}
\UnaryInfC{$ \Gamma, \Diamond p \Rightarrow \Diamond q$}
 \DisplayProof
 \end{tabular}
\end{center}
Over $\mathcal{L}$, the most basic intuitionistic modal calculus that we are interested in is $\mathbf{CK}$ defined as $\LJ + \{K_{\Box}, K_{\Diamond}\}$.
Note that the cut rule is explicitly present in $\mathbf{LJ}$ and hence in $\mathbf{CK}$. For the logic $\mathsf{CK}X$, where $X \subseteq \{T, B, 4, 5\}$, by $\mathbf{CK}X$, we mean the calculus $\mathbf{CK}$ extended by the axioms in $X$ in both $(a)$ and $(b)$ versions.
Over $\mathcal{L}_{\Box}$, we denote the calculus $\LJ+\{K_\Box\}$ by $\mathbf{CK}_\Box$. For the logic $\mathsf{CK}_{\Box}X$, where $X \subseteq \{T, 4\}$, by $\mathbf{CK}_{\Box}X$, we mean the calculus $\mathbf{CK}_{\Box}$ extended by the axioms in $X$ in their $(a)$ version. Over $\mathcal{L}_{\Diamond}$, the interesting systems are the calculus $\mathbf{BLL}$ defined as $\mathbf{LJ}+\{\Diamond L\}$. If we add the axioms $(T_b)$ and $(4_b)$ to the calculus $\mathbf{BLL}$, we reach the calculus $\mathbf{PLL}= \LJ + \{\Diamond L, T_b, 4_b\}$. It is easy to see that the systems introduced here are the sequent calculi for their corresponding logics, introduced in Subsection \ref{SubsectionPreliminariesIntuitionisticModal}.

\begin{rem}
Unlike the usual situation where sequent calculi are assumed to be ``well-behaved", here we are quite liberal to accept the problematic rules such as cut, the unfamiliar rules such as
\begin{tabular}{c}
 \AxiomC{$\Gamma \Rightarrow p \wedge q$}
 \UnaryInfC{$\Gamma \Rightarrow p$}
 \DisplayProof
\end{tabular}
and initial sequents like $\Gamma \Rightarrow \Box p \to p$, in our systems. Therefore, one may read our sequent calculi as our preferred way to represent the general proof systems and nothing more. For instance, it is easy to represent the natural deduction systems using the sequent calculi we allow. In this sense, our results about sequent calculi can be safely applied on the natural deduction systems, as well.
\end{rem}

\begin{dfn}\label{Def: Visser Rules}
Let $L$ be a logic over the language $\mathfrak{L} \in \{\mathcal{L}, \mathcal{L}_\Box, \mathcal{L}_\Diamond, \mathcal{L}_p\}$. We say that $L$ has the \emph{disjunction property (DP)}, if $A \vee B \in L$ implies either $A \in L$ or $B \in L$, for any $\mathfrak{L}$-formulas $A$ and $B$.
Define \emph{Visser's rules} as
\begin{center}
\begin{tabular}{c}
$V_n$ \quad
 \AxiomC{$\Rightarrow \left(\bigwedge_{i=1}^n\left(p_i \rightarrow q_i\right) \rightarrow p_{n+1} \vee p_{n+2}\right) \vee r$}
 \UnaryInfC{$\Rightarrow \bigvee_{j=1}^{n+2}\left(\bigwedge_{i=1}^n\left(p_i \rightarrow q_i\right) \rightarrow p_j\right) \vee r$}
 \DisplayProof
\end{tabular}
\end{center}
for $n \geq 1$. By abuse of notation, if $L$ has DP, we say $L$ admits $V_0$ to consider DP as one of Visser's rules. We say the logic $L$ \emph{admits all Visser's rules} when it admits all $V_n$ for $n\geq 0$.
\end{dfn}

\begin{rem}
It is easy to see that a logic $L$ admits all Visser's rules if and only if it has the property that if $\bigwedge_{i \in I}(A_i \to B_i) \to (C \vee D) \in L$, for some (possibly empty) finite index set $I$, then one of the following formulas:\small \begin{center}\begin{tabular}{c c c}$\bigwedge_{i \in I}(A_i \to B_i) \to C\;$ or & $\bigwedge_{i \in I}(A_i \to B_i) \to D\;$ or& $\bigwedge_{i \in I}(A_i \to B_i) \to A_i\;$,\\\end{tabular}\end{center}\normalsize for some $i \in I$, is in $L$.
\end{rem}

\begin{lem}\label{LemAdStrong}
Let $G$ be a calculus over $\mathcal{L}$ for a logic $L$ over $\mathcal{L}$ such that $L \supseteq \mathsf{CK}$. Then, $G+\mathbf{CK}$ is also a calculus for $L$. The same holds if we replace the triple $(\mathcal{L}, \mathsf{CK}, \mathbf{CK})$ by $(\mathcal{L}_{\Box},\mathsf{CK}_{\Box}, \mathbf{CK}_{\Box})$, $(\mathcal{L}_{\Diamond}, \mathsf{BLL}, \mathbf{BLL})$, or $(\mathcal{L}_p, \mathsf{IPC}, \mathbf{LJ})$.
\end{lem}
\begin{proof}
We only prove the case for the language $\mathcal{L}$. The rest are similar. For the case $\mathcal{L}$, we know that adding the admissible rules of $L$ to $G$ does not change the theorems of $L$. Therefore, it is enough to show that all the axioms and rules of $\mathbf{CK}$ are admissible in $L$. We only investigate the crucial rules of cut and $(K_\Box)$, as the other are similar.
First, as $L \supseteq \mathsf{CK}$, we have:

\begin{enumerate}
    \item\label{va} $(C \to D) \to (C \to C \wedge D) \in L$,
    \item\label{na}
    $(E \to F) \to [(F \to I) \to (E \to I)] \in L$.
    \item\label{nna}
    $\bigwedge \Box \Gamma \to \Box \bigwedge \Gamma \in L$
\end{enumerate}
for any formulas $C, D, E, F$, $I$ and multiset $\Gamma$. The reason is that all these formulas are in $\mathsf{CK}$ and hence in $L$.
Now, for the admissibility of the cut rule, suppose $\bigwedge \Gamma \to A \in L$ and $(\bigwedge \Gamma \wedge A) \to B \in L$. By \ref{va} and the former formula, as $L$ is closed under modus ponens, we get $\bigwedge \Gamma \to (\bigwedge \Gamma \wedge A) \in L$. By \ref{na}, we have
\[
(\bigwedge \Gamma \to (\bigwedge \Gamma \wedge A)) \to [(\bigwedge \Gamma \wedge A) \to B) \to (\bigwedge \Gamma \to B)] \in L.
\]
Therefore, as $L$ is closed under modus ponens, we have $\bigwedge \Gamma \to B \in L$. \\
For the admissibility of the rule $(K_\Box)$, suppose $\bigwedge \Gamma \to A \in L$. By definition of a logic, $L$ is closed under necessitation. Hence, $\Box (\bigwedge \Gamma \to A) \in L$. As $L \supseteq \mathsf{CK}$, we have $\Box (\bigwedge \Gamma \to A) \to (\Box \bigwedge \Gamma \to \Box A) \in L$. Since $L$ is closed under modus ponens, we get $\Box \bigwedge \Gamma \to \Box A \in L$. By \ref{nna}, $\bigwedge \Box \Gamma \to \Box \bigwedge \Gamma \in L$. Therefore, by \ref{na} and modus ponens, we get $\bigwedge \Box \Gamma \to \Box A \in L$.
\end{proof}

Finally, we set some basic conventions on the complexity-theoretic part of the paper. The \emph{size} of a formula $A$, a multiset $\Gamma$, a sequent $S$ or a proof $\pi$ is defined as the number of symbols in it, and denoted by $|A|$, $|\Gamma|$, $|S|$, and $|\pi|$, respectively.
For a multiset $\Gamma$, by $\parallel \Gamma \parallel$, we mean the cardinality of $\Gamma$, i.e., the number of the elements in $\Gamma$, counting their multiplicity. An algorithm is called \emph{feasible}, if it runs in polynomial time in the size of the input. We use feasible and polynomial time computable, interchangeably. A proof is called \emph{tree-like}, if every sequent in the proof is used at most once as a hypothesis of a rule in the proof. From the complexity theoretic point of view, usually it makes a difference if we use tree-like proofs or \emph{general (dag-like)} proofs. However, in the presence of the cut rule, conjunction and implication in the language and their intuitionistic rules, it is possible to feasibly simulate dag-like proofs by tree-like ones, see \cite{Jan}. Therefore, w.l.o.g., throughout this paper, whenever we write down a proof, we always mean a tree-like proof. The reason for this preference is the inequality 
\begin{center}
  $\sum_{i \in I} |\pi_i|+|S| \leq |\pi|,$  
\end{center}
where $\pi$ is a tree-like proof with the conclusion $S$ and the immediate sub-proofs $\{\pi_i\}_{i \in I}$, i.e., the proofs of the premises. This inequality helps to bound the time of different constructions in which we will use $\pi$ as an input.


\begin{dfn}\label{dfnAdmissiblyStrong}
Let $\mathfrak{L} \in \{\mathcal{L}, \mathcal{L}_\Box, \mathcal{L}_\Diamond, \mathcal{L}_p\}$ be a language and $G$ be a sequent calculus over $\mathfrak{L}$. A rule $R$ over $\mathfrak{L}$ is called \emph{feasibly provable} in $G$, if there is a polynomial time algorithm $f_R$ such that for any $\mathfrak{L}$-instance of $R$ with the premises $\{S_i\}_{i \in I}$ and the consequence $S$, $f(\{S_i\}_{i \in I}, S)$ provides a $G$-proof of $S$ with the set of assumptions $\{S_i\}_{i \in I}$, i.e.,  $\{S_i\}_{i \in I} \vdash_G^{f(\{S_i\}_{i \in I}, S)} S$. A sequent calculus $G$ is called \emph{strong} over the language
\begin{itemize}
\item
$\mathcal{L}$, if every axiom and rule of $\mathbf{CK}$ is feasibly provable in $G$;
\item
$\mathcal{L}_\Box$, if every axiom and rule of $\mathbf{CK}_\Box$ is feasibly provable in $G$;
\item
$\mathcal{L}_\Diamond$, if every axiom and rule of $\mathbf{BLL}$ is feasibly provable in $G$;
\item
$\mathcal{L}_p$, if every axiom and rule of $\mathbf{LJ}$ is feasibly provable in $G$.
\end{itemize}
\end{dfn}

\begin{rem}\label{RemarkRule}
Clearly if a rule is in $G$, it is feasibly provable in $G$. Hence, if $G$ is defined over $\mathcal{L}$ and has all the axioms and rules of $\mathbf{CK}$ as its axioms and rules, then $G$ is trivially strong over $\mathcal{L}$. A corresponding claim also holds for the fragments of $\mathcal{L}$. 
\end{rem}


\begin{dfn}
Let $G$ and $H$ be two sequent calculi over the same language. A \emph{polynomial deduction simulation} (\emph{pd-simulation}) of $G$ in $H$ is
a polynomial time function $f$ such that for any proof $\pi$ of $S$ from $\{S_i\}_{i \in I}$ in $G$, $f(\pi)$ is a proof of $S$ from $\{S_i\}_{i \in I}$ in $H$. We say $H$ \emph{pd-simulates} $G$, denoted by $G \leq_{pd} H$ when there is a pd-simulation of $G$ in $H$. The calculi $G$ and $H$ are \emph{pd-equivalent}, when $G \leq_{pd} H$ and $H \leq_{pd} G$. Sequent calculi $G$ and $H$ are \emph{d-equivalent} if $\{S_i\}_{i \in I} \vdash_G S$ if and only if $\{S_i\}_{i \in I} \vdash_H S$, for any family of sequents $\{S_i\}_{i \in I} \cup \{S\}$.
\end{dfn}

\begin{rem}
The well-known notion of p-simulation of one proof system by another is a feasible machinery to simulate proofs without any assumptions. Our pd-simulation strengthens this notion to also cover \emph{deductions}, i.e., proofs where assumptions can also be present. For the specific systems we are interested in, these two notions are equivalent. However, to avoid any confusion, throughout this paper, we only use the pd-simulations.
\end{rem}


\begin{lem}\label{LocalToGlobal}
Let $G$ and $H$ be two sequent calculi over $\mathfrak{L} \in \{\mathcal{L}, \mathcal{L}_\Box, \mathcal{L}_\Diamond, \mathcal{L}_p\}$.
\begin{enumerate}
    \item \label{Lemmamorede1}
    For any rule $R$ over $\mathfrak{L}$, if $G \leq_{pd} H$ and $R$ is  feasibly provable in $G$, then $R$ is feasibly provable in $H$.
    \item \label{Lemmamorede2}
    If each rule of $G$ is feasibly provable in $H$, then $G \leq_{pd} H$.
\end{enumerate}
\end{lem}
\begin{proof}
Part \ref{Lemmamorede1} is easy. For \ref{Lemmamorede2}, by the assumption, for each rule $R$ of $G$, there is a feasible function $f_R$ such that for any instance $(\{S_k\}_{k \in K}, S)$ of $R$, $f_R(\{S_k\}_{k \in K}, S)$ is an $H$-proof of $S$ from the assumptions $\{S_k\}_{k \in K}$. Therefore, $f_R$ is computable in time $O((\sum_{k \in K}|S_k|+|S|)^{c_R})$, for some constant $c_R$. As $G$ has finitely many rules, we set $c$ as the maximum of all $c_R$'s. Hence, each  $f_R$ is computable in time $O((\sum_{k \in K}|S_k|+|S|)^{c})$. Now, to simulate any $G$-proof $\pi$ in $H$, it is enough to replace the instance $(\{S_k\}_{k \in K}, S)$ of any rule $R$ used in $\pi$ by $f_R(\{S_k\}_{k \in K}, S)$. Call the result of this replacement $g(\pi)$ and notice that $g(\pi)$ is a proof in $H$. To compute the time that $g$ requires, note that any sequent in $\pi$ is bounded in size by $|\pi|$. Therefore, for each $f_R(\{S_k\}_{k \in K}, S)$ in $g(\pi)$, we need $O(|\pi|^{c})$ steps for the computation and as the number of the rules in $\pi$ is less than $|\pi|$, the whole time that $g(\pi)$ requires is $O(|\pi|^{c+1})$.
\end{proof}

\begin{rem}
Let $\mathcal{A}$ be a finite set of axioms and $G$ be a sequent calculus in which the weakening rules are feasibly provable. Recall that we defined $G+\mathcal{A}$ as a sequent calculus obtained by adding the meta-sequent $(\,\Rightarrow A)$ as an axiom, for any $A \in \mathcal{A}$. It is also possible to define $G+\mathcal{A}$ as the calculus obtained by adding the meta-sequent $\Gamma \Rightarrow A$, for any $A \in \mathcal{A}$, where $\Gamma$ is a multiset variable. Call this new definition $G+^{'} \mathcal{A}$. We claim that $G+^{'} \mathcal{A}$ and $G+^{'} \mathcal{A}$ are pd-equivalent. It is clear that $G+^{'} \mathcal{A}$ pd-simulates $G+ \mathcal{A}$. For the converse, as the weakening rule is feasibly provable in $G$, it is easy to see that the axiom $\Gamma \Rightarrow A$ is feasibly provable in $G$ and hence by Lemma \ref{LocalToGlobal}, $G+ \mathcal{A}$ pd-simulates $G+^{'} \mathcal{A}$. Using this observation, from now on, as we only work with the calculi that feasibly prove the weakening rules, we will use these two definitions interchangeably.
\end{rem}

\section{Constructive Formulas and Rules} \label{SectionAlmostnegative}
One of the main goals of the present paper is to identify a general form for \textit{constructive rules} in the intuitionistic realm of modal logics.  Informally speaking, by a constructive rule, we mean a rule that its addition to the base system $\mathbf{CK}$ respects the constructive character of the intuitionistic ground. In this section, we first provide a formalization for these constructive rules. Then, by presenting many examples, we see how general these rules are and finally we provide a justification for our choice.

Our strategy to find a natural candidate for constructive rules consists of two parts. First, by employing the least possible syntactical limit on formulas, we introduce a class of \textit{constructive formulas} as the simplest case of the general form of the constructive rules. 
By a constructive formula, we mean a formula $A$ such that if we extend $\mathbf{CK}$ with the axiom $(\, \Rightarrow A)$, the calculus remains constructively acceptable. Second, we exhaustively investigate a
general form for rules and select the ones that are equivalent to a constructive formula. To show that this syntactical limit provides a \emph{true} constructive system, in Section \ref{SectionMain} we show that any calculus extending $\mathbf{CK}$ by some constructive rules and satisfying a mild technical condition admits the feasible version of Visser's rules and specially has the feasible disjunction property. This provides a substantial evidence for the claim that the introduced rules are constructive, as well as a uniform machinery to prove the feasible version of Visser's rules for a huge family of intuitionistic modal logics. 

\subsection{Constructive Formulas}\label{SubsectionConstructiveFormulas}
To identify the aforementioned class of constructive formulas, we employ the well-known heuristic that constructive formulas, whatever they are, must be careful with positive occurrences of the \emph{disjunction-like} connectives, i.e., $\{\vee, \Diamond\}$. For instance, the axioms\footnote{Although the sequent $(\, \Rightarrow A)$ is the axiom, for simplicity, we also call $A$ an axiom.}
$\neg (p \wedge q) \rightarrow \neg p \vee \neg q$ and
$ \neg \Box p \rightarrow \Diamond \neg p$ are clearly non-constructive.  The former proves the axiom of the weak excluded middle, i.e., $\neg r \vee \neg \neg r$, that breaks the disjunction property and hence is not constructively acceptable. For the latter, using the Kripke-style first-order reading of a modal formula, $ \neg \Box p \rightarrow \Diamond \neg p$  is reminiscent of the first-order formula $\neg \forall x P(x) \to \exists x \neg P(x)$, which is not constructively acceptable.

Following this heuristic, one might naively demand that constructive formulas must be defined as formulas that have no positive occurrences of the connectives $\{\vee, \Diamond\}$. Here are two objections to this proposal. First, note that this restriction is too strict and rejects even some constructively accepted formulas with positive occurrences of $\{\vee, \Diamond\}$ that are essential for a constructive reasoning. For instance, consider the axioms
$p \rightarrow  p \vee q$
and $ \Diamond p \vee \Diamond q \rightarrow  \Diamond(p \vee q)$, where the former is even an axiom of the intuitionistic logic. 
Second, even in such a harsh form, the restriction may not be sufficient to ensure the constructive character of the resulting system. The reason is that the disjunction-like connectives are sometimes introduced in an indirect manner via an occurrence of a nested \emph{implication-like} connective, i.e., $\{\Box, \to\}$. For instance, the axiom 
$\neg \neg p \rightarrow p$ indirectly proves $q \vee \neg q$ in a constructively acceptable manner. Hence, to have a constructive axiom, we should also be careful with nested implication-like connectives. Here, one might object that what really is problematic is nested implications and boxes are harmless. However, reading the box as what its first-order interpretation dictates, it is clear that the box has the same character as the universal quantifier and hence behaves similarly as the implication.

To solve these issues, we first extend the heuristic of avoiding the positive occurrences of $\{\vee, \Diamond\}$, by also adding a limit on the depth of the nested implication-like connectives. We limit this depth to two, as the problematic formulas such as $\neg \neg p \rightarrow p$ have the depth three or more.\footnote{By the depth function, we formally mean the function defined by:
$d(p)=d(\bot)=d(\top)=0$, $d(\phi \circ \psi)=max\{d(\phi), d(\psi)\}$, for any $\circ \in \{\wedge, \vee\}$, $d(\Diamond \phi)=d(\phi)$, $d(\phi \to \psi)=max\{d(\phi)+1, d(\psi)\}$ and $d(\Box \phi)=max\{d(\phi), 1\}$. For instance, the depth of $(p \to q) \to r$ and $\Box \Box p \to q$ are two, while the depth of $p \to (q \to r)$ and $p \to \Box q$ are one. However, for the sake of our informal discussion, what we need is an informal notion of depth, counting the depth of the nested implications appearing in the antecedents of the implications.} 
This new limit hopefully solves the second problem we encountered. For the first problem, though, we need to observe that the real problem with the disjunction-like connectives is not their positive occurrences but the mix of such occurrences with the implication-like connectives. For instance, think about the formulas $p \vee \neg p$ and $(p \to q) \vee (q \to p)$, where the disjunction-like connectives are applied on the formulas with an implication-like connective inside. Having this observation, it seems that if we confine the disjunction-like connectives to some blocks that have no occurrence of the implication-like connectives, then substituting the atoms of an acceptable formula with these blocks may keep the formula unproblematic. Therefore, as the final proposal, we define the class of all constructive formulas as the least class of formulas containing the ones without any positive occurrence of $\{\vee, \Diamond\}$ or any nested occurrence of $\{\Box, \to\}$ with depth three or more, closed under the substitution by the formulas that only consist of $\{\wedge, \vee, \Diamond\}$. Note that by excluding $\{\Box, \to\}$ from the substituting formulas, we also avoid any increase in depth which has a crucial role in our investigation. We will see that this weakening covers many natural constructively acceptable axioms with positive occurrences of disjunction-like connectives, while it is still constructively acceptable. To justify the latter claim, we show that if we add any of these formulas to $\mathbf{CK}$ as new axioms, if the result satisfies a mild technical condition, the resulting system admits all Visser's rules and specially enjoys the disjunction property. 

\noindent In the following, we see a more precise definition for constructive formulas.
\begin{dfn}\label{DefPositiveFormulas}
We define the following three sets of $\L$-formulas: 
\begin{itemize}
\item
The set of \emph{basic} formulas is the smallest set containing the atomic formulas and the constants $\top$ and $\bot$ and closed under $\{\wedge, \vee, \Diamond\}$.
\item
The set of \emph{almost positive} formulas is the smallest set containing the basic formulas and closed under $\{\wedge, \vee, \Box, \Diamond\}$ and the implications of the form $A \to B$, where $A$ is basic and $B$ is almost positive.
\item
The set of \emph{constructive} formulas is the smallest set containing the basic formulas and closed under $\{\wedge, \Box\}$ and the implications of the form $A \to B$, where $A$ is almost positive and $B$ is constructive.
\end{itemize}
\noindent A formula in the languages $\mathcal{L}_{\Box}$, $\mathcal{L}_{\Diamond}$ and $\mathcal{L}_p$ is called basic, almost positive or constructive, if it is basic, almost positive or constructive as a formula in the extended language $\mathcal{L}$.
\end{dfn}

\begin{rem}
Note that the basic formulas are exactly the ones used for the substitutions in the opening discussion of Subsection \ref{SubsectionConstructiveFormulas}. Another point is that it is customary to call the formulas constructed from the atoms by the positive operators $\{\wedge, \vee, \Box, \Diamond\}$, the positive formulas. Our almost positive formulas deviate from this usual definition by allowing a very limited form of implication, i.e., the ones with the basic antecedents. Therefore, almost positive formulas provide a definition for the formulas that avoid any nested application of the implication-like connectives. Finally, notice that the constructive formulas are exactly the ones we motivated before. First, by their structure, we see that they do not contain any positive occurrence of the connectives $\{\vee, \Diamond\}$, except when the occurrences are confined in the basic formulas. Second, the antecedent of any implication in a constructive formula is almost positive which implies that the depth of the nested implication-like connectives in a constructive formula is at most two.
\end{rem}

\begin{exam} \label{ExampleOfFormulas}  The following table provides some examples and non-examples for basic, almost positive and constructive formulas

\vspace{-15pt}
\footnotesize \begin{center}
\begin{tabular}{c c c}
 & \textbf{is} & \textbf{is not} \\
 
 \hline

\textbf{basic} & $p \wedge q$, $p \vee q$, $\Diamond^n p$ 
& \hspace{-10pt}$\neg p$, $p \to q$, $\Box p$\\

\hline

\textbf{almost \!positive} & $\neg p$, $(p \vee \neg p)$, $\Diamond^m \Box^n p$, $\Box^m \Diamond^n p$ & \hspace{-10pt}$\neg \neg p$, $(p \to q) \to r$, $\Box p \to q$\\

\hline

\textbf{constructive} &  $\neg \neg p$, $(p \to q) \to r$, $\Box^m \Diamond^n p$  , $\Box (\Diamond p \vee q)$ & \hspace{-10pt}$(p \vee \neg p)$, $\Diamond^{m+1} \Box^{n+1} p$\\

\end{tabular}
\end{center}
\normalsize
where $m,n \geq 0$. It is easy to see that the formulas $p$, $(p \wedge q)$, $(p \vee q)$, $(p \to q)$, $\neg p$, $\Box p$, $\Diamond p$ as well as all the basic formulas are both almost positive and constructive. For more constructive formulas, see Table \ref{tableAxiom}. To have some examples of the formulas that are neither almost positive nor constructive, consider $((p \to q) \to r) \to s$ and $(\Box p \to q) \to r$.
\end{exam}



\subsection{Constructive Rules}
In this subsection, we present our proposal for the promised constructive rules and see some of their examples and non-examples. Then, in Subsection \ref{SubsectionJustification}, to justify our proposal, we present a general form for the sequent-style rules and show that the constructive rules are exactly the ones that are equivalent to a constructive formula over $\mathbf{CK}$. This observation provides a justification for our definition for the constructive rules and its tightly chosen form.
\begin{dfn} \label{DefnegativeRules}
 Let $I$ and $J$ be finite (possibly empty) index sets, $\Gamma$ and $\Delta$ multiset variables, $\overline{P}, \overline{P_i}$ multisets of almost positive formulas, $\overline{C}, \overline{C_j}$ multisets of constructive formulas, and $\overline{B_i}$ a multiset of basic formulas, for any $i \in I$ and $j \in J$. A single-conclusion rule is called:

\begin{itemize}
\item
\emph{left constructive}, when it is of the form
\begin{center}
\begin{tabular}{c c c}
\AxiomC{$ \{ \Gamma, \overline{B_i} \Rightarrow \overline{P_i} \}_{i \in I}$}
\AxiomC{$ \{ \Gamma, \overline{C_j} \Rightarrow \Delta \}_{j \in J}$}
 \BinaryInfC{$\Gamma, \overline{P} \Rightarrow  \Delta $}
 \DisplayProof
\end{tabular}
\end{center}
and if $|J| > 1$, then all the formulas in $\bigcup_{j \in J} (\overline{C_j})$ are basic. Note that by convention, we assume that even if $|J|=0$, the multiset variable $\Delta$ appears in the succedent of the conclusion.




\item 
\emph{right constructive}, when it is of the form
\begin{center}
\begin{tabular}{c}
 \AxiomC{$\{ \Gamma, \overline{B_i} \Rightarrow \overline{P_i} \}_{i\in I}$}
 \UnaryInfC{$\Gamma, \overline{P} \Rightarrow \overline{C}$}
 \DisplayProof
\end{tabular}
\end{center}
\end{itemize}
Note that as both types of the rules are single-conclusion, each of the multisets $\overline{P_i}, \overline{C}$, and $\Delta$ can have at most one formula.

As a special case, an axiom is called \emph{constructive} if it is either of the form $\Gamma, \overline{P} \Rightarrow \Delta$ or of the form $\Gamma, \overline{P} \Rightarrow \overline{C}$. 
\end{dfn}

\begin{rem}
Here is a terminological remark. Note that the form we used for our left constructive rules generalizes the usual form of the left rules in the single-conclusion sequent-style systems. For instance, consider the left rules $(L \to), (L \wedge)$ and $(L \vee)$ in the calculus $\mathbf{LJ}$. However, our form is slightly more general, as in $\overline{P}$ in the antecedent of the conclusion, we also allow the empty set or more than one formula to appear. The form we use in our right constructive rules also generalizes the usual form of the right rules such as $(R \to), (R \wedge)$ and $(R \vee)$. It is also slightly more general, as it allows a multiset $\overline{P}$ in the antecedent of the conclusion and the empty set in its succedent. One might object that as the right rules have formulas both in the antecedent and succedent of the conclusion, it is not reasonable to call them right, anymore. However, in the presence of $\mathbf{CK}$, having formulas in the left hand side in $\overline{P}$ is not a real extension, as one can safely change $\overline{P}$ and $\overline{C}$ to $\bigwedge \overline{P}$ and $\bigvee \overline{C}$, respectively and then move $\bigwedge \overline{P}$ to the right, by $(R \to)$. Note that the formula $\bigwedge \overline{P} \to \bigvee \overline{C}$ is still constructive. It is also possible to reveres the changes. It is enough to use cut with the $\mathbf{CK}$-provable sequent $\bigwedge \overline{P} \to \bigvee \overline{C}, \bigwedge \overline{P} \Rightarrow \bigvee \overline{C}$ and then cut with the $\mathbf{CK}$-provable sequents $\overline{P} \Rightarrow \bigwedge \overline{P}$ and $\bigvee \overline{C} \Rightarrow \overline{C}$. Therefore, one can argue that the right constructive rule as introduced is still essentially a right rule.
\end{rem}

\begin{rem}
To avoid confusion, let us emphasize that if we accept a constructive rule in a system, all of its substitutions are also allowed. However, it does not mean that a substitution of a constructive rule is also constructive. For instance, although the rule $(L \vee)$ is constructive, its substituted version
\begin{center}
\begin{tabular}{c}
 \AxiomC{$\Gamma, p \Rightarrow \Delta $} 
 \AxiomC{$\Gamma, \neg p \Rightarrow \Delta $} 
 \BinaryInfC{$\Gamma, p \vee \neg p \Rightarrow \Delta$}
 \DisplayProof
\end{tabular}
\end{center}
is not constructive, as the rule has two premises while $\neg p$ is not basic.
\end{rem}

\begin{dfn}\label{ConstructiveSystem}
Let $\mathfrak{L} \in \{\mathcal{L}, \mathcal{L}_\Box, \mathcal{L}_\Diamond, \mathcal{L}_p\}$ be a language. A rule $R$ over $\mathfrak{L}$ is called \emph{constructive} if it is constructive as a rule over the extended language $\mathcal{L}$. A sequent calculus over $\mathfrak{L}$ is called \emph{constructive}, when its axioms and rules are all either constructive or
\begin{itemize}
\item
$(K_{\Box})$ or $(K_{\Diamond})$, if $\mathfrak{L}=\mathcal{L}$;
\item
$(K_{\Box})$, if $\mathfrak{L}=\mathcal{L}_{\Box}$;
\item
no other rules, if $\mathfrak{L}=\mathcal{L}_{\Diamond}$ or $\mathfrak{L}=\mathcal{L}_{p}$.
\end{itemize}
\end{dfn}


\begin{table}[!ht]
  \centering
  \vline 
\resizebox{\columnwidth}{!}{%
  \begin{tabular}{ p{4cm} p{4cm} p{4cm} p{3cm} } 
   \toprule \small
   \AxiomC{$\Gamma \Rightarrow \Diamond \bot$}
 \UnaryInfC{$ \Gamma \Rightarrow \bot$}
\DisplayProof
& \small
\AxiomC{$\Gamma \Rightarrow \Diamond (p \vee q)$}
 \UnaryInfC{$ \Gamma \Rightarrow \Diamond p \vee \Diamond q$}
 \DisplayProof
 & 
 \small \AxiomC{$\Gamma, \Diamond p \Rightarrow \Box q$}
 \UnaryInfC{$ \Gamma \Rightarrow \Box(p \to q)$}
  \DisplayProof
  &
  \small  \AxiomC{$\Gamma \Rightarrow \Box p $}
 \UnaryInfC{$ \Gamma \Rightarrow \Diamond p$}
\DisplayProof\\
 \midrule
 
 \small \AxiomC{}
 \UnaryInfC{$\Gamma, \Diamond \bot \Rightarrow \Delta$}
 \DisplayProof 
 &
 \small \AxiomC{}
 \UnaryInfC{$\Gamma, \Box p \Rightarrow \Diamond p$}
 \DisplayProof 
 &
\multicolumn{2}{c}{\small \AxiomC{$\Gamma, \Diamond p \Rightarrow \Delta$}
 \AxiomC{$\Gamma, \Diamond q \Rightarrow \Delta$}
 \BinaryInfC{$\Gamma, \Diamond (p \vee q) \Rightarrow \Delta$}
 \DisplayProof }
 \\
 \midrule
  
 \small
 \AxiomC{$\Gamma \Rightarrow \Box p$}
 \UnaryInfC{$ \Gamma \Rightarrow p$}
 \DisplayProof
& \small
\AxiomC{$\Gamma \Rightarrow p$}
 \UnaryInfC{$ \Gamma \Rightarrow \Diamond p$}
 \DisplayProof
 & \small
 \AxiomC{$\Gamma \Rightarrow  \Box p$}
 \UnaryInfC{$ \Gamma \Rightarrow \Box \Box p$}
 \DisplayProof
& \small
\AxiomC{$\Gamma \Rightarrow \Diamond \Diamond p$}
 \UnaryInfC{$ \Gamma \Rightarrow \Diamond p$}
 \DisplayProof \\
 \midrule
 
 \small \AxiomC{$\Gamma \Rightarrow \Diamond \Box p$}
 \UnaryInfC{$ \Gamma \Rightarrow p$}
 \DisplayProof
& \small
\AxiomC{$\Gamma \Rightarrow p$}
 \UnaryInfC{$ \Gamma \Rightarrow \Box \Diamond p$}
 \DisplayProof 
 &
  \small \AxiomC{$\Gamma \Rightarrow \Diamond \Box p$}
 \UnaryInfC{$ \Gamma \Rightarrow   \Box p$}
 \DisplayProof
 &
 \small \AxiomC{$\Gamma \Rightarrow   \Diamond p$}
 \UnaryInfC{$ \Gamma \Rightarrow \Box \Diamond p$}
 \DisplayProof
  \\
 \midrule

\small \AxiomC{$\Gamma \Rightarrow \Box^{n+1} p$}
 \UnaryInfC{$ \Gamma \Rightarrow \Box^n p$}
 \DisplayProof
& \small \AxiomC{$\Gamma \Rightarrow \Diamond^n p$}
 \UnaryInfC{$ \Gamma \Rightarrow \Diamond^{n+1} p$}
 \DisplayProof
 &
  \small \AxiomC{$\Gamma \Rightarrow \Box^n p$}
 \UnaryInfC{$ \Gamma \Rightarrow \Box^m p$}
 \DisplayProof
 &
 \small \AxiomC{$\Gamma \Rightarrow \Diamond^m p$}
 \UnaryInfC{$ \Gamma \Rightarrow \Diamond^n p$}
 \DisplayProof
  \\
 \midrule
 
\small \AxiomC{$\{ \Gamma \Rightarrow \Box^i p\}_{i=0}^n$}
 \UnaryInfC{$ \Gamma \Rightarrow \Box^{n+1} p$}
 \DisplayProof
 &
\small \AxiomC{$ \Gamma \Rightarrow \Diamond^{n+1} p$}
 \UnaryInfC{$ \Gamma \Rightarrow \bigvee_{0 \leq i}^n \Diamond^{i} p$}
 \DisplayProof
 &
 \small \AxiomC{$\Gamma \Rightarrow \Diamond \Box p$}
 \UnaryInfC{$ \Gamma \Rightarrow \Box \Diamond p$}
 \DisplayProof
 &
 \small \AxiomC{$\Gamma \Rightarrow \Diamond^k \Box^l p$}
 \UnaryInfC{$ \Gamma \Rightarrow \Box^m \Diamond^n p$}
 \DisplayProof
 \\
 \midrule
 
 \small \AxiomC{$\Gamma, p \Rightarrow \Delta$}
 \UnaryInfC{$ \Gamma, \Box p \Rightarrow \Delta$}
 \DisplayProof
 &
\small \AxiomC{$\Gamma \Rightarrow p$}
 \UnaryInfC{ $\Gamma \Rightarrow   \Box p$}
\hspace*{-1em}
\DisplayProof 
 & 
\small  \AxiomC{$\Gamma, p \Rightarrow \Delta $}
 \UnaryInfC{$ \Gamma, \Diamond p \Rightarrow \Delta$}
 \hspace*{-3em}
 \DisplayProof
 &
 \small \begin{tabular}{c c}
 
 \AxiomC{$ $}
 \UnaryInfC{$ \Gamma, \Box \bot \Rightarrow \Delta$}
 \hspace*{-6em}
 \DisplayProof
 \hspace{1em}
 \small \AxiomC{$ $}
 \UnaryInfC{$ \Gamma \Rightarrow \Diamond \top$}
 \DisplayProof 
  \end{tabular}
 \\
 \midrule
 
\small \AxiomC{$ \Gamma \Rightarrow \Diamond (p \wedge \Box q)$}
 \UnaryInfC{$ \Gamma \Rightarrow \Box (p \vee \Diamond q )$}
 \DisplayProof
 &
 \small \AxiomC{$\Gamma \Rightarrow \neg \Diamond p$}
 \UnaryInfC{$\Gamma \Rightarrow \Box \neg p$}
 \DisplayProof
 &
 \small \AxiomC{$\Gamma \Rightarrow \Box \neg p$}
 \UnaryInfC{$\Gamma \Rightarrow \neg \Diamond p$}
 \DisplayProof 
 &\small   \AxiomC{$\Gamma \Rightarrow  \Diamond \neg p$}
 \UnaryInfC{$\Gamma \Rightarrow \neg \Box p$}
 \DisplayProof
 \\
 \midrule  
 
 \small \AxiomC{$\{ \Gamma, \Diamond^{i} p \Rightarrow \Delta\}_{i=0}^n$}
 \UnaryInfC{$ \Gamma, \Diamond^{n+1} p \Rightarrow \Delta$}
 \DisplayProof
 &
 \small \AxiomC{$\Gamma, \Diamond p \Rightarrow  \bot$}
 \UnaryInfC{$\Gamma \Rightarrow \Box \neg p$}
 \DisplayProof
 &
 \small \AxiomC{$\Gamma \Rightarrow \Box \neg p$}
 \UnaryInfC{$\Gamma, \Diamond p \Rightarrow \Delta$}
 \DisplayProof 
 &
 \small   \AxiomC{$\Gamma \Rightarrow  \Diamond \neg p$}
 \UnaryInfC{$\Gamma, \Box p \Rightarrow \Delta$}
 \DisplayProof
 
 \\
 \midrule

\multicolumn{2}{c}{\small \AxiomC{$\{\Gamma, \Diamond (p_i \wedge (p_j \vee \Diamond p_j)) \Rightarrow \Delta \}_{0 \leq i \neq j \leq r} $}
 \UnaryInfC{$ \Gamma, \{\Diamond p_i\}_{i=0}^r \Rightarrow \Delta$}
 \hspace*{2em}\DisplayProof}
 &
 \multicolumn{2}{c}{\small \AxiomC{$\Gamma, p_n \Rightarrow \Delta$}
 \UnaryInfC{$ \Gamma, \Diamond (\Box p_n \wedge bd_n) \Rightarrow \Delta$}
 \DisplayProof}
 \\
 \midrule 
 
\multicolumn{2}{c}{\small \AxiomC{$\Gamma \Rightarrow p$}
 \UnaryInfC{$ \Gamma \Rightarrow \Box (\Diamond p \to p)$}
 \DisplayProof}
&
\multicolumn{2}{c}{\small \AxiomC{$\Gamma \Rightarrow \Diamond (\Box p \wedge q)$}
 \UnaryInfC{$ \Gamma \Rightarrow \Box (\Diamond p \vee q)$}
 \DisplayProof}\\
   \bottomrule  
 \end{tabular}\vline }
\caption{\footnotesize The numbers $n, m, l, k \geq 0$ and $r\geq 1$ are arbitrary natural numbers. The formula $bd_{n}$ is defined recursively by $bd_0=\bot$, and $bd_{n+1}=p_n \vee \Box (\Diamond \neg p_n \vee bd_n)$.\label{table}}
\end{table}

\normalsize\begin{exam}
The rules of Table \ref{table} correspond to the axioms in Table \ref{tableAxiom}. Some of them may seem unfamiliar or even unnatural. However, as the presence of the cut rule in the studied systems will be assumed, they are equivalent to the usual rules used for the corresponding axioms. Moreover, notice that some of the rules in Table \ref{table} are equivalent. The point to mention the different versions is to convey the different forms that a rule might have. \\
Now, referring to Example \ref{ExampleOfFormulas}, we easily see that the rules of $\mathbf{LJ}$, and the rules in Table \ref{table} are all constructive. For instance, consider the rules:
\begin{center}
\begin{tabular}{c c c}
 \AxiomC{$\Gamma \Rightarrow p $}
  \AxiomC{$\Gamma, p \Rightarrow \Delta $}
    \RightLabel{\small$(cut)$}
 \BinaryInfC{$ \Gamma \Rightarrow \Delta$}
 \DisplayProof
 &
 \AxiomC{$\Gamma \Rightarrow p$}
  \RightLabel{\small$(R\vee)$}
 \UnaryInfC{$ \Gamma \Rightarrow p \vee q$}
 \DisplayProof
 &
 \AxiomC{$\Gamma, p \Rightarrow q$}
  \RightLabel{\small$(R\to)$}
 \UnaryInfC{$ \Gamma \Rightarrow p \to q$}
 \DisplayProof
\end{tabular}
\end{center}
The cut rule is left constructive as $p$ is both almost positive and constructive. The rules $(R\vee)$ and $(R\to)$ are right constructive as $p$ is basic, $q$ is almost positive and $p \to q$ and $p \vee q$ are both constructive. 
To see a modal example, consider the following rule which corresponds to the axiom $(ga_{klmn})$:
\begin{center}
\begin{tabular}{c c c}
 \AxiomC{$\Gamma \Rightarrow \Diamond^k \Box^l p $}
 \UnaryInfC{$ \Gamma \Rightarrow \Box^m \Diamond^n p$}
 \DisplayProof
\end{tabular}
\end{center}
It is a right constructive rule as $\Diamond^k \Box^l p$ is almost positive and $\Box^m \Diamond^n p$ is constructive. This implies that all the rules in the third, fourth and fifth rows of Table \ref{table} 
are also constructive, for they are special cases of the rule we just mentioned. To have more complex examples, consider:
\begin{center}
\begin{tabular}{c c}
\AxiomC{$\Gamma \Rightarrow \Diamond (p \vee q)$}
 \UnaryInfC{$ \Gamma \Rightarrow \Diamond p \vee \Diamond q$}
 \DisplayProof \;\;\;
 & 
\AxiomC{$\Gamma, \Diamond p \Rightarrow \Box q$}
 \UnaryInfC{$ \Gamma \Rightarrow \Box(p \to q)$}
\DisplayProof
\end{tabular}
\end{center}
The left rule is a right constructive rule as $\Diamond (p \vee q)$ is almost positive and $\Diamond p \vee \Diamond q$ is basic and hence constructive. The right rule is also a right constructive rule as $\Diamond p$ is basic, $\Box q$ is almost positive, and $\Box(p \to q)$ is a constructive formula. 
To have an example that uses the full power of the definition of the basic formulas, consider the rule:
\begin{center}
\begin{tabular}{c}
 \AxiomC{$\{ \Gamma \Rightarrow \Diamond p_i\}_{i=0}^r$}
 \UnaryInfC{$ \Gamma \Rightarrow \bigvee_{0 \leq i \neq j}^r \Diamond (p_i \wedge (p_j \vee \Diamond p_j))$}
 \DisplayProof
\end{tabular}
\end{center}
Despite its complicated form, the rule is also a right constructive rule as the complex formula $\bigvee_{0 \leq i \neq j}^r \Diamond (p_i \wedge (p_j \vee \Diamond p_j))$ is basic and $\Diamond p_i$ is almost positive. 
\end{exam}

\begin{exam}\label{Non-example}
To have some non-examples, consider the following rules:
\begin{center}
\small{
\begin{tabular}{c c c c}
 \AxiomC{ }
 \UnaryInfC{$\Gamma \Rightarrow p \vee \neg p $}
 \DisplayProof
&
 \AxiomC{$\Gamma \Rightarrow \neg \neg p$}
 \UnaryInfC{$\Gamma \Rightarrow p $}
 \DisplayProof
 &
 \AxiomC{$\Gamma, \neg p \Rightarrow \bot$}
 \UnaryInfC{$\Gamma \Rightarrow p $}
 \DisplayProof
 &
 \AxiomC{$\Gamma, p \Rightarrow \Delta$}
  \AxiomC{$\Gamma, \neg p \Rightarrow \Delta$}
 \BinaryInfC{$\Gamma \Rightarrow \Delta$}
 \DisplayProof
\end{tabular}}
\end{center}
None of these rules are constructive. For instance, the leftmost rule is an axiom, and to be a constructive axiom, the formula in the succedent of the sequent must be a constructive formula. However,  according to Definition \ref{DefPositiveFormulas} and as observed in Example \ref{ExampleOfFormulas}, the formula $p \vee \neg p$ is not constructive. For the other three, the reason is that the formula $\neg \neg p$ is not almost positive and $\neg p$ is not basic as observed in Example \ref{ExampleOfFormulas}. Another non-example is 
\begin{center}
\small{
\begin{tabular}{c}
 \AxiomC{$\Gamma, \neg p \Rightarrow \bot$}
 \AxiomC{$\Gamma, p \Rightarrow \Delta$}
 \BinaryInfC{$\Gamma \Rightarrow  \Delta $}
 \DisplayProof
\end{tabular}}
\end{center}
Here, according to the side conditions in Definition \ref{DefnegativeRules}, all the formulas in the antecedent of $(\Gamma, \neg p \Rightarrow \bot)$ must be basic, which is not the case, as $\neg p$ is not basic. 
Note that each of these rules implies an equivalent version of the axiom of the excluded middle. (In the rightmost rule, set $\Gamma= \{\neg \neg p\}$ and $\Delta=\{p\}$ to obtain $\neg \neg p \Rightarrow p$). Hence, we see how the conditions in the definition of the constructive rules are necessary to keep the system constructive. 

To see some modal rules as the non-examples, first note that the rules $(K_{\Box})$ and $(K_{\Diamond})$ are not constructive, as they change the context $\Gamma$ to $\Box \Gamma$ and hence
do not follow the enforced form. To see some other types of non-examples, consider the following rules: 
\begin{center}
\small \begin{tabular}{c c c}
 \AxiomC{$\Gamma, \Box(\Box p \to q) \Rightarrow \Delta$} 
 \AxiomC{$\Gamma, \Box(\Box q \to p) \Rightarrow \Delta$}
 \RightLabel{$(sc)$}
 \BinaryInfC{$\Gamma \Rightarrow \Delta$}
 \DisplayProof
 &
 \small   \AxiomC{$\Gamma \Rightarrow \Box \Diamond p$}
 \RightLabel{$(ma)$}
 \UnaryInfC{$\Gamma \Rightarrow \Diamond \Box p$}
 \DisplayProof
\end{tabular}
\end{center}
\begin{center}
\small \begin{tabular}{c c c c}
\AxiomC{$\Gamma \Rightarrow \Box (\Box p \to p)$}
 \RightLabel{$(la)$}
 \UnaryInfC{$\Gamma \Rightarrow \Box p $}
 \DisplayProof
 &
  \AxiomC{$\Gamma \Rightarrow \Box (\Box (p \to \Box p)\to p) $} 
  \RightLabel{$(grz)$}
 \UnaryInfC{$\Gamma \Rightarrow p$}
 \DisplayProof
 &
 \small
   \AxiomC{$\Gamma \Rightarrow \neg \Box p $} 
 \UnaryInfC{$\Gamma \Rightarrow \Diamond \neg p$}
 \DisplayProof
\end{tabular}
\end{center}
\begin{center}
\small \begin{tabular}{c c }
 \AxiomC{$\Gamma \Rightarrow \Box \Diamond p$} 
 \AxiomC{$\Gamma \Rightarrow \Diamond p$}
 \RightLabel{$(.1)$}
 \BinaryInfC{$\Gamma \Rightarrow \Diamond \Box p \vee \Box p$}
 \DisplayProof
 &
 \small \AxiomC{$ $}
\UnaryInfC{$\Gamma \Rightarrow \Diamond \top \vee \Box \bot$} 
\DisplayProof
\end{tabular}
\end{center}
None of these rules are constructive. The rule $(sc)$ is not constructive, as otherwise, it must have been a left constructive rule, while it has two premises, with $\Delta$ as the succedents, and the formulas $\Box (\Box p \to q)$ and $\Box (\Box q \to p)$ in the antecedents are not basic. The rules $(ma)$ and $(la)$ are not constructive either, as otherwise, they must have been right constructive rules, but the formulas $\Box (\Box p \to p)$ and $\Diamond \Box p$ are not almost positive and constructive, respectively. The other rules are not constructive, as the formula $\Box (\Box (p \to \Box p)\to p) $ is not almost positive and the formulas $\Diamond \neg p$, $\Diamond \Box p \vee \Box p$ and $\Diamond \top \vee \Box \bot$ are not constructive.
It is worth mentioning that the rules $(ma)$, $(la)$ and $(grz)$ correspond to the McKinsey axiom $\Box \Diamond p \to \Diamond \Box p$, the L\"{o}b axiom $\Box (\Box p \to p) \to \Box p$, and the Grzegorczyk axiom $\Box (\Box (p \to \Box p)\to p) \to p$, respectively. Note that none of these formulas are constructive.
\end{exam}

\subsection{General Rules and the Justification} \label{SubsectionJustification}
To follow our strategy as explained in the opening discussion of Section \ref{SectionAlmostnegative}, let us first introduce a general form for the rules that we want to consider and then select the ones that are equivalent to a constructive formula.
\begin{dfn}\label{Def:RulesR}
Let $\mathfrak{R}$ be the set of all the rules in one of the forms: 
\begin{center}
\begin{tabular}{c c}
\AxiomC{$ \{ \Gamma, \overline{\phi}_i \Rightarrow \overline{\psi}_i \}_{i \in I}$}
\AxiomC{$ \{ \Gamma, \overline{\theta}_j \Rightarrow \Delta \}_{j \in J}$}
 \BinaryInfC{$\Gamma, \overline{\eta} \Rightarrow  \Delta $}
 \DisplayProof
&
 \AxiomC{$\{ \Gamma, \overline{\phi}_i \Rightarrow \overline{\psi}_i \}_{i\in I}$}
 \UnaryInfC{$\Gamma, \overline{\theta} \Rightarrow \overline{\eta} $}
 \DisplayProof
\end{tabular}
\end{center}
where $\Gamma$ and $\Delta$ are multiset  variables, free to be substituted by any multisets and $\overline{\phi}_i$'s, $\overline{\psi}_i$'s, $\overline{\theta}_j$'s, $\overline{\theta}$ and $\overline{\eta}$ are multiset of formulas. We call any rule in the left form a \emph{left} rule and any rule in the right form a \emph{right} rule. 
\end{dfn}
These two forms capture a general common structure in the usual single-conclusion rules covering all possible combinations of formulas and multiset variables subject to two main restrictions. First, they must be single-conclusion and hence in the succedent we do not allow to have both multisets of formulas and $\Delta$. The second is the presence of the multiset variables that are free for any multiset substitution. Note that they remain intact in the rule application. These forms cover all the axioms and rules of the system $\mathbf{LJ}$ including the cut rule. They also cover all possible axioms in the form $\Gamma \Rightarrow \eta$. Note that in the presence of the cut rule, many rules are equivalent to an axiom and hence the forms must be seen as quite powerful and general. Having that said, we must also emphasize that these forms unfortunately do not cover all possible interesting rules in the literature. The reason is mostly the restriction on the context.
For instance, consider the following rules:
\begin{center}
\begin{tabular}{c c c}
\AxiomC{$\neg p \Rightarrow q \vee r$}
\RightLabel{\footnotesize $(KP)$}
 \UnaryInfC{$ \Rightarrow  (\neg p \to q) \vee (\neg p \to r) $}
 \DisplayProof
&
 \AxiomC{$\Box\Gamma \Rightarrow p$}
 \RightLabel{\footnotesize $(RS4)$}
 \UnaryInfC{$\Box\Gamma \Rightarrow \Box p$}
 \DisplayProof
 &
  \AxiomC{$\Gamma \Rightarrow p$}
  \RightLabel{\footnotesize $(K_{\Box})$}
 \UnaryInfC{$\Box \Gamma \Rightarrow \Box p$}
 \DisplayProof
\end{tabular}
\end{center}
In $(KP)$, as typical for non-derivable admissible rules, the multiset variable $\Gamma$ is missing. In $(RS4)$, there is a restriction on the form of the antecedents of the meta-sequents (only boxed formulas are allowed). And in $(K_{\Box})$, the multiset variable changes from $\Gamma$ in the premise to $\Box \Gamma$ in the conclusion.

Having the forms fixed, in the second step (the following theorem), we prove that over $\mathbf{CK}$, any rule $R \in \mathfrak{R}$ is equivalent to a formula, denoted by $Ax_R$, and $R$ is constructive if{f} $Ax_R$ is constructive. Therefore, we can argue that the constructive rules are all constructively acceptable and as the forms of the rules in $\mathfrak{R}$ are chosen to be quite general, the constructive rules are also sufficiently general to capture the constructively acceptable rules. 

\begin{thm}\label{AxR}
For any rule $R \in \mathfrak{R}$, there is a formula $Ax_R$ such that:
\begin{itemize}
\item[$(i)$]
$\mathbf{LJ}+R$ feasibly proves $(\Rightarrow Ax_R)$. 
\item[$(ii)$]
$\mathbf{LJ}+Ax_R$ feasibly proves the rule $R$.
\end{itemize}
Moreover, $R$ is constructive if and only if $Ax_R$ is a constructive formula.
\end{thm}
\begin{proof}
For any right rule $R \in \mathfrak{R}$
\begin{center}
\begin{tabular}{c c}
 \AxiomC{$\{ \Gamma, \overline{\phi}_i \Rightarrow \overline{\psi}_i \}_{i\in I}$}
 \UnaryInfC{$\Gamma, \overline{\theta} \Rightarrow \overline{\eta} $}
 \DisplayProof
\end{tabular}
\end{center}
define $Ax_R$ as $[\bigwedge_{i \in I} [\bigwedge \overline{\phi}_i \to  \bigvee \overline{\psi}_i] \wedge \bigwedge \overline{\theta}] \to \bigvee \overline{\eta}$ and for any left rule $R \in \mathfrak{R}$
\begin{center}
\begin{tabular}{c c c}
\AxiomC{$ \{ \Gamma, \overline{\phi}_i \Rightarrow \overline{\psi}_i \}_{i \in I}$}
\AxiomC{$ \{ \Gamma, \overline{\theta}_j \Rightarrow \Delta \}_{j \in J}$}
 \BinaryInfC{$\Gamma, \overline{\eta} \Rightarrow  \Delta $}
 \DisplayProof
\end{tabular}
\end{center}
define $Ax_R$ as $\bigwedge_{i \in I} (\bigwedge \overline{\phi}_i \to \bigvee \overline{\psi}_i) \wedge \bigwedge \overline{\eta} \to \bigvee_{j \in J} \bigwedge \overline{\theta}_j$. We only prove the claim for left rules. The case for right rules is similar. 

For $(i)$, we first show that $(\Rightarrow Ax_R)$ is provable in $\mathbf{LJ} + R$. The algorithm is as follows. Set $\Gamma =\{\bigwedge \overline{\phi}_i \to \bigvee \overline{\psi}_i\}_{i \in I}$ and $\Delta= \bigvee_{j \in J}(\bigwedge \overline{\theta}_j)$. Therefore, we have $\mathbf{LJ} \vdash \Gamma, \overline{\phi}_i \Rightarrow \bigvee \overline{\psi}_i$ and $\mathbf{LJ} \vdash \Gamma, \overline{\theta}_j \Rightarrow \Delta$. By applying the rule $R$, we reach
$\mathbf{LJ} + R \vdash \{\bigwedge \overline{\phi}_i \to \bigvee \overline{\psi}_i\}_{i \in I}, \overline{\eta} \Rightarrow \bigvee_{j \in J}(\bigwedge \overline{\theta}_j)$. Then, by using the rules $(L\wedge)$, $(Lc)$ and $(R \to)$ in $\mathbf{LJ}$, we get
$\mathbf{LJ} + R \vdash (\Rightarrow Ax_R)$. The feasibility of the algorithm is clear.

For $(ii)$, denoting $\mathbf{LJ}+Ax_R $ by $H$, we have to provide a polynomial time algorithm $f_R$ that reads the sequents in the multiset $\{ \Gamma, \overline{\phi}_i \Rightarrow \overline{\psi}_i \}_{i \in I} \cup \{ \Gamma, \overline{\theta}_j \Rightarrow \Delta \}_{j \in J}$ and $\Gamma, \overline{\eta} \Rightarrow \Delta$ and provides an $H$-proof witnessing
\begin{center}
$\{ \Gamma, \overline{\phi}_i \Rightarrow \overline{\psi}_i \}_{i \in I} , \{ \Gamma, \overline{\theta}_j \Rightarrow \Delta \}_{j \in J} \vdash_H \Gamma, \overline{\eta} \Rightarrow \Delta$. 
\end{center}
We only present the algorithm $f_R$. The fact that it is polynomial time is trivial.
First, observe that by a simple application of the cut rule on the axiom $(\Rightarrow Ax_R)$ and the $\mathbf{LJ}$-provable sequent
\[
Ax_R, [\bigwedge_{i \in I}(\bigwedge \overline{\phi}_i \to \bigvee \overline{\psi}_i) \wedge \bigwedge \overline{\eta}] \Rightarrow \bigvee_{j \in J}(\bigwedge \overline{\theta}_j).
\]
we have
\[
H \vdash [\bigwedge_{i \in I}(\bigwedge \overline{\phi}_i \to \bigvee \overline{\psi}_i) \wedge \bigwedge \overline{\eta}] \Rightarrow \bigvee_{j \in J}(\bigwedge \overline{\theta}_j)
\]
Then, notice that 
\begin{center}
    $\{\Gamma, \overline{\phi}_i \Rightarrow \overline{\psi}_i \}_{i \in I} \vdash_H \Gamma \Rightarrow \bigwedge_{i \in I}(\bigwedge \overline{\phi}_i \to \bigvee \overline{\psi}_i)$
\end{center}
and hence
\begin{center}
   $\{\Gamma, \overline{\phi}_i \Rightarrow \overline{\psi}_i \}_{i \in I} \vdash_H \Gamma, \overline{\eta} \Rightarrow \bigwedge_{i \in I}(\bigwedge \overline{\phi}_i \to \bigvee \overline{\psi}_i) \wedge \bigwedge \overline{\eta}$.
\end{center}
By using the cut rule on the above sequent and the aforementioned $H$-provable sequent $[\bigwedge_{i \in I}(\bigwedge \overline{\phi}_i \to \bigvee \overline{\psi}_i) \wedge \bigwedge \overline{\eta}] \Rightarrow \bigvee_{j \in J}(\bigwedge \overline{\theta}_j)$, we have 
\begin{center}
$\{\Gamma, \overline{\phi}_i \Rightarrow \overline{\psi}_i \}_{i \in I} \vdash_H \Gamma, \overline{\eta} \Rightarrow \bigvee_{j \in J}(\bigwedge \overline{\theta}_j)$ \;\; (1)
\end{center}
On the other hand, we have 
\begin{center}
    $\{ \Gamma, \overline{\theta}_j \Rightarrow \Delta \}_{j \in J} \vdash_H \Gamma, \bigvee_{j \in J}(\bigwedge \overline{\theta}_j) \Rightarrow \Delta$ \;\; (2)
\end{center}
Now, using the cut rule on (1) and (2) and several weakening rules we get 
\begin{center}
   $\{\Gamma, \overline{\phi}_i \Rightarrow \overline{\psi}_i \}_{i \in I}, \{ \Gamma, \overline{\theta}_j \Rightarrow \Delta \}_{j \in J} \vdash_H \Gamma, \overline{\eta} \Rightarrow \Delta$.
\end{center}
The above process of using some simple $H$-provable facts with their simple $H$-proofs and then applying some cut rules clearly takes polynomial time in the sum of the sizes of the sequents $\{ \Gamma, \overline{\phi}_i \Rightarrow \overline{\psi}_i \}_{i \in I} \cup \{ \Gamma, \overline{\theta}_j \Rightarrow \Delta \}_{j \in J}$ and $\Gamma, \overline{\eta} \Rightarrow \Delta$. The presence of the cut rule, assumed to be a primitive rule in $\mathbf{LJ}$, is crucial to provide the required short proofs and make $f_R$ feasible.

As the last part of the proof, we show that $R$ is constructive if{f} $Ax_R$ is constructive. 
For any left rule $R \in \mathfrak{R}$
\begin{center}
\begin{tabular}{c c c}
\AxiomC{$ \{ \Gamma, \overline{\phi}_i \Rightarrow \overline{\psi}_i \}_{i \in I}$}
\AxiomC{$ \{ \Gamma, \overline{\theta}_j \Rightarrow \Delta \}_{j \in J}$}
 \BinaryInfC{$\Gamma, \overline{\eta} \Rightarrow  \Delta $}
 \DisplayProof
\end{tabular}
\end{center}
the formula $Ax_R$ defined as 
\begin{center}
    $[\bigwedge_{i \in I}(\bigwedge \overline{\phi}_i \to \bigvee \overline{\psi}_i) \wedge \bigwedge \overline{\eta}] \to \bigvee_{j \in J}(\bigwedge \overline{\theta}_j)$
\end{center}
is constructive if{f} the antecedent, i.e., $\bigwedge_{i \in I}(\bigwedge \overline{\phi}_i \to \bigvee \overline{\psi}_i) \wedge \bigwedge \overline{\eta}$, is almost positive and the succedent, i.e., $\bigvee_{j \in J}(\bigwedge \overline{\theta}_j)$ is constructive. To enforce that, there are several cases to consider. For the succedent:
\begin{itemize}
    \item 
    If $|J|>1$, then all the formulas in $\overline{\theta}_j$ must be basic as the disjunction of at least two formulas is constructive if{f} all of them are basic.
\item
If $|J|=1$, then all the formulas in any of $\overline{\theta}_j$'s must be constructive.
\item
If $|J|=0$, then  $\bigvee_{j \in J}(\bigwedge \overline{\theta}_j)=\bot$, hence it is constructive automatically.
\end{itemize}
For the antecedent of the formula, 
\begin{itemize}
\item
If $|I|=0$, then the antecedent of the formula is defined as $\top$, which is an almost positive formula automatically.
    \item 
    If $|I|\geq 1 $, then all the formulas in $\overline{\phi}_i$'s must be basic while all the formulas in $\overline{\psi}_i$'s and $\overline{\eta}$ must be almost positive as this is the only way to make the antecedent almost positive.
\end{itemize}
These conditions altogether are the ones to make the rule constructive.\\
For any right rule $R \in \mathfrak{R}$ of the form
\begin{center}
\begin{tabular}{c c}
 \AxiomC{$\{ \Gamma, \overline{\phi}_i \Rightarrow \overline{\psi}_i \}_{i\in I}$}
 \UnaryInfC{$\Gamma, \overline{\theta} \Rightarrow \overline{\eta} $}
 \DisplayProof
\end{tabular}
\end{center}
the formula $Ax_R$ defined as $[\bigwedge_{i \in I} [\bigwedge \overline{\phi}_i \to  \bigvee \overline{\psi}_i] \wedge \bigwedge \overline{\theta}] \to \bigvee \overline{\eta}$ is constructive if{f} $[\bigwedge_{i \in I} [\bigwedge \overline{\phi}_i \to  \bigvee \overline{\psi}_i]$ and  $\bigwedge \overline{\theta}$ are almost positive and $\bigvee \overline{\eta}$ is constructive. This means that all $\bigwedge \overline{\phi}_i \to  \bigvee \overline{\psi}_i$'s and all formulas in $\overline{\theta}$ are almost positive and any formula in $\overline{\eta}$ is constructive. Notice that as the rule is single-conclusion, the disjunction in $\bigvee \overline{\eta}$ is over either one or zero formulas and hence is not problematic. Finally, $\bigwedge \overline{\phi}_i \to  \bigvee \overline{\psi}_i$ is almost positive if{f} all the elements in $\bar{\phi}_i$'s and $\overline{\psi}_i$'s are basic and almost positive, respectively. These conditions altogether are the ones to make the rule constructive.
\end{proof}

\begin{cor}\label{EquivProvableAndFeasiblyProvable}
Let $\mathfrak{L} \in \{\mathcal{L}, \mathcal{L}_{\Box}, \mathcal{L}_{\Diamond}, \mathcal{L}_{p}\}$ be a language, $R \in \mathfrak{R}$ a rule over $\mathfrak{L}$ and $G$ a sequent calculus over $\mathfrak{L}$ such that $G$ feasibly proves all the rules in $\mathbf{LJ}$. Then, $R$ is provable in $G$ if{f} it is feasibly provable in $G$.
\end{cor}
\begin{proof}
We prove the claim for $\mathcal{L}$, the others are similar. One direction is clear. For the other, assume that $R$ is provable in $G$. As $G$ proves all the rules in $\mathbf{LJ}$ and $\mathbf{LJ}+R$ proves $(\, \Rightarrow Ax_R)$ by Theorem \ref{AxR}, then $G$ proves $(\, \Rightarrow Ax_R)$. We claim that $G$ also feasibly proves $(\, \Rightarrow Ax_R)$. Take a proof of $(\, \Rightarrow Ax_R)$ in $G$ and call it $\rho$. Then, for any instance of $(\, \Rightarrow Ax_R)$ resulting from the substitution $\sigma$, the application $\sigma$ on $\rho$, denoted by $\sigma(\rho)$, is a proof for that instance. As $\rho$ is fixed, clearly the process of reading the instance of $(\Rightarrow Ax_R)$ and providing the $G$-proof $\sigma(\rho)$ is feasible. Hence, $G$ feasibly proves $(\, \Rightarrow Ax_R)$.
By Theorem \ref{AxR}, $\mathbf{LJ}+Ax_R$ feasibly proves $R$ and as $G$ feasibly proves all the rules in $\mathbf{LJ}+Ax_R$, by Lemma \ref{LocalToGlobal} we have $G$ pd-simulates $\mathbf{LJ}+Ax_R$. Finally, by Lemma \ref{LocalToGlobal} part \ref{Lemmamorede1}, $R$ is feasibly provable in $G$.
\end{proof}

\begin{cor}\label{thm: p-simulation of admissibly strong}
Let $G$ be a strong constructive sequent calculus over $\mathcal{L}$. Then, there exists a finite set of constructive $\mathcal{L}$-formulas $\mathcal{C}$ such that $G$ is pd-equivalent to $\mathbf{CK}+\mathcal{C}$. A similar claim holds, replacing the pair $(\mathcal{L}, \mathbf{CK})$ by $(\mathcal{L}_{\Box}, \mathbf{CK}_{\Box})$, $(\mathcal{L}_{\Diamond}, \mathbf{BLL})$, or $(\mathcal{L}_{p}, \mathbf{LJ})$.
\end{cor}
\begin{proof}
We only prove the case for $(\mathcal{L}, \mathbf{CK})$. The proof for the fragments is similar. Define $\mathcal{R}$ as the set of the constructive rules of $G$ and set $\mathcal{C}=\{Ax_R \mid R \in \mathcal{R}\}$. As $G$ is strong, it feasibly proves all the rules in $\mathbf{CK}$ and specifically all the rules in $\LJ$. Hence, by Lemma \ref{LocalToGlobal} $G$ pd-simulates $\LJ$. As $R$ is in $G$, by Remark \ref{RemarkRule}, $G$ feasibly proves $R$. By Theorem \ref{AxR}, $\mathbf{LJ}+R$ feasibly proves $(\Rightarrow Ax_R)$. Therefore,  by Lemma \ref{LocalToGlobal} part \ref{Lemmamorede1}, $G$ also feasibly proves $(\, \Rightarrow Ax_R)$. Hence, again by Lemma \ref{LocalToGlobal}, $G$ pd-simulates $\mathbf{CK}+\mathcal{C}$. For the other direction, notice that $\mathbf{CK}+\mathcal{C}$ feasibly proves the rule $R \in \mathcal{R}$, by Theorem \ref{AxR} and Lemma \ref{LocalToGlobal}. For the other rules of $G$, as they can only be $(K_{\Box})$ or $(K_{\Diamond})$, they are feasibly provable in $\mathbf{CK}+\mathcal{C}$. Therefore, by Lemma \ref{LocalToGlobal}, $\mathbf{CK}+\mathcal{C}$ pd-simulates $G$.
\end{proof}

\section{Feasible Visser-Harrop Property} \label{SectionMain}
In this section, we prove the main result of the paper stating that any strong constructive sequent calculus over $\mathcal{L}$, satisfying a modest technical condition, feasibly admits a generalization of Visser's rules. For the moment, let us ignore the generalization. By the feasible admissibility of Visser's rule, we mean there is a polynomial time algorithm that reads a $G$-proof $\pi$ for \[ \{ A_i \to B_i \}_{i \in I} \Rightarrow C \vee D,\] where $I$ is an index set and outputs a $G$-proof for one of the following sequents
\small \begin{center}
\begin{tabular}{c c c}
$ \{ A_i \to B_i \}_{i \in I} \Rightarrow C\;$ or & $  \{ A_i \to B_i \}_{i \in I} \Rightarrow D\;$ or
& $ \{ A_i \to B_i \}_{i \in I} \Rightarrow A_i$,\\
\end{tabular}
\end{center}
\normalsize for some $i \in I$. 
Using Corollary \ref{thm: p-simulation of admissibly strong}, it is enough to prove the claim for the systems $G=\mathbf{CK}+\mathcal{C}$, where $\mathcal{C}$ is a finite set of constructive formulas. To prove this claim, we need to develop the following two ingredients. First, a translation function to transform the provable sequent $\{ A_i \to B_i \}_{i \in I} \Rightarrow C \vee D$ to a simpler provable sequent $\Sigma,\{\neg p_i \}_{i \in I} \Rightarrow q \vee r$, where $p_i$'s, $q$ and $r$ are all atomic formulas and $\Sigma$ consists of some sort of simple formulas. The second ingredient is a proof theoretical version of the usual unit propagation algorithm to read the result of the first part, i.e., $\Sigma,\{ \neg p_i \}_{i \in I} \Rightarrow q \vee r$ and find a proof for either $\Sigma, \{\neg p_i\}_{i \in I} \Rightarrow q$ or $\Sigma, \{\neg p_i\}_{i \in I} \Rightarrow q$ or $\Sigma, \{\neg p_i \}_{i \in I} \Rightarrow p_i$, for some $i \in I$. Finally, by applying the converse of the transformation of the first part, we can provide a proof for either $\{A_i \to B_i \}_{i \in I} \Rightarrow C$ or $\{A_i \to B_i \}_{i \in I} \Rightarrow D$ or $\{A_i \to B_i \}_{i \in I} \Rightarrow A_i$, for some $i \in I$.
We will cover the first ingredient in Subsection \ref{SubsectionTranslation} and Subsection \ref{SubsectionPreservability} and the second ingredient will be explained in Subsection \ref{SubsectionUnitPropagation} and Subsection \ref{SubsectionTfreeTfull}. Finally, in Subsection \ref{SubsectionMain}, we will combine these two ingredients to prove the theorem.

\subsection{The Translations and their Properties}\label{SubsectionTranslation}
In this subsection, we introduce the two translations we mentioned before and investigate their effects on basic, almost positive and constructive formulas.

\begin{dfn} \label{Language}
For any formula $\phi \in \mathcal{L}$, set $\langle \phi \rangle$ as a fresh 
 atomic formula, 
 called an \emph{angled atom}, and add it to $\mathcal{L}$. The new language is denoted 
 by $\mathcal{L}^{+}$.
\end{dfn}
Note that an \emph{atom} in
$\mathcal{L}^+$ is either an atom in $\mathcal{L}$ or an angled atom, and these possibilities do not intersect. Moreover, notice that in computing the size of the atom $\langle \phi \rangle$, we consider all the symbols in $\langle \phi \rangle$. Hence, $|\langle \phi \rangle|=|\phi|+2$.

\begin{dfn} \label{Translation}
The translation function $t: \mathcal{L} \to \mathcal{L}^+$ is defined as:
\begin{itemize}
\item[$\bullet$] 
$\bot^t = \bot$, $\top^t = \langle \top \rangle$, and $p^t= \langle p \rangle$, for any atomic formula $p$;
\item[$\bullet$]
$(A \circ B) ^t = (A^t \circ B^t) \wedge \langle A \circ B \rangle$, for any $\circ \in \{\wedge, \vee, \to\}$;
\item[$\bullet$]
$(\bigcirc A)^t =(\bigcirc A^t) \wedge \langle \bigcirc A \rangle$, for any $\bigcirc \in \{\Box, \Diamond\}$. 
\end{itemize}
For a multiset $\Gamma$, by $\Gamma^t$, we mean the multiset consisting of the translation of all the elements of $\Gamma$, i.e., $\Gamma^t = \{\gamma^t \mid \gamma \in \Gamma\}$. 
\end{dfn}
\begin{dfn}
The \emph{standard} substitution $s: {\L}^+ \to \L$ is defined as:
\begin{itemize}
\item[$\bullet$]
$\langle \varphi \rangle^s = \varphi$, $p^s=p$, for any formula $\varphi \in \mathcal{L}$ and $p$ is an atom in $\mathcal{L}$, $\bot$ or $\top$;
\item[$\bullet$]
$(A \circ B) ^s = A^s \circ B^s$, for any $\circ \in \{\wedge, \vee, \to\}$;
\item[$\bullet$]
$(\bigcirc A)^s =\bigcirc A^s$, for any $\bigcirc \in \{\Box, \Diamond\}$.
\end{itemize}
Define $\Gamma^s$ in the usual way, i.e., $\Gamma^s =\{\gamma^s \mid \gamma \in \Gamma\}$.   
\end{dfn}
The standard substitution $s: {\L}^+ \to \L$ is the map that substitutes the angled atom $\langle \phi \rangle$ by $\phi$ and leaves the non-angled atomic formulas intact. By induction, it is easy to see that for any formula $A(p_1, \cdots, p_n) \in \mathcal{L}$, we have  $A(\langle \phi_1 \rangle , \dots , \langle \phi_n \rangle)^s = A (\phi_1, \dots, \phi_n)$.
The substitution $s$ can be interpreted as a translation function, mapping $\mathcal{L}^+$ into $\mathcal{L}$, cancelling all the changes made by the translation $t$ and tracing back the original formula.
\begin{lem}\label{LemTimeT}
The 
functions $t$ and $s$ are polynomial time computable. 
\end{lem}
\begin{proof}
First, we prove the claim for the function $t$. Consider the canonical recursive algorithm that computes $A^t$ and denote the time of this algorithm by $T_t(A)$. We show that $T_t(A) \leq O(|A|^2)$, using the following inequalities:
\begin{itemize}
\item
$T_t(A) \leq O(1)$, for the atomic $A$ (including $\bot$ and $\top$),
\item
$T_t(B \circ C) \leq T_t(B)+T_t(C)+|B|+|C|+O(1)$,
\item
$T_t(\bigcirc B) \leq T_t(B)+|B|+O(1)$.
\end{itemize}
for any $\circ \in \{\wedge, \vee, \to\}$ and $\bigcirc \in \{\Box, \Diamond\}$. The atomic case is obvious. For the second case, to compute $(B \circ C)^t$ we have to first compute $B^t$ and $C^t$, put them together separated by $\circ$ and finally add $\langle B \circ C \rangle$ to its end, separated by a $\wedge$. Therefore, the bound is clear. The argument for the third case is similar. By these inequalities and induction on the structure of $A$, it is easy to see that $T_t(A) \leq O(|A|^2)$. The case for the function  $s$ is similar and in fact easier. Consider the canonical recursive algorithm that computes $A^s$ by first implementing the substitutions for the atoms and then mimicking the structure of $A$. Denoting the time of this algorithm by $T_s(A)$ and using a similar type of inequalities as before, we get $T_s(A) \leq O(|A|^2)$. 
\end{proof}

\begin{lem} \label{TranslationAndAtoms}
$\mathbf{CK} \vdash A^t \Rightarrow \langle A \rangle$, for any formula $A \in \mathcal{L}$.
\end{lem}

\begin{proof}
The case where $A$ is atomic, $\bot$, or $\top$ is trivial, as $A^t=\langle A \rangle$ or $A^t=\bot$.  If $A$ is of the form $B \circ C$ or $\bigcirc B$, where $\circ \in \{\wedge, \vee, \to\}$ or $\bigcirc \in \{\Box, \Diamond\}$, then by Definition \ref{Translation}, we have $(B \circ C)^t = (B^t \circ C^t) \wedge \langle B \circ C \rangle= (B^t \circ C^t) \wedge \langle A \rangle$ and $(\bigcirc B)^t=\bigcirc B^t \wedge \langle \bigcirc B \rangle =\bigcirc B^t \wedge \langle A \rangle$. In either case, the proof is clear.
\end{proof}


Recall that an atom $p \in \L^+$ is either an atomic formula in the original language $\mathcal{L}$ or a new added angled atom.

\begin{dfn}\label{ImplicationalHorn}
The set of \emph{implicational Horn} formulas is the smallest set of $\mathcal{L}^+$-formulas containing $\bot$, atomic formulas, and closed under implications of the form $\bigwedge Q \to r$, where $Q=\{q_1, \dots , q_n\} \subseteq \L^+$ is a non-empty multiset of atoms and $r$ is either $\bot$ or an atom in $\L^+$. The set of \emph{modal Horn} formulas is the smallest set of $\mathcal{L}^+$-formulas containing $\bot$, atomic formulas, and closed under $\Box$ and implications of the form $A \to B$, where $A$ is of the form $\bigwedge_{i=1}^{k} \Diamond^{n_i} p_{i}$ for some $k \geq 1$ and $n_i \geq 0$ and $B$ is a modal Horn formula. 
\end{dfn}

\begin{rem}
Here are two remarks. First, as mentioned in Preliminaries, by $\Diamond^0 p$, we mean $p$. Therefore, it is easy to see that any implicational Horn formula is also a modal Horn formula. Second, in the literature usually Horn formulas are defined as formulas in the conjunctive normal form (CNF) such that each conjunct contains at most one positive literal. In the classical logic, our implicational Horn formulas are equivalent to Horn formulas defined as CNF's. However, this is not generally the case for non-classical logics. 
\end{rem}
Our first aim, as it usually happens in translations, is to show that $t$ preserves the provability in $\mathbf{CK}+\mathcal{C}$, for any finite set $\mathcal{C}$ of constructive formulas. Unfortunately, due to the addition of the new atoms and their use in the translation, the preservation in general does not hold. However, a slightly weaker form is true if we see the sequents up to some ``harmless'' modal Horn formulas in the antecedents. More precisely, if $\Gamma \Rightarrow \Delta$ is provable in $\mathbf{CK}+\mathcal{C}$, then there is a multiset $\Sigma$ of modal Horn formulas such that $\Sigma, \Gamma^t \Rightarrow \Delta^t$ is also provable in $\mathbf{CK}+\mathcal{C}$, while the formulas in $\Sigma$ are harmless in the sense that $\mathbf{CK}+\mathcal{C} \vdash \, \Rightarrow \bigwedge \Sigma^s$. Roughly speaking, the translation preserves the provability, up to the modal Horn formulas that $s$ sees as $(\mathbf{CK}+\mathcal{C})$-provable. To prove this property, we need a machinery to commute the translation $t$ with the constructive formulas, simply because we need to show that the translation of an instance of an axiom is an instance of the axiom itself. The most desirable such commutation would be the provability of both $A(\overline{\phi})^t \Rightarrow A(\overline{\phi^t})$ and $A(\overline{\phi^t}) \Rightarrow A(\overline{\phi})^t$ in $\mathbf{CK}$, for any constructive formula $A(\overline{p})$ and any formulas $\overline{\phi}$. Unfortunately, such a situation rarely takes place. To solve the issue, again adding a harmless multiset $\Sigma'$ of modal Horn formulas helps. In fact, we show that for a basic, almost positive or constructive formula $A(\overline{p})$ either $\Sigma', A(\overline{\phi})^t \Rightarrow A(\overline{\phi^t})$ or $\Sigma', A(\overline{\phi^t}) \Rightarrow A(\overline{\phi})^t$ or both are provable in $\mathbf{CK}$. For basic formulas $A(\overline{p})$, both directions are provable. However, for almost positive and constructive formulas, only one direction can be proved, and in the case of constructive formulas, even an additional formula is needed to make the sequent provable in $\mathbf{CK}$. The following theorem is devoted to these commutations. Note that although our main goal is the commutation with the constructive formulas, we also need to address the other two families as the steps to reach our goal.

\begin{thm}\label{Commutation}
We have the following commutations.
\begin{itemize}
\item[$(i)$]
For any basic formula $A(\overline{p}) \in \mathcal{L}$ and formulas $\overline{\phi} \in \mathcal{L}^+$, there is a multiset of modal Horn formulas $\Phi_{A, \overline{\phi}}$ constructed from angled atoms such that
\[
\mathbf{CK} \vdash \Phi_{A, \overline{\phi}}, (A(\overline{\phi}))^t \Rightarrow A(\overline{\phi^t}) \quad \text{and} \quad \mathbf{CK} \vdash \Phi_{A, \overline{\phi}}, A(\overline{\phi^t}) \Rightarrow (A(\overline{\phi}))^t.
\]
\item[$(ii)$]
For any almost positive formula $A(\overline{p})\in \mathcal{L}$ and formulas $\overline{\phi} \in \mathcal{L}^+$, there is a multiset of modal Horn formulas $\Pi_{A, \overline{\phi}}$ constructed from angled atoms such that 
\begin{center}
    $\mathbf{CK} \vdash \Pi_{A, \overline{\phi}}, (A(\overline{\phi}))^t \Rightarrow A(\overline{\phi^t}).$
\end{center}
\item[$(iii)$]
For any constructive formula $A(\overline{p})\in \mathcal{L}$ and formulas $\overline{\phi} \in \mathcal{L}^+$, there is a multiset of modal Horn formulas $\Upsilon_{A, \overline{\phi}}$ constructed from angled atoms such that 
\begin{center}
    $\mathbf{CK} \vdash \Upsilon_{A, \overline{\phi}}, \langle A(\overline{\phi}) \rangle, A(\overline{\phi^t}) \Rightarrow (A(\overline{\phi}))^t.$
\end{center}
\end{itemize}
For any $\Theta \in \{\Phi, \Pi, \Upsilon \}$ and formulas $A(\overline{p}) \in \mathcal{L}$ and $\overline{\phi} \in \mathcal{L}^+$, there is a proof $\sigma_{\Theta, A, \overline{\phi}}$ such that $\mathbf{CK} \vdash^{\sigma_{\Theta, A, \overline{\phi}}} (\, \Rightarrow \bigwedge \Theta_{A, \overline{\phi}}^{s})$, where $s$ is the standard substitution. Moreover, the processes of finding $\Theta_{A, \overline{\phi}}$ and $\sigma_{\Theta, A, \overline{\phi}}$ are polynomial time computable in the inputs $A(\overline{p})$ and $\overline{\phi}$. 
\end{thm}

\begin{proof} 
We first provide $\Theta_{A, \overline{\phi}}$ and $\sigma_{\Theta, A, \overline{\phi}}$ in each case, ignoring the complexity issues altogether. Then, we will address the feasibility of the algorithms in the last part of the proof.

For $(i)$, we use recursion on the structure of the basic formula $A(\overline{p})$ to define $\Phi_{A, \overline{\phi}}$ and $\sigma_{\Phi, A, \overline{\phi}}$. If $A(\overline{p})$ is either an atom or $\bot$, for any formulas $\overline{\phi}$, we take $\Phi_{A, \overline{\phi}}$ to be the empty set. Therefore, both sequents are trivially provable in $\mathbf{CK}$. Moreover, since $\bigwedge \emptyset$ is defined as $\top$, the sequent $(\, \Rightarrow \bigwedge \Phi_{A, \overline{\phi}}^{s})$ is an axiom an hence provable in $\mathbf{CK}$. Define the proof $\sigma_{\Phi, A, \overline{\phi}}$ as $(\Rightarrow \top)$.
If $A(\overline{p})=\top$, we take $\Phi_{A, \overline{\phi}}$ to be $\{\langle \top \rangle\}$. 
The sequents in this case are $\langle \top \rangle , \langle \top \rangle \Rightarrow \top$ and $\langle \top \rangle, \top \Rightarrow \langle \top \rangle$, both provable in $\mathbf{CK}$. Moreover, $ \bigwedge \Phi_{A, \overline{\phi}}^{s}$ is equal to $\top$ and hence $(\, \Rightarrow \bigwedge \Phi_{A, \overline{\phi}}^{s})$ is an instance of an axiom an provable in $\mathbf{CK}$. Define the proof $\sigma_{\Phi, A, \overline{\phi}}$ as this axiom.\\
For $A(\overline{p})= B(\overline{p}) \wedge C(\overline{p})$, by Definition \ref{Translation},  $(A(\overline{\phi}))^t=(B(\overline{\phi}))^t \wedge (C(\overline{\phi}))^t \wedge \langle B(\overline{\phi}) \wedge C(\overline{\phi}) \rangle$. By recursion, there are multisets $\Phi_{B, \overline{\phi}}$ and $\Phi_{C, \overline{\phi}}$ such that
\begin{center}
    $\Phi_{B, \overline{\phi}}, (B(\overline{\phi}))^t \Rightarrow B(\overline{\phi^t}) \quad (1) \quad , \quad \Phi_{B, \overline{\phi}}, B(\overline{\phi^t}) \Rightarrow (B(\overline{\phi}))^t \quad (2),$
\end{center}
\begin{center}
    $\Phi_{C, \overline{\phi}}, (C(\overline{\phi}))^t \Rightarrow C(\overline{\phi^t}) \quad (3) \quad , \quad \Phi_{C, \overline{\phi}}, C(\overline{\phi^t}) \Rightarrow (C(\overline{\phi}))^t \quad (4),$
\end{center}
hold in $\mathbf{CK}$. Define $F_{\Phi, A, \overline{\phi}}=\big(\langle B(\overline{\phi}) \rangle \wedge \langle C(\overline{\phi}) \rangle\big) \to \langle B(\overline{\phi})  \wedge  C(\overline{\phi}) \rangle$ and 
\begin{center}
   $\Phi_{A, \overline{\phi}}=\Phi_{B, \overline{\phi}} \cup \Phi_{C, \overline{\phi}} \cup \{F_{\Phi, A, \overline{\phi}}\}$
\end{center}
Note that $\Phi_{A, \overline{\phi}}$ only consists of modal Horn formulas constructed from angled atoms. Now,
let us first prove the trickier sequent, namely $\Phi_{A, \overline{\phi}}, A(\overline{\phi^t}) \Rightarrow (A(\overline{\phi}))^t$. Using the rules in $\mathbf{CK}$, we easily get from $(2)$ and $(4)$
\begin{center}
    $\mathbf{CK} \vdash \Phi_{B, \overline{\phi}}, \Phi_{C, \overline{\phi}}, B(\overline{\phi^t}) \wedge C(\overline{\phi^t}) \Rightarrow (B(\overline{\phi}))^t \wedge (C(\overline{\phi}))^t \quad (5).$
\end{center}
By Lemma \ref{TranslationAndAtoms}, both $(B(\overline{\phi}))^t \Rightarrow \langle B(\overline{\phi}) \rangle$ and $(C(\overline{\phi}))^t \Rightarrow \langle C(\overline{\phi}) \rangle$ are 
provable in $\mathbf{CK}$, hence so is 
\begin{center}
    $(B(\overline{\phi}))^t \wedge (C(\overline{\phi}))^t \Rightarrow \langle B(\overline{\phi}) \rangle \wedge \langle C(\overline{\phi}) \rangle \quad (6).$
\end{center}
Applications of the cut rule on the $\mathbf{CK}$-provable sequent
\begin{center}
    $\big(\langle B(\overline{\phi}) \rangle \wedge \langle C(\overline{\phi}) \rangle\big) \to \langle B(\overline{\phi})  \wedge  C(\overline{\phi}) \rangle , \langle B(\overline{\phi}) \rangle \wedge \langle C(\overline{\phi}) \rangle \Rightarrow \langle B(\overline{\phi})  \wedge  C(\overline{\phi}) \rangle$
\end{center}
and $(5)$ and $(6)$, we get 
\begin{center}
    $\mathbf{CK} \vdash \Phi_{A, \overline{\phi}}, A(\overline{\phi^t}) \Rightarrow \langle B(\overline{\phi})  \wedge  C(\overline{\phi}) \rangle.$
\end{center}
Moreover, using $(Lw)$ on $(5)$ we get \begin{center}
    $\mathbf{CK} \vdash \Phi_{A, \overline{\phi}}, A(\overline{\phi^t}) \Rightarrow (B(\overline{\phi}))^t \wedge (C(\overline{\phi}))^t$
\end{center}
which together with the above sequent we finally obtain
\begin{center}
    $\mathbf{CK} \vdash \Phi_{A, \overline{\phi}}, A(\overline{\phi^t}) \Rightarrow (A(\overline{\phi}))^t.$
\end{center}
The other sequent, i.e., $\Phi_{A, \overline{\phi}}, (A(\overline{\phi}))^t \Rightarrow A(\overline{\phi^t})$, is easier. Using the rules in $\mathbf{CK}$, we easily get from $(1)$ and $(3)$
\begin{center}
    $\mathbf{CK} \vdash \Phi_{B, \overline{\phi}}, \Phi_{C, \overline{\phi}}, (B(\overline{\phi}))^t \wedge (C(\overline{\phi}))^t \Rightarrow B(\overline{\phi^t}) \wedge C(\overline{\phi^t}).$
\end{center}
Then, using the weakening and the rule $(L\wedge)$, we have 
\begin{center}
    $\mathbf{CK} \vdash \Phi_{A, \overline{\phi}}, (A(\overline{\phi}))^t \Rightarrow A(\overline{\phi^t}).$
\end{center}
Finally, to provide $\sigma_{\Phi, A, \overline{\phi}}$, first note that by the axiom $(id)$ and then applying the rule $(R\to)$ in $\mathbf{CK}$, we have $\mathbf{CK} \vdash \; \Rightarrow (F_{\Phi, A, \overline{\phi}})^s$. Call this proof $\rho_{\Phi, A, \overline{\phi}}$. We have already the $\mathbf{CK}$-proofs $\sigma_{\Phi, B, \overline{\phi}}$ and $\sigma_{\Phi, C, \overline{\phi}}$ for $(\, \Rightarrow \bigwedge \Phi_{B, \overline{\phi}}^{s})$ and $(\, \Rightarrow \bigwedge \Phi_{C, \overline{\phi}}^{s})$, respectively. Therefore, together with $\rho_{\Phi, A, \overline{\phi}}$ and some applications of $(R\wedge)$, they form a proof $\sigma_{\Phi, A, \overline{\phi}}$ for $\mathbf{CK} \vdash (\, \Rightarrow \bigwedge \Phi_{A, \overline{\phi}}^{s})$. \\
Similarly, we can prove that in the case $A(\overline{p}) = B(\overline{p}) \vee C(\overline{p})$, setting
\begin{center}
$\Phi_{A, \overline{\phi}}=\Phi_{B, \overline{\phi}} \cup \Phi_{C, \overline{\phi}} \cup \{F_{\Phi, A, \overline{\phi}},  F'_{\Phi, A, \overline{\phi}}\}$
\end{center}
works, where $F_{\Phi, A, \overline{\phi}}=\langle B(\overline{\phi}) \rangle  \to \langle B(\overline{\phi})  \vee C(\overline{\phi}) \rangle$ and  $F'_{\Phi, A, \overline{\phi}}=\langle C(\overline{\phi}) \rangle \to \langle B(\overline{\phi})  \vee  C(\overline{\phi}) \rangle$. Finding $\sigma_{\Phi, A, \overline{\phi}}$ is also similar.\\
For the case $A(\overline{p})= \Diamond B(\overline{p})$, by Definition \ref{TranslationAndAtoms}, we have $(A(\overline{\phi}))^t=\Diamond (B(\overline{\phi}))^t \wedge \langle \Diamond B(\overline{\phi}) \rangle$. We already know that
\begin{center}
    $\Phi_{B, \overline{\phi}}, (B(\overline{\phi}))^t \Rightarrow B(\overline{\phi^t}) \quad (7) \quad , \quad \Phi_{B, \overline{\phi}}, B(\overline{\phi^t}) \Rightarrow (B(\overline{\phi}))^t \quad (8)$
\end{center}
are provable in $\mathbf{CK}$. Define
\begin{center}
    $\Phi_{A, \overline{\phi}} = \Box \Phi_{B, \overline{\phi}} \cup \{F_{\Phi, A, \overline{\phi}}\}$,
\end{center}
where $F_{\Phi, A, \overline{\phi}}=\Diamond \langle B(\overline{\phi}) \rangle \to \langle \Diamond B(\overline{\phi}) \rangle$. Note that $\Phi_{A, \overline{\phi}}$ consists of modal Horn formulas constructed from angled atoms. Let us investigate the more complicated case, namely the provability of the sequent $\Phi_{A, \overline{\phi}}, A(\overline{\phi^t}) \Rightarrow (A(\overline{\phi}))^t$ in $\mathbf{CK}$. The other case is easier. Now, applying the rule $(K_{\Diamond})$ on $(8)$, we get 
\begin{center}
    $\mathbf{CK} \vdash \Box \Phi_{B, \overline{\phi}}, \Diamond B(\overline{\phi^t}) \Rightarrow \Diamond (B(\overline{\phi}))^t \quad (9).$
\end{center}
By Lemma \ref{TranslationAndAtoms}, we have $\mathbf{CK} \vdash (B(\overline{\phi}))^t \Rightarrow \langle B(\overline{\phi}) \rangle$ and by $(K_{\Diamond})$, we get $\mathbf{CK} \vdash \Diamond (B(\overline{\phi}))^t \Rightarrow \Diamond \langle B(\overline{\phi}) \rangle$. Therefore, using cut and $(9)$, we get $\mathbf{CK} \vdash \Box \Phi_{B, \overline{\phi}}, \Diamond B(\overline{\phi^t}) \Rightarrow \Diamond \langle B(\overline{\phi}) \rangle$. By cut on the latter sequent and the $\mathbf{CK}$-provable sequent $\Diamond \langle B(\overline{\phi}) \rangle, \Diamond \langle B(\overline{\phi}) \rangle \to \langle \Diamond B(\overline{\phi}) \rangle \Rightarrow \langle \Diamond B(\overline{\phi}) \rangle$, we get
\begin{center}
    $\mathbf{CK} \vdash \Box \Phi_{B, \overline{\phi}}, \Diamond B(\overline{\phi^t}) , \Diamond \langle B(\overline{\phi}) \rangle \to \langle \Diamond B(\overline{\phi}) \rangle\Rightarrow \langle \Diamond B(\overline{\phi}) \rangle \quad (10).$
\end{center}
Using the rule $(L w)$ on $(9)$ and then applying the rule $(R \wedge)$ on the resulting sequent and $(10)$, we get $\Phi_{A, \overline{\phi}}, A(\overline{\phi^t}) \Rightarrow (A(\overline{\phi}))^t$ in $\mathbf{CK}$. \\
To provide $\sigma_{\Phi, A, \overline{\phi}}$, note that by the axiom $(id)$ and then applying $(R\to)$, we have $\mathbf{CK} \vdash \; \Rightarrow (F_{\Phi, A, \overline{\phi}})^s$.
Call this proof $\rho_{\Phi, A, \overline{\phi}}$. Applying $(K_{\Box})$ on the already existing proof $\mathbf{CK} \vdash^{\sigma_{\Phi, B, \overline{\phi}}} \; \Rightarrow \bigwedge (\Phi_{B, \overline{\phi}})^s$, we get $\mathbf{CK} \vdash \; \Rightarrow \Box \bigwedge (\Phi_{B, \overline{\phi}})^s$. To prove $\mathbf{CK} \vdash \; \Rightarrow \bigwedge (\Box \Phi_{B, \overline{\phi}})^s$, we will provide a $\mathbf{CK}$-proof for $\Box \bigwedge \Gamma \Rightarrow \bigwedge \Box \Gamma$ and investigate its complexity, for any multiset $\Gamma$. First, notice that by the axiom $(id)$, we have $\mathbf{CK} \vdash \gamma \Rightarrow \gamma$, for any $\gamma \in \Gamma$. Applying the rule $(L \wedge)$ for $\parallel \Gamma \parallel-1$ many times, we get $\mathbf{CK} \vdash \bigwedge \Gamma \Rightarrow \gamma$, where $\parallel \Gamma \parallel$ is the cardinality of $\Gamma$. Applying the rule $(K \Box)$, we get $\mathbf{CK} \vdash \Box \bigwedge \Gamma \Rightarrow \Box \gamma$, for any $\gamma \in \Gamma$. Applying the rule $(R \wedge)$ for $\parallel \Gamma \parallel-1$ many times, we finally get $\mathbf{CK} \vdash \Box \bigwedge \Gamma \Rightarrow \bigwedge \Box \Gamma$.
Notice that producing the whole proof takes $|\Gamma|^{O(1)}$ many steps.
Using this proof for $\Gamma=(\Phi_{B, \overline{\phi}})^s$, we get a $\mathbf{CK}$-proof for $\Box \bigwedge (\Phi_{B, \overline{\phi}})^s \Rightarrow \bigwedge \Box (\Phi_{B, \overline{\phi}})^s$, in time $|\Phi_{B, \overline{\phi}}|^{O(1)}$. Then, using the proof $\rho_{\Phi, A, \overline{\phi}}$, we can easily construct a proof $\sigma_{\Phi, A, \overline{\phi}}$ for $\mathbf{CK} \vdash \; \Rightarrow \bigwedge (\Phi_{A, \overline{\phi}})^s$.

For $(ii)$, again, we use recursion on the structure of the almost positive formula $A(\overline{p})$. The base case, where $A(\overline{p})$ is a basic formula, is covered in $(i)$. The cases $A(\overline{p})=B(\overline{p}) \circ C(\overline{p})$ or $A(\overline{p})=\bigcirc B(\overline{p})$, where $\circ \in \{\wedge, \vee\}$ and $\bigcirc \in \{\Box, \Diamond\}$ are simple and similar to the cases in $(i)$. It is easy to see that in the former cases $\Pi_{A, \overline{\phi}}= \Pi_{B, \overline{\phi}} \cup \Pi_{C, \overline{\phi}}$ and in the latter cases $\Pi_{A, \overline{\phi}}= \Box \Pi_{B, \overline{\phi}}$ works. The structure of $\sigma_{\Pi, A, \overline{\phi}}$ is similar to that of the case $(i)$. The only remaining case, which is also simple, is when $A(\overline{p}) = B(\overline{p}) \to C(\overline{p})$, where $B(\overline{p})$ is a basic formula and $C(\overline{p})$ is almost positive. Here again $\Pi_{A, \overline{\phi}}= \Phi_{B, \overline{\phi}} \cup \Pi_{C, \overline{\phi}}$ works. It is clear that $\Pi_{A, \overline{\phi}}$ is a multiset of modal Horn formulas constructed from angled atoms. Using Definition \ref{Translation}, we have $(A(\overline{\phi}))^t = [(B(\overline{\phi}))^t \to (C(\overline{\phi}))^t] \wedge \langle B(\overline{\phi}) \to C(\overline{\phi}) \rangle$. By $(i)$ and the recursive step, we have the multisets $\Phi_{B, \overline{\phi}}$ and $\Pi_{C, \overline{\phi}}$ such that
\begin{center}
    $\Phi_{B, \overline{\phi}}, B(\overline{\phi^t}) \Rightarrow (B(\overline{\phi}))^t \quad (11) \quad, \quad
\Pi_{C, \overline{\phi}}, (C(\overline{\phi}))^t \Rightarrow C(\overline{\phi^t}) \quad (12)$
\end{center}
are provable in $\mathbf{CK}$. Using the rule $(L w)$ and then $(L \to)$ on $(11)$ and $(12)$ and then the rule $(R \to)$, we get
\begin{center}
    $\mathbf{CK} \vdash \Phi_{B, \overline{\phi}}, \Pi_{C, \overline{\phi}}, (B(\overline{\phi}))^t \to (C(\overline{\phi}))^t \Rightarrow B(\overline{\phi^t}) \to C(\overline{\phi^t}).$
\end{center}
Now, using the rule $(L \wedge_1)$ to introduce $\langle B(\overline{\phi}) \to C(\overline{\phi}) \rangle$ in the antecedent of the sequent, and setting $\Pi_{A, \overline{\phi}}= \Phi_{B, \overline{\phi}} \cup \Pi_{C, \overline{\phi}}$, we get $\mathbf{CK} \vdash \Pi_{A, \overline{\phi}}, (A(\overline{\phi}))^t \Rightarrow A(\overline{\phi^t})$. Finally, similar to the case $(i)$, it is easy to use $\sigma_{\Phi, B, \overline{\phi}}$ and $\sigma_{\Pi, C, \overline{\phi}}$ to construct $\sigma_{\Pi, A, \overline{\phi}}$ such that $\mathbf{CK} \vdash^{\sigma_{\Pi, A, \overline{\phi}}} \; \Rightarrow \bigwedge \Theta_{A, \overline{\phi}}^{s}$.

For $(iii)$, the proof again proceeds by recursion on the structure of the constructive formula $A(\overline{p})$. The base case is covered in $(i)$. For the case $A(\overline{p})= B(\overline{p}) \wedge C(\overline{p})$, set $\Upsilon_{A, \overline{\phi}}$ as
\begin{center}
    $\Upsilon_{B, \overline{\phi}} \cup \Upsilon_{C, \overline{\phi}} \cup
\{F_{\Upsilon, A, \overline{\phi}}, F'_{\Upsilon, A, \overline{\phi}}\},$
\end{center}
where $F_{\Upsilon, A, \overline{\phi}}=\langle B(\overline{\phi}) \wedge C(\overline{\phi}) \rangle \to \langle B(\overline{\phi}) \rangle$ and $F'_{\Upsilon, A, \overline{\phi}}=\langle B(\overline{\phi}) \wedge C(\overline{\phi}) \rangle \to \langle C(\overline{\phi}) \rangle$, and for the case that $A(\overline{p})=\Box B(\overline{p})$ take
\begin{center}
    $\Upsilon_{A, \overline{\phi}}= \Box \Upsilon_{B, \overline{\phi}} \cup \{ F_{\Upsilon, A, \overline{\phi}} \},$
\end{center}
where $F_{\Upsilon, A, \overline{\phi}}=\langle \Box B (\overline{\phi}) \rangle \to  \Box \langle B(\overline{\phi}) \rangle$.
It is easy to see that in both cases $\Upsilon_{A, \overline{\phi}}$ works, it is a multiset of modal Horn formulas constructed from angled atoms, and $\mathbf{CK} \vdash (\, \Rightarrow \bigwedge \Theta_{A, \overline{\phi}}^{s})$ by a proof $\sigma_{\Upsilon, A, \overline{\phi}}$, constructed in a similar fashion as in $(i)$. The only case left is $A(\overline{p})=B(\overline{p}) \to C(\overline{p})$, where $B(\overline{p})$ is almost positive and $C(\overline{p})$ is constructive.  By Definition \ref{TranslationAndAtoms}, we have $(A(\overline{\phi}))^t=[(B(\overline{\phi}))^t \to (C(\overline{\phi}))^t] \wedge \langle B(\overline{\phi}) \to C(\overline{\phi}) \rangle$. By $(ii)$ and the recursive step, we have the multisets $\Pi_{B, \overline{\phi}}$ and $\Upsilon_{C, \overline{\phi}}$ such that
\begin{center}
    $\Pi_{B, \overline{\phi}}, (B(\overline{\phi}))^t \Rightarrow B(\overline{\phi^t}) \quad (13) \quad , \quad \Upsilon_{C, \overline{\phi}}, \langle C(\overline{\phi}) \rangle, C(\overline{\phi^t}) \Rightarrow (C(\overline{\phi}))^t \quad (14)$
\end{center}
are provable in $\mathbf{CK}$. We claim taking
\begin{center}
    $\Upsilon_{A, \overline{\phi}}=\Pi_{B, \overline{\phi}} \cup \Upsilon_{C, \overline{\phi}} \cup \{F_{\Upsilon, A, \overline{\phi}}\}$
\end{center}
where $F_{\Upsilon, A, \overline{\phi}}=\big(\langle B(\overline{\phi}) \to C(\overline{\phi}) \rangle \wedge \langle B(\overline{\phi}) \rangle\big) \to \langle C(\overline{\phi}) \rangle$ works. Applying the rule $(L w)$ and then $(L \to)$ on $(13)$ and $(14)$, we have 
\begin{center}
    $\mathbf{CK} \vdash \Pi_{B, \overline{\phi}}, \Upsilon_{C, \overline{\phi}}, (B(\overline{\phi}))^t  , \langle C(\overline{\phi}) \rangle, B(\overline{\phi^t}) \to C(\overline{\phi^t}) \Rightarrow (C(\overline{\phi}))^t.$
\end{center}
Using the cut rule on the above sequent and $(B(\overline{\phi}))^t, (B(\overline{\phi}))^t \to \langle C(\overline{\phi}) \rangle \Rightarrow \langle C(\overline{\phi}) \rangle $ and then $(L c)$ and $(R \to)$, we get
\begin{center}
    $\Pi_{B, \overline{\phi}}, \Upsilon_{C, \overline{\phi}}, (B(\overline{\phi}))^t \to \langle C(\overline{\phi}) \rangle, B(\overline{\phi^t}) \to C(\overline{\phi^t}) \Rightarrow (B(\overline{\phi}))^t  \to (C(\overline{\phi}))^t \;\; (15)$
\end{center}
is provable in $\mathbf{CK}$.
Using the cut rule on the following $\mathbf{CK}$-provable sequents
\begin{center}
    $\langle A(\overline{\phi}) \rangle, \langle A(\overline{\phi}) \rangle \wedge \langle B(\overline{\phi}) \rangle \to \langle C(\overline{\phi}) \rangle \Rightarrow  \langle B(\overline{\phi}) \rangle \to \langle C(\overline{\phi}) \rangle \quad ,$ \par $\langle B(\overline{\phi}) \rangle \to \langle C(\overline{\phi}) \rangle \Rightarrow (B(\overline{\phi}))^t \to \langle C(\overline{\phi}) \rangle$
\end{center}
we get
\begin{center}
    $\mathbf{CK} \vdash \langle A(\overline{\phi}) \rangle, \langle A(\overline{\phi}) \rangle \wedge \langle B(\overline{\phi}) \rangle \to \langle C(\overline{\phi}) \rangle \Rightarrow  (B(\overline{\phi}))^t \to \langle C(\overline{\phi}) \rangle$.
\end{center}
Using the cut rule on the above sequent and $(15)$, we obtain
\begin{center}
    $\mathbf{CK} \vdash \Pi_{B, \overline{\phi}}, \Upsilon_{C, \overline{\phi}}, \langle A(\overline{\phi}) \rangle, \langle A(\overline{\phi}) \rangle \wedge \langle B(\overline{\phi}) \rangle \to \langle C(\overline{\phi}) \rangle, B(\overline{\phi^t}) \to C(\overline{\phi^t}) \Rightarrow  (B(\overline{\phi}))^t \to (C(\overline{\phi}))^t.$
\end{center}
Using the left weakening rule on $\langle A(\overline{\phi}) \rangle \Rightarrow \langle A(\overline{\phi}) \rangle$ we get
\begin{center}
    $\mathbf{CK} \vdash \Upsilon_{A, \overline{\phi}}, \langle A(\overline{\phi}) \rangle, A(\overline{\phi^t}) \Rightarrow  \langle A(\overline{\phi}) \rangle.$
\end{center}
Applying the rule $(R \wedge)$ on the above two sequent, we get
\begin{center}
    $\mathbf{CK} \vdash \Upsilon_{A, \overline{\phi}}, \langle A(\overline{\phi}) \rangle, A(\overline{\phi^t}) \Rightarrow (A(\overline{\phi}))^t,$
\end{center}
as required. Again, it is clear that $\Upsilon_{A, \overline{\phi}}$ is a multiset of modal Horn formulas constructed from angled atoms and $\mathbf{CK} \vdash^{\sigma_{\Upsilon, A, \overline{\phi}}} (\, \Rightarrow \bigwedge \Upsilon_{A, \overline{\phi}}^{s})$, where the structure of $\sigma_{\Upsilon, A, \overline{\phi}}$ is similar to that of the case $(i)$.

The only issue remained to investigate is the feasibility of the algorithms for $\Theta_{A, \overline{\phi}}$ and $\sigma_{\Theta, A, \overline{\phi}}$. From now on, for simplicity, we use $A$ and $|A|$, when we want to refer to $A(\overline{p})$ and $|A(\overline{p})|$. For $\Theta_{A, \overline{\phi}}$, we use the above algorithm that reads $A$ and $\overline{\phi}$ and computes $\Theta_{A, \overline{\phi}}$, by recursion on the structure of $A$. Let $\Theta \in \{\Phi, \Pi, \Upsilon \}$, $\circ \in \{\wedge, \vee\}$, and $\bigcirc \in \{\Box, \Diamond\}$. First, we need an upper bound on $\parallel \Theta_{A, \overline{\phi}} \parallel$ and $|\Theta_{A, \overline{\phi}}|$. For the former, by a simple induction on $A$, observe that $\parallel \Theta_{A, \overline{\phi}} \parallel \leq O(|A|)$. For the latter, we have the following inequalities:
\begin{enumerate}
\item
$|\Theta_{A, \overline{\phi}}| \leq O(1)$, where $A(\overline{p})$ is an atom, $\bot$, or $\top$,
\item
$ |\Theta_{B \circ C, \overline{\phi}}| \leq |\Theta_{B, \overline{\phi}}| + |\Theta_{C, \overline{\phi}}| + O(|A||\overline{\phi}|),$ except when $\{\Theta=\Upsilon, \circ=\vee\}$,
\item
$|\Pi_{B \to C, \overline{\phi}}| \leq |\Phi_{B, \overline{\phi}}| + |\Pi_{C, \overline{\phi}}| + O(|A||\overline{\phi}|),$ 
\item
$|\Upsilon_{B \to C, \overline{\phi}}| \leq |\Pi_{B, \overline{\phi}}| + |\Upsilon_{C, \overline{\phi}}| + O(|A||\overline{\phi}|),$ 
\item
$|\Theta_{\bigcirc B, \overline{\phi}}| \leq |\Theta_{B, \overline{\phi}}| + O(|A||\overline{\phi}|),$ except when $\{\Theta=\Phi, \bigcirc=\Box\}$ or $\{\Theta=\Upsilon, \bigcirc=\Diamond\}$.
\end{enumerate}
First, note that the cases which are excluded in $2$ and $5$, are the cases where the theorem does not apply to. For instance, in $5$, we have excluded the case where $\Theta=\Phi$ and $\bigcirc=\Box$. The reason is that the multiset $\Phi_{A, \overline{\phi}}$ corresponds to the case where $A$ is a basic formula, and by Definition \ref{DefPositiveFormulas}, $A$ cannot be of the form $\Box B$. Similarly for the other excluded cases.

Now, to justify the inequalities, based on how the multisets were constructed, the bounds are easy to compute. In each case, $\Theta_{A, \overline{\phi}}$ is the union or the box of $\Phi$, $\Pi$ or $\Upsilon$ of the immediate subformulas of $A$ and the formulas $F_{\Theta, A, \overline{\phi}}$ and $F'_{\Theta, A, \overline{\phi}}$. As the addend $O(|A||\overline{\phi}|)$ represents an upper bound for $|F_{\Theta, A, \overline{\phi}}|$ or $|F_{\Theta, A, \overline{\phi}}| + |F'_{\Theta, A, \overline{\phi}}|$, depending on the case, the bounds are trivially in place. Now, using a simple induction on the structure of $A$ (first starting with basic formulas, then almost positive and finally constructive formulas) and the above inequalities, we can show that $|\Theta_{A, \overline{\phi}}| \leq (|A|+ |\overline{\phi}|)^{O(1)}$.

Having the upper bounds 
on $\parallel \Theta_{A, \overline{\phi}} \parallel$ and $|\Theta_{A, \overline{\phi}}|$
established, we are now ready to address the feasibility of the computation of $ \Theta_{A, \overline{\phi}}$. Denote the time that the algorithm takes to compute $\Theta_{A, \overline{\phi}}$ by $T_{\Theta}(A, \overline{\phi})$. We have:
\begin{enumerate}
\item
$T_{\Theta}(A, \overline{\phi}) \leq O(1)$, where $A(\overline{p})$ is an atom, $\bot$, or $\top$,
\item
$
 T_{\Theta}(B \circ C, \overline{\phi}) \leq T_{\Theta}(B, \overline{\phi}) + T_{\Theta}(C, \overline{\phi})+ O(|A||\overline{\phi}|),
$
except when $\{\Theta=\Upsilon, \circ=\vee\}$,
\item
$
 T_{\Pi}(B \to C, \overline{\phi}) \leq T_{\Phi}(B, \overline{\phi}) + T_{\Pi}(C, \overline{\phi})+ O(|A||\overline{\phi}|),
$
\item
$
 T_{\Upsilon}(B \to C, \overline{\phi}) \leq T_{\Pi}(B, \overline{\phi}) + T_{\Upsilon}(C, \overline{\phi})+ O(|A||\overline{\phi}|), 
$
\item
$
 T_{\Theta}(\bigcirc B, \overline{\phi}) \leq T_{\Theta}(B, \overline{\phi})+O(|\Theta_{B, \overline{\phi}}|+\parallel \Theta_{B, \overline{\phi}} \parallel) + O(|A||\overline{\phi}|), 
$
except when $\{\Theta=\Phi, \bigcirc=\Box\}$ or $\{\Theta=\Upsilon, \bigcirc=\Diamond\}$.
\end{enumerate}
It is easy to see why these inequalities hold, based on how the multisets were constructed. The reason simply is that in each case, we must first compute the appropriate multiset among $\Phi$, $\Pi$ or $\Upsilon$ of the immediate subformulas of $A$, and then possibly the formulas $F_{\Theta, A, \overline{\phi}}$ and $F'_{\Theta, A, \overline{\phi}}$. 
Note that in the modal cases (the last inequality), the algorithm also needs to add boxes to $\Theta_{A, \overline{\phi}}$ that takes
$O(|\Theta_{B, \overline{\phi}}|+\parallel \Theta_{B, \overline{\phi}} \parallel)$ steps. As the addend $O(|A||\overline{\phi}|)$ represents an upper bound on the time to compute $F_{\Theta, A, \overline{\phi}}$ and $F'_{\Theta, A, \overline{\phi}}$, the bounds are trivially in place.
Now, using a simple induction on the structure of $A$ (first starting with the basic formulas, then the almost positive and finally the constructive formulas), by the above inequalities together with the fact that $\parallel \Theta_{A, \overline{\phi}} \parallel \leq O(|A|)$ and $|\Theta_{A, \overline{\phi}}| \leq (|A|+ |\overline{\phi}|)^{O(1)}$, we can show that $T_\Theta(A, \overline{\phi}) \leq (|A|+ |\overline{\phi}|)^{O(1)}$. 

Similarly, we follow the above algorithm to compute $\sigma_{\Theta, A, \overline{\phi}}$, where the time of the algorithm is denoted by $T_{\sigma, \Theta}(A, \overline{\phi})$. We have the inequalities:
\begin{itemize}
\item[$1'.$]
 $T_{\sigma, \Theta}(A, \overline{\phi}) \leq  O(1),$ where $A(\overline{p})$ is an atom, $\bot$, or $\top$,
\item[$2'.$]
 $T_{\sigma, \Theta}(B \circ C, \overline{\phi}) \leq T_{\sigma, \Theta}(B, \overline{\phi}) + T_{\sigma, \Theta}(C, \overline{\phi}) + T_{\rho, \Theta}(A, \overline{\phi}) + O(|\Theta_{B, \overline{\phi}}| + |\Theta_{C, \overline{\phi}}| + |A||\overline{\phi}|),$ except when $\{\Theta=\Upsilon, \circ=\vee\}$,
 \item[$3'.$]
 $T_{\sigma, \Pi}(B \to C, \overline{\phi}) \leq T_{\sigma, \Phi}(B, \overline{\phi}) + T_{\sigma, \Pi}(C, \overline{\phi}) + T_{\rho, \Pi}(A, \overline{\phi}) + O(|\Phi_{B, \overline{\phi}}| + |\Pi_{C, \overline{\phi}}| + |A||\overline{\phi}|),$
 \item[$4'.$]
 $T_{\sigma, \Upsilon}(B \to C, \overline{\phi}) \leq T_{\sigma, \Pi}(B, \overline{\phi}) + T_{\sigma, \Upsilon}(C, \overline{\phi}) + T_{\rho, \Upsilon}(A, \overline{\phi}) + O(|\Pi_{B, \overline{\phi}}| + |\Upsilon_{C, \overline{\phi}}| + |A||\overline{\phi}|),$
\item[$5'.$]
$T_{\sigma, \Theta}(\bigcirc B, \overline{\phi}) \leq T_{\sigma, \Theta}(B, \overline{\phi}) + T_{\rho, \Theta}(A, \overline{\phi}) + (|\Theta_{B, \overline{\phi}}|)^{O(1)}+O(|A||\overline{\phi}|),$\\ except when $\{\Theta=\Phi, \bigcirc=\Box\}$ or $\{\Theta=\Upsilon, \bigcirc=\Diamond\}$.
\end{itemize}
where $T_{\rho, \Theta}(A, \phi)$ is the time to compute $\rho_{\Theta, A, \overline{\phi}}$. To show why, as $1'$ is trivial, we split the bounds $2'-5'$ into two families: the propositional and the modal cases. For the propositional cases, $2'-4'$, the proof $\sigma_{\Theta, A, \overline{\phi}}$ is the combination of the corresponding proofs for $\Phi$, $\Pi$ or $\Upsilon$ of the immediate subformulas of $A$, the proofs of the sequents $(\, \Rightarrow (F_{\Theta, A, \overline{\phi}})^{s})$ and $(\, \Rightarrow (F'_{\Theta, A, \overline{\phi}})^{s})$, denoted by $\rho_{\Theta, A, \overline{\phi}}$ throughout the construction, and finally some constant number of the applications of the rule $(R \wedge)$. Note that the last part expands the time of the computation by a constant number of the sum of the sizes of $\Phi$, $\Pi$ or $\Upsilon$ for the immediate subformulas and the size of $F_{\Theta, A, \overline{\phi}}$ and $F'_{\Theta, A, \overline{\phi}}$.
For instance, in the case $4'$, the  proof $\sigma_{\Upsilon, B \to C, \overline{\phi}}$ looks like:
\begin{center}
\begin{tabular}{c}
\small 
\AxiomC{$\sigma_{\Pi, B, \overline{\phi}}$}
\noLine
\small \UnaryInfC{$\Rightarrow \bigwedge \Pi^s_{B, \overline{\phi}}$}
\AxiomC{$\sigma_{\Upsilon, C, \overline{\phi}}$}
\noLine
\small \UnaryInfC{$\Rightarrow \bigwedge \Upsilon^s_{C, \overline{\phi}}$}
\small \BinaryInfC{$\Rightarrow \bigwedge \Pi^s_{B, \overline{\phi}} \wedge \bigwedge \Upsilon^s_{C, \overline{\phi}}$}
\AxiomC{$\rho_{\Upsilon, A, \overline{\phi}}$}
\noLine
\small \UnaryInfC{$\Rightarrow (F_{\Upsilon, A, \overline{\phi}})^s$}
\small \BinaryInfC{$\Rightarrow \bigwedge \Pi^s_{B, \overline{\phi}} \wedge \bigwedge \Upsilon^s_{C, \overline{\phi}} \wedge (F_{\Upsilon, A, \overline{\phi}})^s$}
 \DisplayProof
 \end{tabular}
\end{center}
For the modal case, i.e., $5'$, we start with $\sigma_{\Theta, B, \overline{\phi}}$ to which we apply the rule $(K_{\Box})$. Then, we prove $\Box \bigwedge \Theta_{B, \overline{\phi}}^s \Rightarrow \bigwedge \Box (\Theta_{B, \overline{\phi}})^s$ and finally we add the proof $\rho_{\Theta, A, \overline{\phi}}$ for $F_{\Theta, A, \overline{\phi}}$ together with some constant number of the applications of $(R\wedge)$. The addend $(|\Theta_{B, \overline{\phi}}|)^{O(1)}$ in $5'$ is a bound for the time of the box distribution part, while $O(|A||\overline{\phi}|)$ is a bound for the size of $F_{\Theta, A, \overline{\phi}}$.\\
Finally, having the inequalities established, by using the inequalities $|\Theta_{A, \overline{\phi}}| \leq (|A|+|\overline{\phi}|)^{O(1)}$ and $T_{\rho, \Theta}(A, \overline{\phi}) \leq (|A|+|\overline{\phi}|)^{O(1)}$, it is easy to use an induction on the structure of $A$ to prove $T_{\sigma, \Theta}(A, \overline{\phi}) \leq (|A|+|\overline{\phi}|)^{O(1)} $.
\end{proof}

\begin{rem}
Here are two remarks. First, note that Theorem \ref{Commutation} holds for \emph{any} multiset of formulas $\overline{\phi}$, as long as the formula $A$ has the described structure. Another point to make is that in the proof of Theorem \ref{Commutation}, other (sometimes simpler) choices exist for the set of modal Horn formulas such that it makes the translated sequent provable in $\mathbf{CK}$. The crucial point of our choices for these sets of modal Horn formulas is the condition that the standard translation of each of their elements are provable in $\mathbf{CK}$.
\end{rem}

The last part of this subsection is devoted to investigate the relationship between the translation $t$ and the Harrop formulas as defined below.

\begin{dfn} \label{DfnHarrop}
The set of \emph{Harrop} formulas in the language $\mathcal{L}$ is the smallest set of formulas containing the atoms in $\mathcal{L}$ and $\bot$, $\top$, and is closed under $\wedge, \Box$, and implications of the form $A \to B$, where $A$ is an arbitrary formula and $B$ a Harrop formula.
A formula in the language $\mathcal{L}_{\Box}$, $\mathcal{L}_{\Diamond}$ or $\mathcal{L}_p$ is called Harrop, if it is Harrop as a formula in the extended language $\mathcal{L}$.
\end{dfn}


\begin{lem} \label{LemHarrop}
There is a feasible algorithm that reads a Harrop formula $A \in \mathcal{L}$ and outputs 
a multiset $\Gamma_A$ and a proof $\sigma_A$ such that:
\begin{enumerate}
\item[$(i)$]
$\Gamma_A$ consists of modal Horn formulas, constructed only from $\bot$, and angled atoms, 
\item[$(ii)$]
$\mathbf{CK} \vdash \Gamma_A \Rightarrow A^t$, and
\item[$(iii)$]
$\mathbf{CK} \vdash^{\sigma_A} \bigwedge \Gamma_A^s \Leftrightarrow A$.
\end{enumerate}
\end{lem}

\begin{proof}
We first explain the algorithm to compute $\Gamma_A$ and $\sigma_A$. The feasibility part will be explained afterwards. To construct $\Gamma_A$ and $\sigma_A$, we use recursion on the structure of $A$.
If $A$ is atomic, $\bot$, or $\top$, then it is easy to see that $\Gamma_{A}= \{A^t\}$ satisfies the Conditions $(i)$ and $(ii)$. Moreover, notice that $\bigwedge \Gamma_A^s \Leftrightarrow A$ is an instance of the axiom $(id)$ in $\mathbf{CK}$. Therefore, it is enough to define $\sigma_A$ as that instance. If $A=B \wedge C$, where $B$ and $C$ are Harrop formulas, define $\Gamma_A$ as $\Gamma_B \cup \Gamma_C \cup \{\langle B \wedge C \rangle \}$. From the recursion step, we know that the multisets $\Gamma_B$ and $\Gamma_C$ only consist of modal Horn formulas, $\mathbf{CK} \vdash \Gamma_B \Rightarrow B^t$, $\mathbf{CK} \vdash \Gamma_C \Rightarrow C^t$, and $\sigma_B$ and $\sigma_C$ satisfy $\mathbf{CK} \vdash^{\sigma_B} \bigwedge \Gamma_B^s \Leftrightarrow B$ and $\mathbf{CK} \vdash^{\sigma_C} \bigwedge \Gamma_C^s \Leftrightarrow C$. Given this data, it is easy to see that Conditions $(i)$ and $(ii)$ are satisfied for $\Gamma_{A}$. For $\sigma_A$, it is easy to use the proofs $\sigma_B$ and $\sigma_C$ to construct the proof $\sigma_A$ for $\bigwedge \Gamma_B^s \wedge \bigwedge \Gamma_C^s \wedge \langle B \wedge C \rangle^s \Leftrightarrow A$. \\
For $A =B \to C$, where $C$ is Harrop, take 
\begin{center}
    $\Gamma_A = \{ \langle B \to C \rangle \} \cup \{\langle B \rangle \to \gamma \mid \gamma \in \Gamma_C \}.$
\end{center}
Condition $(i)$ is satisfied for $\Gamma_A$: the formula $\langle B \to C \rangle$ is an angled atom and as $\langle B \rangle$ is an atom, $\Gamma_C$ consists of modal Horn formulas, and by Definition \ref{ImplicationalHorn} the set of modal Horn formulas are closed under implications with atomic antecedent and modal Horn succedents, the formula $\langle B \rangle \to \gamma$ is modal Horn.

For Condition $(ii)$, we know that $\mathbf{CK} \vdash \Gamma_C \Rightarrow C^t$ and hence $\mathbf{CK} \vdash \bigwedge\Gamma_C \Rightarrow C^t$. By Lemma \ref{TranslationAndAtoms} we have $\mathbf{CK} \vdash B^t \Rightarrow \langle B \rangle$ and by $(L \to)$ we get $\mathbf{CK} \vdash B^t, \langle B \rangle \to \bigwedge \Gamma_C \Rightarrow C^t$. As $\mathbf{CK} \vdash \{\langle B \rangle \to \gamma \mid \gamma \in \Gamma_C \} \Rightarrow \langle B \rangle \to \bigwedge \Gamma_C$, we finally get $\mathbf{CK} \vdash \{ \langle B \to C \rangle \} \cup \{\langle B \rangle \to \gamma \mid \gamma \in \Gamma_C \} \Rightarrow (B \to C)^t $. 

For $\sigma_{A}$, take the following derivable sequents in $\mathbf{CK}$:
\small \begin{center}
\begin{tabular}{c c c}
$\bigwedge \Gamma_C^s \Leftrightarrow C$ , & $\langle B \to C \rangle^s \Leftrightarrow B \to C$ ,
& $\bigwedge_{\gamma \in \Gamma_C} (\langle B \rangle \to \gamma)^s \Leftrightarrow B \to \bigwedge \Gamma_C^s$,\\
\end{tabular}
\end{center}
\normalsize where the leftmost sequent is provable by $\sigma_C$, the middle one is an instance of the axiom $(id)$ and the left to right direction of the rightmost sequent is a result of applying the rule $(R \wedge)$ for $\parallel \Gamma_C \parallel -1$ many times on the canonical proof of $\bigwedge_{\gamma \in \Gamma_C} (B \to \gamma^s), B \Rightarrow \gamma^s$ and then using $(R \to)$, while the other direction is clear. Using these three proofs, it is easy to construct the proof $\sigma_A$.
It is noteworthy that although the choice $\{\langle B \rangle \to \bigwedge \Gamma_C , \langle A \rangle\}$ for $\Gamma_A$ seems more reasonable, it is not a possibility, as the formula $\langle B \rangle \to \bigwedge \Gamma_C$ is not necessarily in the modal Horn form as the class of modal Horn formulas is not closed under conjunctions.\\ 
For $A = \Box B$, where $B$ is Harrop, define $\Gamma_A= \Box \Gamma_B \cup \{\langle \Box B \rangle\}$. First, as the set of modal Horn formulas is closed under box, $\Gamma_A$ consists of modal Horn formulas, built only from angled atoms and $\bot$. Second, as we already have $\mathbf{CK} \vdash \Gamma_B \Rightarrow B^t$, using the rules $(K_{\Box}), (L w)$, and $(R \wedge)$, we have $\mathbf{CK} \vdash \Box \Gamma_B, \langle \Box B \rangle \Rightarrow (\Box B)^t$. Third, using $\sigma_B$ followed by two applications of $(K_{\Box})$, we have a proof for $\Box \bigwedge \Gamma_B^s \Leftrightarrow \Box B$. As observed in the proof of Theorem \ref{Commutation}, for any multiset $\Omega$ we have $\mathbf{CK} \vdash \Box \bigwedge \Omega \Rightarrow \bigwedge \Box \Omega$ and the proof takes $|\Omega|^{O(1)}$ many steps. Using this proof for $\Omega=\Gamma_B^s$, we get a proof for $\bigwedge (\Box \Gamma_B^s) \Leftrightarrow \Box B$ which provides the proof $\sigma_A$ for $\bigwedge \Gamma_A^s \Leftrightarrow A$ in $\mathbf{CK}$.

Now, we discuss the feasibility of the above algorithms we used to compute $\Gamma_A$ and $\sigma_A$. Let us start with $\Gamma_A$ and denote the time that the algorithm takes to compute $\Gamma_A$ by $T_{\Gamma}(A)$. 
First, we need an upper bound on the cardinality and the size of $\Gamma_A$.  For the former, based on how $\Gamma_A$ is constructed, we have:
\begin{itemize}
    \item 
    $\parallel \Gamma_A \parallel =1$, when $A$ is either an atom, $\top$ or $\bot$;
    \item 
    $\parallel \Gamma_{B \wedge C} \parallel =\parallel \Gamma_B \parallel + \parallel \Gamma_C \parallel +1$;
    \item 
    $\parallel \Gamma_{B \to C} \parallel = \parallel \Gamma_C \parallel +1$;
    \item 
    $\parallel \Box \Gamma_{B} \parallel =\parallel \Gamma_B \parallel +1$.
\end{itemize}
Hence, $\parallel \Gamma_A \parallel \leq |A|$. For $|\Gamma_A|$, we trivially have the following inequalities:
\begin{itemize}
    \item
$|\Gamma_A| \leq O(1)$, when $A$ is either an atom, $\top$, or $\bot$;
    \item 
    $|\Gamma_{B \wedge C}| \leq |\Gamma_B|+ |\Gamma_C| + O(|A|)$, for $A=B \wedge C$;
    \item
    $|\Gamma_{B \to C}| \leq |\Gamma_C| + O(\parallel \Gamma_C \parallel |B|) +O(|A|)$, for $A= B \to C$;
    \item
    $|\Gamma_{\Box B}| \leq |\Gamma_B| + O(\parallel \Gamma_B \parallel) +O(|A|)$, for $A=\Box B$;
\end{itemize}
Therefore, using the upper bound $\parallel \Gamma_A \parallel \leq |A|$, we can prove $|\Gamma_A| \leq |A|^{O(1)}$. Now, having the upper bounds on $\parallel \Gamma_A \parallel$ and $|\Gamma_A|$ established, we can move to compute $T_{\Gamma}(A)$. For that purpose, we have the following inequalities:
\begin{itemize}
\item
$T_\Gamma(A) \leq O(1)$, when $A$ is either an atom, $\top$, or $\bot$;
    \item 
    $T_\Gamma(B \wedge C) \leq T_\Gamma(B)+ T_\Gamma(C) + O(|A|)$, for $A=B \wedge C$;
    \item
    $T_\Gamma(B \to C) \leq T_\Gamma(C) + O(|\Gamma_C|+\parallel \Gamma_C \parallel |B|) +O(|A|)$, for $A= B \to C$;
    \item
    $T_\Gamma(\Box B) \leq T_\Gamma(B) + O(\parallel \Gamma_B \parallel+|\Gamma_B|) +O(|A|)$, for $A=\Box B$;
\end{itemize}
To justify, note that the addend $O(|\Gamma_C|+\parallel \Gamma_C \parallel |B|)$ is the time required to compute $\{\langle B \rangle \to \gamma \mid \gamma \in \Gamma_C \}$ in $\Gamma_{B \to C}$ and the addend $O(\parallel \Gamma_B \parallel+|\Gamma_B|)$ is the time required to compute $\Box \Gamma_B$ in $\Gamma_{\Box B}$. Finally, using these inequalities and the upper bounds $\parallel \Gamma_A \parallel \leq |A|$ and $|\Gamma_A| \leq |A|^{O(1)}$, 
we get $T_\Gamma(A) \leq |A|^{O(1)}$.\\
To compute the time complexity of $\sigma_A$, denote the time that the algorithm takes by $T_\sigma(A)$. Using each recursive step of the construction of $\sigma_A$, we have:
\begin{itemize}
\item
$T_\sigma(A) \leq O(1)$, where $A$ is either an atom, $\top$ or $\bot$;
\item
$T_\sigma(B \wedge C) \leq T_\sigma (B)+ T_\sigma(C)+ O(|\Gamma_B|+|\Gamma_C|+|A|)$, for $A=B \wedge C$;
\item
$T_\sigma(B \to C) \leq T_\sigma(C) + O(\parallel \Gamma_C \parallel (|\Gamma_C|+|A|))$, for $A=B \to C$;
\item
$T_\sigma(\Box B) \leq T_\sigma(B)+|\Gamma_B|^{O(1)}+ O(|A|)$, for $A=\Box B$.
\end{itemize}
It is easy to see why these inequalities hold. We only explain the modal case $A=\Box B$. The only non-trivial addend is $|\Gamma_B|^{O(1)}$ which is the bound on the time of the computation of the proof of $\bigwedge (\Box \Gamma^s_B) \Leftrightarrow \Box \bigwedge \Gamma^s_B$.
Finally, using the inequalities and the bounds we used before, we get $T_\sigma(A) \leq |A|^{O(1)}$.
\end{proof}

\subsection{The Provability Preservation}\label{SubsectionPreservability}
In this subsection, we complete the first ingredient of our strategy as explained in the opening of Section \ref{SectionMain}. We show that the translation $t$ preserves the $(\mathbf{CK}+\mathcal{C})$-provability, for any finite set $\mathcal{C}$ of constructive formulas. As mentioned before, this preservation holds if we interpret the provability up to the presence of a multiset of modal Horn assumptions that the standard substitution $s$ sees as provable. 

\begin{thm} (Provability Preservation) \label{MainTheorem}
Let $G=\mathbf{CK}+\mathcal{C}$, where $\mathcal{C}$ is a finite set of constructive formulas.
There is a feasible algorithm that reads a $G$-proof $\pi$ of $\Omega \Rightarrow \Lambda$ and outputs a multiset $\Sigma_{\pi}$ and a $G$-proof $\sigma_{\pi}$ such that:
\begin{enumerate}
\item\label{0}
$G \vdash \Sigma_{\pi}, \Omega^t \Rightarrow  \Lambda^t$,
\item\label{i}
formulas in $\Sigma_{\pi}$ are modal Horn formulas constructed from angled atoms,
\item\label{ii}
$G\vdash^{\sigma_{\pi}} \; \Rightarrow \bigwedge \Sigma_{\pi}^{s}$, where $s$ is the standard substitution.
\end{enumerate}
\end{thm}

\begin{proof} 
We first provide the recursive algorithm that produces $\Sigma_{\pi}$ and $\sigma_{\pi}$ and check the feasibility of the algorithm later. There are three cases to consider. Either the proof $\pi$ for $\Omega \Rightarrow \Lambda$ is an instance of an axiom in $\mathbf{CK}$, or an instance $(\, \Rightarrow C_\pi(\overline{\phi}_\pi))$ of the axiom $(\, \Rightarrow C_\pi(\overline{p}))$, where $C_\pi(\overline{p}) \in \mathcal{C}$ is a constructive formula, or it is a consequence of a rule in $\mathbf{CK}$. In the first case, set $\Sigma_{\pi}=\{\langle \top \rangle\}$. For \eqref{0}, it is easy to see that $\Sigma_{\pi}, \Omega^t \Rightarrow \Lambda^t$ is an instance of an axiom in $\mathbf{CK}$ and hence is provable in $G$.
Condition \eqref{i} is clear. For \eqref{ii}, as $(\, \Rightarrow \top)$ is an instance of an axiom, we can use it itself as the proof $\sigma_{\pi}$.

For the second case,
set $\Sigma_{\pi}=\Upsilon_{C_\pi, \overline{\phi}_\pi} \cup \{\langle C_\pi(\overline{\phi}_\pi) \rangle\}$, where $\Upsilon_{C_\pi, \overline{\phi}_\pi}$ is the set of modal Horn formulas constructed from angled atoms provided by Theorem \ref{Commutation}. Note that there are only a constant number of formulas in $\mathcal{C}$. Therefore, finding $C_\pi$ and $\overline{\phi}_\pi$ from $\pi$ is a polynomial time process and hence it is possible to feasibly compute $\Upsilon_{C_\pi, \overline{\phi}_\pi}$. For \eqref{0}, as $(\, \Rightarrow C_\pi(\overline{\phi_\pi^t}))$ is an instance of the axiom and by Theorem \ref{Commutation}, we have $\Upsilon_{C_\pi, \overline{\phi}_\pi}, \langle C_\pi(\overline{\phi}_\pi) \rangle, C_\pi(\overline{\phi_\pi^t}) \Rightarrow C_\pi(\overline{\phi}_\pi)^t$, using the cut rule, we get $\Sigma_{\pi} \Rightarrow C_\pi(\overline{\phi}_\pi)^t$ in $G$. Condition \eqref{i} is clear by Theorem \ref{Commutation}. For \eqref{ii}, by Theorem \ref{Commutation}, we have $\mathbf{CK} \vdash^{\sigma_{\Upsilon, C_\pi, \overline{\phi}_\pi}} (\Rightarrow \bigwedge \Upsilon^s_{C_\pi, \overline{\phi}_\pi})$.
Moreover, the standard translation of $\langle C_\pi(\overline{\phi}_\pi) \rangle$ is $ C_\pi(\overline{\phi}_\pi)$ and $(\, \Rightarrow C_\pi(\overline{\phi}_\pi))$ is an axiom in $G$ with the proof $\pi$. Therefore, it is easy to construct the proof $\sigma_{\pi}$ for the sequent $(\, \Rightarrow \bigwedge \Sigma_{\pi}^{s})$, by applying $(R\wedge)$ on $\sigma_{\Upsilon, C_\pi, \overline{\phi}_\pi}$ and $\pi$.

Now, suppose that we are in the third case. Here, we need to investigate all the possibilities for the last rule in the proof $\pi$. We will only address the cases where the last rule is either cut, the conjunction rules, or the modal rules. The other cases are similar. 
For cut, the last rule is of the form
\begin{center}
\begin{tabular}{c c c}
\AxiomC{$\pi_{1}$}
\noLine
\UnaryInfC{$\Omega \Rightarrow A$}
\AxiomC{$\pi_{2}$}
\noLine
\UnaryInfC{$\Omega, A \Rightarrow \Lambda$}
\BinaryInfC{$\Omega \Rightarrow \Lambda$}
\DisplayProof
\end{tabular}
\end{center}
By recursion, we have $\Sigma_{\pi_1}$ and $\Sigma_{\pi_2}$ such that $G \vdash \Sigma_{\pi_1}, \Omega^t \Rightarrow A^t $ and $G \vdash \Sigma_{\pi_2}, \Omega^t, A^t \Rightarrow \Lambda^t$. Set $\Sigma_{\pi}=\Sigma_{\pi_1} \cup \Sigma_{\pi_2}$. We have:
\begin{center}
\begin{tabular}{c c c}
\AxiomC{$\Sigma_{\pi_1}, \Omega^t \Rightarrow A^t$}
\doubleLine 
\UnaryInfC{$\Sigma_{\pi_1}, \Sigma_{\pi_2}, \Omega^t \Rightarrow A^t$}
\AxiomC{$\Sigma_{\pi_2}, \Omega^t, A^t \Rightarrow \Lambda^t$}
\doubleLine 
\UnaryInfC{$\Sigma_{\pi_1}, \Sigma_{\pi_2}, \Omega^t, A^t \Rightarrow \Lambda^t$}
\BinaryInfC{$\Sigma_{\pi_1}, \Sigma_{\pi_2}, \Omega^t \Rightarrow \Lambda^t$}
\DisplayProof
\end{tabular}
\end{center}
where the double lines mean applying the left weakening rule as many times as needed to get the required sequents and the last rule is cut. This proves Condition \eqref{0}. Condition \eqref{i} is clearly satisfied. For \eqref{ii}, it is enough to define $\sigma_{\pi}$ as an application of $(R\wedge)$ over the proofs $\sigma_{\pi_1}$ and $\sigma_{\pi_2}$.

If the last rule is $(L \wedge_1)$, then the proof is of the form
\begin{center}
\begin{tabular}{c}
\AxiomC{$\pi_{1}$}
\noLine
\UnaryInfC{$\Omega, A \Rightarrow \Lambda $}
\UnaryInfC{$\Omega, A \wedge B \Rightarrow \Lambda$}
\DisplayProof
\end{tabular}
\end{center}
By recursion, we have $\Sigma_{\pi_1}$ such that $G \vdash \Sigma_{\pi_1}, \Omega^t, A^t \Rightarrow \Lambda^t$. Set
$\Sigma_\pi=\Sigma_{\pi_1}$. For Condition \eqref{0}, by the rule itself: 
\begin{center}
\begin{tabular}{c}
\AxiomC{$\Sigma_{\pi_1}, \Omega^t, A^t \Rightarrow \Lambda^t $}
\UnaryInfC{$\Sigma_{\pi_1}, \Omega^t, A^t \wedge B^t \Rightarrow \Lambda^t$}
\DisplayProof
\end{tabular}
\end{center}
As $G \vdash (A \wedge B)^t \Rightarrow A^t \wedge B^t$, by cut we have
$G \vdash \Sigma_{\pi_1}, \Omega^t, (A \wedge B)^t \Rightarrow \Lambda^t$. Condition \eqref{i} is clear. For Condition \eqref{ii}, setting $\sigma_{\pi}=\sigma_{\pi_1}$ clearly works.

If the last rules is $(R \wedge)$, then the proof is of the form
\begin{center}
\begin{tabular}{c c c}
\AxiomC{$\pi_{1}$}
\noLine
\UnaryInfC{$\Omega \Rightarrow A $}
\AxiomC{$\pi_{2}$}
\noLine
\UnaryInfC{$\Omega \Rightarrow B $}
\BinaryInfC{$\Omega \Rightarrow A \wedge B$}
\DisplayProof
\end{tabular}
\end{center}
By recursion, we have $\Sigma_{\pi_1}$ and $\Sigma_{\pi_2}$ such that $G \vdash \Sigma_{\pi_1}, \Omega^t \Rightarrow A^t $ and $G \vdash \Sigma_{\pi_2}, \Omega^t \Rightarrow B^t$. Set $\Sigma_\pi=\{\bigwedge_{\omega \in \Omega} \langle \omega \rangle \to \langle A \wedge B \rangle\} \cup \Sigma_{\pi_1} \cup \Sigma_{\pi_2}$. For Condition \eqref{0}, we have:
\begin{center}
\begin{tabular}{c c c}
\AxiomC{$\Sigma_{\pi_1}, \Omega^t \Rightarrow A^t $}
\doubleLine 
\UnaryInfC{$\Sigma_{\pi_1}, \Sigma_{\pi_2}, \Omega^t \Rightarrow A^t $}
\AxiomC{$\Sigma_{\pi_2}, \Omega^t \Rightarrow B^t $}
\doubleLine 
\UnaryInfC{$\Sigma_{\pi_1}, \Sigma_{\pi_2}, \Omega^t \Rightarrow B^t $}
\BinaryInfC{$\Sigma_{\pi_1}, \Sigma_{\pi_2}, \Omega^t \Rightarrow A^t \wedge B^t$}
\DisplayProof
\end{tabular}
\end{center}
where the double lines again mean applying the left weakening rule for multiple times and the last rule is $(R\wedge)$. This is almost what we wanted, except that the succedent of the conclusion must be of the form $(A \wedge B)^t=(A^t \wedge B^t) \wedge \langle A \wedge B \rangle$. However, by Lemma \ref{TranslationAndAtoms}, we have $\mathbf{CK} \vdash \omega^t \Rightarrow \langle \omega \rangle$, for any $\omega \in \Omega$. Therefore, as $\bigwedge_{\omega \in \Omega} \langle \omega \rangle \to \langle A \wedge B \rangle \in \Sigma_{\pi}$, we get $\Sigma_\pi, \Omega^t \Rightarrow (A \wedge B)^t$. Condition \eqref{i} is clear by Definition \ref{ImplicationalHorn}. For Condition \eqref{ii}, use the proof $\pi$ for the sequent $\Omega \Rightarrow A \wedge B$ together with $\parallel \Omega \parallel$ many applications of $(L\wedge)$, $\parallel \Omega \parallel-1$ many contractions and then one application of $(R\to)$ to prove $(\, \Rightarrow \bigwedge \Omega \to  A \wedge B)$ or equivalently $\, \Rightarrow (\bigwedge_{\omega \in \Omega} \langle \omega \rangle \to \langle A \wedge B \rangle)^s$. Then,
using the proofs $\sigma_{\pi_1}$ and $\sigma_{\pi_2}$, it is easy to construct the proof $\sigma_{\pi}$.

Before moving to the modal rules, we explain how a general rule in $\LJ$ is addressed. The structural rules are easy to handle, as they always commute with the translations. If the last rule in $\pi$ is a left rule in $\LJ$, the multiset $\Sigma_\pi$ is defined as the union of the multisets $\Sigma_{\pi_i}$, obtained from the recursive step for the proofs $\pi_i$ of the premises of the rule. If the last rule is a right rule for $\circ \in \{\wedge, \vee, \to\}$ in $\LJ$, we also need to add the formula $\bigwedge_{\omega \in \Omega} \langle \omega \rangle \to \langle A \circ B \rangle$ to the union of $\Sigma_{\pi_i}$'s, where $\Omega$ is the multiset variable in the antecedent of the conclusion. The proof $\sigma_\pi$ is obtained similar to the cases investigated above.

If the last rule in $\pi$ is either $(K \Box)$ or $(K \Diamond)$, the proof is of the form
\begin{center}
\begin{tabular}{c c}
\AxiomC{$\pi_1$}
\noLine 
 \UnaryInfC{$\Omega \Rightarrow A$}
  \RightLabel{$(K_{\Box})$}
 \UnaryInfC{$\Box \Omega \Rightarrow \Box A$}
 \DisplayProof
& \hspace{10 pt}
\AxiomC{$\pi_1$}
\noLine 
 \UnaryInfC{$\Omega, A \Rightarrow B$}
  \RightLabel{$(K_{\Diamond})$}
 \UnaryInfC{$\Box \Omega, \Diamond A \Rightarrow \Diamond B$}
 \DisplayProof 
 \end{tabular}
\end{center}
For $(K \Box)$, set $\Sigma_{\pi}=\Box \Sigma_{\pi_1} \cup \{\bigwedge_{\omega \in \Omega} \langle \Box \omega \rangle \to \langle \Box A \rangle\}$ and for $(K \Diamond)$, set $\Sigma_{\pi}=\Box \Sigma_{\pi_1} \cup \{\bigwedge_{\omega \in \Omega} \langle \Box \omega \rangle \wedge \langle \Diamond A \rangle \to \langle \Diamond B \rangle \}$. We will only investigate the latter case, as the former is similar to this case and to the case of the right rules in $\LJ$. Using Definition \ref{ImplicationalHorn}, by the fact that the set of modal Horn formulas is closed under box, Condition \eqref{i} is satisfied. For Condition \eqref{0}, by recursion, we have $\Sigma_{\pi_1}, \Omega^t , A^t \Rightarrow B^t$. Using the rule $(K \Diamond)$, we get
\begin{center}
$\Box \Sigma_{\pi_1}, \Box \Omega^t , \Diamond A^t \Rightarrow \Diamond B^t$.
\end{center}
By Lemma \ref{TranslationAndAtoms}, we have 
$\mathbf{CK} \vdash (\Diamond A)^t \Rightarrow \langle \Diamond A \rangle$ and $\mathbf{CK} \vdash (\Box \omega)^t \Rightarrow \langle \Box \omega \rangle$, for any $\omega \in \Omega$. Hence,
\begin{center}
$\mathbf{CK} \vdash \Box \Sigma_{\pi_1}, (\bigwedge_{\omega \in \Omega} \langle \Box \omega \rangle \wedge \langle \Diamond A \rangle \to \langle \Diamond B \rangle), (\Box \Omega)^t, (\Diamond A)^t \Rightarrow \Diamond B^t \wedge \langle \Diamond B \rangle$.
\end{center}
For Condition \eqref{ii}, to provide $\sigma_{\pi}$, consider the proof $\pi$ for $\mathbf{CK} \vdash \Box \Omega, \Diamond A \Rightarrow \Diamond B$. Then, use $\parallel \Omega \parallel+1$ many applications of $(L\wedge)$, $\parallel \Omega \parallel$ many contractions and then one application of $(R\to)$ to prove $( \, \Rightarrow \bigwedge \Box \Omega \wedge \Diamond A \to  \Diamond B)$ or equivalently $\Rightarrow (\bigwedge_{\omega \in \Omega} \langle \Box \omega \rangle \wedge \langle \Diamond A \rangle \to \langle \Diamond B \rangle)^s$. Furthermore, by recursion, we have $\sigma_{\pi_1}$ such that $\mathbf{CK} \vdash^{\sigma_{\pi_1}} (\Rightarrow \bigwedge \Sigma_{\pi_1}^s)$. By $(K_{\Box})$, we have $\mathbf{CK} \vdash (\Rightarrow \Box \bigwedge \Sigma_{\pi_1}^s)$ and by the provable sequent $\Box \bigwedge \Sigma_{\pi_1}^s \Rightarrow \bigwedge \Box \Sigma_{\pi_1}^s$, we get $\mathbf{CK} \vdash (\Rightarrow \bigwedge \Box \Sigma_{\pi_1}^s)$. Putting these proofs together, we easily provide the proof $\sigma_{\pi}$.

Finally, we have to prove the feasibility of the described processes to find $\Sigma_{\pi}$ and $\sigma_{\pi}$. Denote the time that the algorithm needs to compute $\Sigma_\pi$ and $\sigma_{\pi}$, by $T_{\Sigma}(\pi)$ and $T_{\sigma}(\pi)$, respectively. We will show that
\begin{itemize}
    \item 
    $T_{\Sigma} (\pi) \leq |\pi|^{O(1)}$, if $\pi$ is an instance of an axiom in $G$, 
    \item 
    $T_{\Sigma} (\pi) \leq T_{\Sigma} (\pi_1) + |\pi|^{O(1)}$, if the last rule used in $\pi$ is a one premise rule in $G$ with the immediate subproof $\pi_1$, and
    \item
    $T_{\Sigma} (\pi) \leq T_{\Sigma} (\pi_1) + T_{\Sigma} (\pi_2) + |\pi|^{O(1)}$, if the last rule used in $\pi$ is a two premise rule in $G$ with the immediate subproofs $\pi_1$ and $\pi_2$.
\end{itemize}
and similarly for $T_{\sigma}(\pi)$. Using these upper bounds, it is clear that both $T_{\Sigma}(\pi)$ and $T_{\sigma}(\pi)$ are polynomially bounded in $|\pi|$. To prove the inequalities, if $\pi$ is an instance of an axiom in $\mathbf{CK}$, it is clear from the construction that the upper bounds are in place. For the other base case, suppose $\pi$ is an instance $(\, \Rightarrow C_{\pi}(\overline{\phi}_{\pi}))$ of the axiom $(\, \Rightarrow C_{\pi}(\overline{p}))$, for a constructive formula $C_{\pi} \in \mathcal{C}$. Recall from the beginning of the proof that the process of finding $C_{\pi}$ and $\overline{\phi}_{\pi}$ from $\pi$ is polynomial time. Now, we first use the following two inequalities to prove the upper bounds for $T_{\Sigma}(\pi)$ and $T_{\sigma}(\pi)$ and then we will justify them:
\begin{enumerate}
\item \label{Sigma}
$T_{\Sigma}(\pi) \leq T_{C_\pi, \overline{\phi}_\pi}(\pi)+ T_{\Upsilon}(C_{\pi}, \overline{\phi}_{\pi})+ O(|\pi|),$ and 
\item \label{sigma}
$T_{\sigma}(\pi) \leq T_{C_\pi, \overline{\phi}_\pi}(\pi)+T_{\sigma, \Upsilon}(C_{\pi}, \overline{\phi}_{\pi})+O(|\sigma_{\Upsilon, C_{\pi}, \overline{\phi}_{\pi}}|)+O(|\pi|)$,
\end{enumerate}
where $T_{C_\pi, \overline{\phi}_\pi}(\pi)$ is the time to compute $C_{\pi}$ and $\overline{\phi}_{\pi}$ from $\pi$, the addends $T_{\Upsilon}(C_{\pi}, \overline{\phi}_{\pi})$ and $T_{\sigma, \Upsilon}(C_{\pi}, \overline{\phi}_{\pi})$ are the times obtained from Theorem \ref{Commutation}, both polynomial in $|C_{\pi}|$ and $|\overline{\phi}_{\pi}|$ and $\sigma_{\Upsilon, C_{\pi}, \overline{\phi}_{\pi}}$ is the proof constructed in Theorem \ref{Commutation}, polynomially computable in $C_{\pi}$ and $\overline{\phi}_{\pi}$. As both $|C_{\pi}|$ and $|\overline{\phi}_{\pi}|$ are bounded by $|\pi|$, the addends $T_{\Upsilon}(C_{\pi}, \overline{\phi}_{\pi})$, $T_{\sigma, \Upsilon}(C_{\pi}, \overline{\phi}_{\pi})$ and $|\sigma_{\Upsilon, C_{\pi}, \overline{\phi}_{\pi}}|$ are all polynomial in $|\pi|$ which proves the intended upper bounds. To justify the inequalities, for 
the first one, notice that to compute $\Sigma_{\pi}$, the algorithm reads the proof $\pi$ and decides whether it is of the form $(\, \Rightarrow C_{\pi}(\overline{\phi}_{\pi}))$ and if it is, it finds $C_{\pi}(\overline{p})$ and $\overline{\phi}_{\pi}$. This process takes $T_{C_\pi, \overline{\phi}_\pi}(\pi)$ many steps. Then, it must compute $\Upsilon_{C_{\pi},\overline{\phi}_{\pi}}$ that takes $T_{\Upsilon}(C_{\pi}, \overline{\phi}_{\pi})$ steps and also write down $\langle C_{\pi}(\overline{\phi}_{\pi}) \rangle$ which needs $O(|\pi|)$ amount of time. 
For 
the second inequality, the algorithm starts similarly, but then it has to find $\sigma_{\Upsilon, C_{\pi}, \overline{\phi}_{\pi}}$ which takes $T_{\sigma, \Upsilon}(C_{\pi}, \overline{\phi}_{\pi})$ steps. Then, we use $\pi$ as a proof for $(\, \Rightarrow C_{\pi}(\overline{\phi}_{\pi}))$ and finally use the rule $(R \wedge)$ that needs $O(|\sigma_{\Upsilon, C_{\pi}, \overline{\phi}_{\pi}}|)+O(|\pi|)$ many steps.

Now, we consider each case where the last rule in $\pi$ is one of the rules in $\mathbf{CK}$. Before diving into the details of these cases, let us first establish some upper bounds on $\parallel \Sigma_{\pi} \parallel$ and $|\Sigma_{\pi}|$ in general. We have:
\begin{itemize}
    \item
$\parallel \Sigma_{\pi} \parallel \leq O(1)$, if $\pi$ is an instance of an axiom in $\mathbf{CK}$, and
    \item
$\parallel \Sigma_{\pi} \parallel \leq \, \parallel \Upsilon_{C_{\pi}, \overline{\phi}_{\pi}} \parallel +1$, if $\pi$ is an instance $(\, \Rightarrow C_{\pi}(\overline{\phi}_{\pi}))$ of the axiom $(\, \Rightarrow C_{\pi}(\overline{p}))$, for a constructive formula $C_{\pi} \in \mathcal{C}$, 
    \item $\parallel \Sigma_{\pi} \parallel \leq \, \parallel \Sigma_{\pi_1} \parallel +1$, when the last rule in $\pi$ is a one premise rule,
    \item
    $\parallel\! \Sigma_{\pi} \!\parallel \leq \, \parallel \!\Sigma_{\pi_1}\!\! \parallel\! +\!\parallel \!\Sigma_{\pi_2}\!\! \parallel \!\!+1$, when the last rule in $\pi$ is a two premise rule.
\end{itemize}
Using these inequalities and, as we observed, the polynomial bound on $| \Upsilon_{C_{\pi}, \overline{\phi}_{\pi}} |$, it is easy  to use an induction on the structure of the proof $\pi$ to show that $\parallel \Sigma_{\pi} \parallel \leq |\pi|^{O(1)}$. Again, in a similar fashion, we have:
\begin{itemize}
     \item
$|\Sigma_{\pi}| \leq O(1)$, if $\pi$ is an instance of an axiom in $\mathbf{CK}$, 
    \item
$|\Sigma_{\pi}| \leq |\Upsilon_{C_{\pi}, \overline{\phi}_{\pi}}| + O(|\pi|)$, if $\pi$ is an instance $(\, \Rightarrow C_{\pi}(\overline{\phi}_{\pi}))$ of the axiom $(\, \Rightarrow C_{\pi}(\overline{p}))$, for a constructive formula $C_{\pi} \in \mathcal{C}$, 
    \item
    $|\Sigma_\pi| \leq |\Sigma_{\pi_1}|\! +\! O(|\pi|)$, if the last rule in $\pi$ is a one premise rule in $\mathbf{LJ}$,  
    \item
    $|\Sigma_\pi| \leq |\Sigma_{\pi_1}| + |\Sigma_{\pi_2}| + O(|\pi|)$, if the last rule in $\pi$ is a two premise rule in $\mathbf{LJ}$, and
    \item
    $|\Sigma_\pi| \leq\! |\Sigma_{\pi_1}| +\! \parallel\! \Sigma_{\pi_1\!}\! \parallel\! \!+ O(|\pi|)$, if the last rule in $\pi$ is a modal rule in $\mathbf{CK}$.
\end{itemize}
Using these inequalities and the polynomial bound on $|\Upsilon_{C_{\pi}, \overline{\phi}_{\pi}}|$, it is easy  to use an induction on the structure of the proof $\pi$ to show that $|\Sigma_{\pi}| \leq |\pi|^{O(1)}$. 

Now, we examine the different cases of the last rule. We only investigate the case where the last rule in $\pi$ is $(K\Diamond)$. The rest are similar. We claim:
\begin{enumerate}
    \item $T_{\Sigma} (\pi) \leq T_{\Sigma} (\pi_1) + O(|\Sigma_{\pi_1}|) + O(|\pi|)$, and
    \item
    $T_{\sigma} (\pi) \leq T_{\sigma}(\pi_1) + O(|\Sigma_{\pi_1}|)+ O(\parallel \Sigma_{\pi_1} \parallel |\Sigma_{\pi_1}|)+ O(|\pi|)$.
\end{enumerate}
The first inequality is derived by inspecting how $\Sigma_\pi$ is defined. More precisely, the algorithm first computes $\Sigma_{\pi_1}$ which requires $T_{\Sigma} (\pi_1)$ many steps. Then, adding boxes to the formulas in $\Sigma_{\pi_1}$ takes $O(|\Sigma_{\pi_1}|)$ many steps. The rest is adding boxes to $\Omega$ and diamonds to $A$ and $B$ and forming the angles which takes $O(|\pi|)$ steps. The second inequality is obtained by investigating the proof $\sigma_\pi$. As $\sigma_{\pi_1}$ is needed in the proof, the addend $T_{\sigma} (\pi_1)$ appears. The time $O(|\Sigma_{\pi_1}|)$ is needed to compute $(\Rightarrow \Box \bigwedge \Sigma^s_{\pi_1})$ from $(\Rightarrow \bigwedge \Sigma^s_{\pi_1})$. The addend $O(\parallel \Sigma_{\pi_1} \parallel |\Sigma_{\pi_1}|)$ is the time needed for the $\mathbf{CK}$-proof of $\Box \bigwedge \Sigma^s_{\pi_1} \Rightarrow \bigwedge \Box \Sigma^s_{\pi_1}$. Finally, $O(|\pi|)$ is the time needed to provide a proof for $\Rightarrow (\bigwedge_{\gamma \in \Gamma} \langle \Box \gamma \rangle \wedge \langle \Diamond C \rangle \to \langle \Diamond D \rangle)^s$. Having the claims established and using $\parallel \Sigma_{\pi} \parallel \leq |\pi|^{O(1)}$ and $|\Sigma_{\pi}| \leq |\pi|^{O(1)}$, the inequalities $T_{\Sigma} (\pi) \leq T_{\Sigma} (\pi_1) + |\pi|^{O(1)}$
   and
    $T_{\sigma} (\pi) \leq T_{\sigma}(\pi_1) + |\pi|^{O(1)}$ are in place as intended.
\end{proof}

\subsection{The Unit Propagation}\label{SubsectionUnitPropagation}
In this subsection, we will cover the second ingredient of our strategy as explained in the opening of Section \ref{SectionMain}. We will present a feasible algorithm to read a classically valid sequent $\Gamma \Rightarrow \bigvee_{i=1}^{n} p_i$, where $\Gamma$ is a multiset of implicational Horn formulas and $p_1, p_2, \dots , p_n$ are atomic formulas in $\mathcal{L}_p$, and output an index $1 \leq i \leq n$ and a proof $\pi$ such that $\mathbf{LJ} \vdash^{\pi} \Gamma \Rightarrow p_{i}$. The algorithm is a proof-theoretic version of the well-known \emph{unit propagation} or \emph{one-literal rule}, used to show the feasibility of Horn satisfiability. Here, we repeat the algorithm as described in \cite{Unit}.
It is also worth emphasizing that the algorithm is designed for the propositional language and hence there is no modality in this subsection. However, later in Subsection \ref{SubsectionTfreeTfull}, we will show how to lift the machinery of this subsection to the modal language.

\begin{thm}(Unit propagation) \label{UnitPropagation}
There is a feasible algorithm that reads classically valid sequents in the form $S=(\Gamma \Rightarrow \bigvee_{i=1}^{n} p_i)$, where $\Gamma$ is a multiset of implicational Horn formulas and $p_1, p_2, \dots , p_n$ are atomic formulas in $\mathcal{L}_p$, and outputs an index $1 \leq i \leq n$ and a proof $\pi$ such that $\mathbf{LJ} \vdash^{\pi} \Gamma \Rightarrow p_{i}$.
\end{thm}

\begin{proof}
By Definition \ref{ImplicationalHorn}, the elements of any multiset $\Pi$ of implicational Horn formulas are either atoms (or $\bot$), called \emph{units}, or formulas of the implicational form $\bigwedge Q \to r$, called \emph{the rest} or \emph{non-units} of $\Pi$, where $Q$ is a non-empty sequence of atomic formulas and $r$ is an atom or $\bot$. W.l.o.g, assume that the multiset $\Gamma$ is presented by a linear order on its elements. Let $U$ and $R$ be the ordered sets consisting of the units and the rest of $\Gamma$, respectively. We may assume that $U$ and $R$ have the same order as they appear in $\Gamma$. 

\noindent \emph{The algorithm:} Start by setting $V=U$, $W=\emptyset$ and $T=R$:

- if $V=\emptyset$, halt. Otherwise, take $u$ as the first unit in $V$;

- if $u=\bot$, then halt and output $p_1$;

- if $u=p_i$ for some $1 \leq i \leq n$, then halt and output $p_i$;

- otherwise, take the first formula $\phi$ in $T$ and check if $u$ can be

  \emph{unit resolved} against $\phi$, as defined below:

$  \left\{ 
  \begin{array}{ l }
    \text{- if}\; \phi= \bigwedge Q \to u, \text{delete}\; \phi\; \text{from}\; T;                 \\
    \text{- if}\; \phi=\bigwedge Q \to s \;\text{and}\; s \neq u \; \text{and}\; u \in Q,
     \text{delete}\; \phi \;\text{from}\; T\; \text{and}\; \text{add}\;\\ \psi=\bigwedge Q' \to s \; \text{to the end of}\; T, \;\text{where}\; Q'= Q -\{u\} \; \text{is non-empty}; \\
    \text{- if}\; \phi=u \to s\; \text{and}\; s \neq u, \text{delete}\; \phi \; \text{from}\; T \; \text{and add}\; s \; \text{to the end of}\; V.
  \end{array}
\right.$

\noindent Repeat the unit resolution process for the rest of the formulas in $T$. If $u$ cannot be unit resolved against any formula in $T$, change $V$ and $W$ to $V-\{u\}$ and $W \cup \{u\}$, respectively, and go to the first line of the algorithm.

\emph{Feasibility:} 
In the worst case, the algorithm keeps running while $V \neq \emptyset$ that takes at most $|S|$ many runs, where $|S|$ is the size of $S$. In each run, it takes a unit $u$, checks if it equals to any $p_i$ or $\bot$, and scans $T$ to see if $u$ can be unit resolved against any formulas in $T$. Then, it applies some small modifications on the elements of $T$, which makes them shorter in size. Hence, each run takes at most $O(|S|)$ many steps. Therefore, the time that the algorithm needs is at most $O(|S|^2)$. For more explanation, see \cite{Unit}.

\emph{Soundness:} We show that the algorithm halts before reaching $V = \emptyset$ and hence outputs $p_i$, for some $1 \leq i \leq n$. We also show that the algorithm indirectly provides a proof $\pi$ for $\Gamma \Rightarrow p_i$ in $\mathbf{LJ}$. Define a \emph{stage} of the algorithm as one run of the while loop and assume that the algorithm took $N$ stages to halt. Suppose, for the sake of contradiction, that the algorithm did not halt before reaching $V=\emptyset$. Let $V_k$, $W_k$ and $T_k$ be the ordered sets by which the stage $k$ of the algorithm starts and let $\Gamma_k=V_k \cup W_k \cup T_k$, for each $0 \leq k \leq N$. First, as $V_0=U$, $W_0=\emptyset$ and $T_0=R$, we have $\Gamma_0=V_0 \cup T_0=U \cup R=\Gamma$. Second, for any $0 \leq k \leq N$, none of the elements of $W_k$ occurs in $T_k$, as $W_k$ consists of the units that we have eliminated from any formula in $T$ before reaching the current stage $k$. Third, there are $\mathbf{LJ}$-proofs of $\bigwedge \Gamma_k \Leftrightarrow \bigwedge \Gamma_{k+1}$, for any $0 \leq k \leq N-1$, since the following are provable in $\LJ$:

\footnotesize \begin{center}
\begin{tabular}{c c c}
$u \wedge (\bigwedge Q \to u) \Leftrightarrow u$ , & $u \wedge (\bigwedge Q \to s) \Leftrightarrow u \wedge (\bigwedge Q' \to s)$ ,
& $u \wedge (u \to s) \Leftrightarrow u \wedge s$\\
\end{tabular}
\end{center}
\normalsize where in the middle case, $s \neq u$ and $Q'= Q -\{u\}\neq \emptyset$. For the later reference, notice that constructing each proof takes $O(|S|)$ many steps.

As $\bigwedge \Gamma_k \Leftrightarrow \bigwedge \Gamma_{k+1}$ for any $0 \leq k \leq N-1$,
and $\Gamma_0 \Rightarrow \bigvee_{i=1}^n p_i$ is classically valid, $\Gamma_N \Rightarrow \bigvee_{i=1}^n p_i$ is also classically valid.
Since the algorithm did not halt before reaching $V_N=\emptyset$, the used units did not intersect with $\{p_i\}_{i=1}^n$, and none of them were $\bot$. Hence, $W_N \cap \{p_1, \cdots, p_n, \bot\}=\emptyset$. As $\Gamma_N \Rightarrow \bigvee_{i=1}^n p_i$ is classically valid and $\Gamma_N=W_N \cup T_N$, we have $W_N , T_N \Rightarrow \bigvee_{i=1}^n p_i$ is classically valid. This is a contradiction for the following reason. Recall that all the formulas in $T_N$ are non-units and hence implicational and none of the elements of $W_N$ occurs in $T_N$. Therefore, taking the valuation that makes every atomic formula in $W_N$ true and the rest false, satisfies $W_N \cup T_N$. Moreover, as $W_N \cap \{p_1, \cdots, p_n, \bot\}=\emptyset$, the valuation maps $\bigvee_{i=1}^n p_i$ to false, which contradicts the fact that $W_N , T_N \Rightarrow \bigvee_{i=1}^n p_i$ is classically valid. Hence, the algorithm halts and outputs a $p_i$, before reaching $V=\varnothing$. 

Now, we have to show that the algorithm indirectly provides an $\LJ$-proof $\pi$ for $\Gamma \Rightarrow p_i$. Assume that the algorithm halts after $N$ many stages. Then, either $p_i \in V_N$ or $\bot \in V_N$. Hence, as $\Gamma_N=V_N \cup W_N \cup T_N$, we have $\mathbf{LJ} \vdash  \Gamma_N \Rightarrow p_i$ and as $\mathbf{LJ} \vdash  \bigwedge \Gamma_N \Leftrightarrow \bigwedge \Gamma_0=\bigwedge \Gamma$, we have a proof $\pi$ for $\Gamma \Rightarrow p_i$ in $\LJ$. Note that the proof $\pi$ uses the aforementioned $\LJ$-equivalence between $\bigwedge \Gamma_k$ and $\bigwedge \Gamma_{k+1}$, some basic propositional rules and $N-1$ many cuts. As the proofs of the equivalences take $O(|S|)$ many steps and $N \leq O(|S|^2)$, by the feasibility part, the time to produce $\pi$ is $O(|S|^3)$.
\end{proof}

\begin{rem}
Here are two remarks on the algorithm provided in Theorem \ref{UnitPropagation}. First, notice that the algorithm only uses the classical validity of the sequent $S=(\Gamma \Rightarrow \bigvee_{i=1}^{n} p_i)$ and not its classical proof. This observation becomes helpful later in the proof of our main result, Theorem \ref{FeasibleHarropForIK}. Focusing on the validity rather than the proof allows us to only control the complexity of the construction of the sequent $S$, to which we will apply Theorem \ref{UnitPropagation}, and not its proof.
Second, note that despite the fact that the assumption is the classical validity of $S$, the algorithm finally provides an $\mathbf{LJ}$-proof for some $\Gamma \Rightarrow p_i$ and not just a proof in classical logic. This second point also plays a crucial role in this paper. In the next subsection, we will see that our main method to lift the unit propagation to the modal setting is extending the modal calculus to collapse all the modalities. The fact that the unit propagation only needs the classical validity makes it possible to extend the calculus even to the classical systems and then using the $\mathbf{LJ}$-proof provided by Theorem \ref{UnitPropagation}, we can land in the intuitionistic realm again.
\end{rem}

\subsection{$T$-freeness and $T$-fullness}\label{SubsectionTfreeTfull}
As mentioned before, the unit propagation is only applicable to the propositional language, while we need a similar machinery for the modal setting. For that purpose, one way to proceed is to reduce the provability in a modal calculus to the classical validity of a propositional formula to employ the unit propagation. Our strategy is to extend the given modal calculus to a calculus for a classical modal logic, where the modalities have no real role and hence can be eliminated. In doing so, there are two canonical logics that one can use, i.e., the logics of the constructive modal Kripke frame with one, either reflexive or irreflexive, node. $T$-freeness and $T$-fullness capture this idea, first applied to logics and then to sequent calculi.

\begin{dfn}\label{Def: T-free T-full logic} 
Let $\mathfrak{L} \in \{\mathcal{L}, \mathcal{L}_{\Box}, \mathcal{L}_\Diamond\}$ be a language and $L$ be a logic over $\mathfrak{L}$. The logic $L$ is called \emph{$T$-free over $\mathfrak{L}$}, if it is valid in the irreflexive node frame $\mathcal{K}_i$ (see Definition \ref{def: Kripke Frames and Models}), and

\begin{itemize}
    \item 
    if $\mathfrak{L}=\mathcal{L}$, then $\mathsf{CK} \subseteq L$.
     \item 
    if $\mathfrak{L}=\mathcal{L}_\Box$, then $\mathsf{CK}_\Box \subseteq L$.
     \item 
    if $\mathfrak{L}=\mathcal{L}_\Diamond$, then $\mathsf{BLL} \subseteq L$.
\end{itemize}
The logic $L$ is \emph{$T$-full over $\mathfrak{L}$} if it is valid in the reflexive node frame $\mathcal{K}_r$ and
\begin{itemize}
    \item 
    if $\mathfrak{L}=\mathcal{L}$, then $\mathsf{CK} \subseteq L$ and $T_a, T_b \in L$.
     \item 
    if $\mathfrak{L}=\mathcal{L}_\Box$, then $\mathsf{CK}_\Box \subseteq L$ and $T_a \in L$.
     \item 
    if $\mathfrak{L}=\mathcal{L}_\Diamond$, then $\mathsf{BLL} \subseteq L$ and $T_b \in L$.
\end{itemize}
\end{dfn}

\begin{exam}\label{ExampleOfTfreeTfullLogic}
Let $L$ be an intermediate logic and $\mathcal{A}$ and $\mathcal{B}$ two finite sets of the axioms in Table \ref{tableAxiom} such that $\mathcal{A}$ does not include any of the axioms $T_a$, $T_b$, $D_a$, $D_b$, $D$, $den_{n,a}$ and $den_{n,b}$, for $n=0$, and $ga_{klmn}$, for $k=m=0$. Then,  $L\mathsf{CK}+\mathcal{A}$ (recall Definition \ref{DefIntermediateModalLogics}) and $L\mathsf{IK}+\mathcal{A}$ are $T$-free and $L\mathsf{CK}+\mathcal{B} \cup \{T_a, T_b\}$ and $L\mathsf{IK}+\mathcal{B} \cup \{T_a, T_b\}$ are $T$-full over $\mathcal{L}$. We only prove the cases for the  $\mathsf{CK}$ versions, the others being similar. Notice that these logics extend $\mathsf{CK}$, all the theorems of $L$ are valid in any constructive modal Kripke frame with one node, and modus ponens and necessitation respect the validity in $\mathcal{K}_i$ and $\mathcal{K}_r$, i.e., the validity of the premises of the rules imply the validity of the conclusion. Therefore, it is enough to show that any axiom in $\mathcal{A}$ (resp. $\mathcal{B}$) is valid in $\mathcal{K}_i$ (resp. $\mathcal{K}_r$). We only check the case for the axiom $ga_{klmn}: \Diamond^k \Box^l p \to \Box^m \Diamond^n p$. The rest are similar. The axiom $ga_{klmn}$ is valid in $\mathcal{K}_r$, as $\Box \phi$ and $\Diamond \phi$ are equivalent to $\phi$ in $\mathcal{K}_r$, for any $\phi \in \mathcal{L}$. Hence, the axiom $ga_{klmn}$ is equivalent to $p \to p$ which is clearly valid in $\mathcal{K}_r$. For $\mathcal{K}_i$, if either $k \neq 0$ or $m \neq 0$, then the axiom $ga_{klmn}$ is valid, simply because for any $\phi \in \mathcal{L}$, the formulas $\Box \phi$ and $\Diamond \phi$ are equivalent to $\top$ and $\bot$ in $\mathcal{K}_i$, respectively. Therefore, $\mathcal{K}_i$ sees those instances of $ga_{klmn}$ as an implication with either $\bot$ in its antecedent or $\top$ in its succedent, which are clearly valid.

A similar claim holds for the fragments $\mathcal{L}_{\Box}$ and $\mathcal{L}_{\Diamond}$. More precisely, if we restrict $\mathcal{A}$ to $\Diamond$-free (resp. $\Box$-free) axioms in Table \ref{tableAxiom} with the restricting condition that $\mathcal{A}$ does not include $D_a$ (resp. $D_b$), then $L\mathsf{CK}_{\Box}+\mathcal{A}$ (resp. $L\mathsf{BLL}+\mathcal{A}$) is either $T$-free or $T$-full over $\mathcal{L}_{\Box}$ (resp. $\mathcal{L}_\Diamond$).

To name a non-example, consider the logic $\mathsf{CK}+D$, which is neither $T$-free nor $T$-full over $\mathcal{L}$. It is not $T$-full, as it cannot prove $T_a$. It is not $T$-free, as $D$ is not valid in the irreflexive node frame. The example may explain the terminology we use. We see $T$-free logics and their validity in one irreflexive node as a witness that not only the axioms $T_a$ and $T_b$ are not provable in the logic, but also the logic has no shadow of these axioms. Dually, $T$-full logics are the ones that embrace the full power of the $T$-axioms by proving both $T_a$ and $T_b$. In this sense, it is clear that $\mathsf{CK}+D$ stands somewhere in between $T$-freeness and $T$-fullness and although it cannot prove the $T$-axioms, it proves $D$ which is a shadow of the $T$-axioms.
\end{exam}

\begin{dfn}\label{Def: T-free T-full calculus}
Let $\mathfrak{L} \in \{\mathcal{L}, \mathcal{L}_{\Box}, \mathcal{L}_\Diamond\}$ be a language. A sequent calculus $G$ over the language $\mathfrak{L}$ is called \emph{$T$-free} over $\mathfrak{L}$ if it is strong over $\mathfrak{L}$ and valid in $\mathcal{K}_i$. $G$ is called \emph{$T$-full} over $\mathfrak{L}$ if it is strong over $\mathfrak{L}$, valid in $\mathcal{K}_r$ and:
\begin{itemize}
\item
if $\mathfrak{L}=\mathcal{L}$, then both the axioms $T_a$ and $T_b$ are provable in it;
\item
if $\mathfrak{L}=\mathcal{L}_\Box$, then the axiom $T_a$ is provable in it;
\item
if $\mathfrak{L}=\mathcal{L}_\Diamond$, then the axiom $T_b$ is provable in it.
\end{itemize}
\end{dfn}

\begin{exam}\label{ExamOfTfreeTfullSequent}
Let $\mathcal{A}$ and $\mathcal{B}$ be two finite sets of the axioms in Table \ref{tableAxiom} such that $\mathcal{A}$ does not include any of the axioms $T_a$, $T_b$, $D_a$, $D_b$, $D$, $den_{n,a}$ and $den_{n,b}$, for $n=0$, and $ga_{klmn}$ for $k=m=0$. Then, $\mathbf{CK}+\mathcal{A}$ and $\mathbf{IK}+\mathcal{A}$ are $T$-free and $\mathbf{CK}+\mathcal{B}\cup \{T_a, T_b\}$ and $\mathbf{IK}+\mathcal{B}\cup \{T_a, T_b\}$ are $T$-full over $\mathcal{L}$. For the fragments, if we restrict $\mathcal{A}$ to $\Diamond$-free (resp. $\Box$-free) axioms in Table \ref{tableAxiom} such that $\mathcal{A}$ does not include $D_a$ (resp. $D_b$), then $\mathbf{CK}_{\Box}+\mathcal{A}$ (resp. $\mathbf{BLL}+\mathcal{A}$) is either $T$-free or $T$-full over $\mathcal{L}_\Box$ (resp. $\mathcal{L}_{\Diamond}$). In all these claims, the proof for the validity in $\mathcal{K}_i$ and $\mathcal{K}_r$ is easy and similar to the ones in Example \ref{ExampleOfTfreeTfullLogic}. 
\end{exam}

\begin{rem}\label{RemarkOnTransportabilityOfTfree}
If $G$ is a sequent calculus for a logic $L$ and $G$ is strong, then $G$ is $T$-free ($T$-full) if{f} $L$ is $T$-free ($T$-full). Moreover, if the strong calculi $G$ and $H$ are equivalent, then $G$ is $T$-free ($T$-full) if{f} $H$ is $T$-free ($T$-full).
\end{rem}
Having the reduction machinery established, we now state and prove the reduction lemma. In the rest of this subsection and in Subsection \ref{SubsectionMain}, we work over the language $\mathcal{L}$ and hence by $T$-free ($T$-full), we always mean $T$-free ($T$-full) over $\mathcal{L}$. We will address the fragments later in Section \ref{SectionFragments}.

\begin{lem}\label{Reduction}
Let $G=\mathbf{CK}+\mathcal{C}$ be a $T$-free or a $T$-full calculus, where $\mathcal{C}$ is a finite set of constructive formulas. There is a feasible algorithm that reads a $G$-provable sequent $\Sigma, \{\neg q_j\}_{j=1}^m  \Rightarrow \bigvee_{i=1}^n p_i$, where $\Sigma$ is a multiset of modal Horn formulas and $p_i$'s and $q_j$'s are atomic formulas, and outputs a multiset $\Sigma'$ consisting of \emph{implicational} Horn formulas and a $G$-proof $\pi$ such that:
\begin{enumerate}
    \item \label{label1}
$\Sigma' \Rightarrow \bigvee_{i=1}^n p_i \vee \bigvee_{j=1}^m q_j $ is classically valid, and
\item \label{label2}
$G\vdash^{\pi} \Sigma \Rightarrow \bigwedge \Sigma'$.
\end{enumerate}
\end{lem}
\begin{proof}
Observe that any modality-free modal Horn formula $F$ is in the form $\bigwedge_{j_1=1}^{n_1} r_{1j_1} \to (\bigwedge_{j_2=1}^{n_2} r_{2j_2} \to (\bigwedge_{j_3=1}^{n_3} r_{3j_3} \to \cdots (\bigwedge_{j_m=1}^{n_m} r_{mj_m} \to s) \cdots ))$, where $r_{ij_k}$'s are atoms and $s$ is either an atom or $\bot$. Clearly, this formula is $\mathbf{LJ}$-equivalent to the implicational Horn formula $F'=\bigwedge_{i=1}^{m}\bigwedge_{l=1}^{n_i} r_{il} \to s$. The process of computing $F'$ and the $\mathbf{LJ}$-proof of the equivalence is feasible in $F$. Therefore, to prove the lemma, it is sufficient to provide a multiset $\Sigma'$ of modality-free modal Horn formulas with the mentioned properties. 

Now, we have to investigate two cases. If $G$ is $T$-free, set $\Sigma' \subseteq \Sigma$ as the set of the modality-free formulas in $\Sigma$. Note that the process of providing this multiset takes polynomial time in the size of $\Sigma$ and hence in the size of the input sequent $\Sigma, \{\neg q_j\}_{j=1}^m  \Rightarrow \bigvee_{i=1}^n p_i$. To show Condition \eqref{label1}, suppose otherwise, i.e., $\Sigma' \Rightarrow \bigvee_{i=1}^n p_i \vee \bigvee_{j=1}^m q_j $ is not classically valid. Then, there is a classical model $I$ such that $I \vDash \Sigma'$ but $I \nvDash p_i$ and $I \nvDash q_j$, for all $1 \leq i \leq n$ and $1 \leq j \leq m$. Consider the irreflexive node frame $\mathcal{K}_i=(\{w\}, =, \varnothing, \varnothing)$. Define the valuation function $V$ by $V(r)=\{w\}$ if $I(r)=1$ and $V(r)=\varnothing$ if $I(r)=0$, for any atom $r$. Clearly, for any modality-free formula $\phi$ we have $w \vDash \phi$ if and only if $I(\phi)=1$. Hence, $w \vDash \Sigma'$. To show that $w$ also satisfies the other elements of $\Sigma$, we prove a stronger claim that if $A$ is a modal Horn formula that is not modality-free, then $w \vDash A$. First, as the node $w$ is irreflexive, for any formula $B$, $w \vDash \Box B$, and $w \nVdash \Diamond B$. Second, if the modal Horn formula $A$ is not modality-free, it is either in the form $\Box^k C$, for some $k\geq 1$ or it is in the form $\bigwedge_{n_{i} \geq 0} \Diamond^{n_i} r_i \to C$, where either $n_i \geq 1$, for some $i$ or all $n_i=0$, which means that $C$ is a modal Horn formula that contains some modality. In either case we have $w \vDash A$. Hence, $w \vDash \Sigma$. Now, the atoms $p_i$ and $q_j$ are not satisfied in the model $(\mathcal{K}_i, V)$, while $G$ is valid in it, which is a contradiction. Hence, $\Sigma' \Rightarrow \bigvee_{i=1}^n p_i \vee \bigvee_{j=1}^m q_j $ is classically valid. Condition \eqref{label2} is clear as $\Sigma' \subseteq \Sigma$. The algorithm in this case is feasible as the time that it needs to compute $\pi$ is polynomial in the size of $\Sigma$ and hence in the size of the input sequent $\Sigma, \{\neg q_j\}_{j=1}^m  \Rightarrow \bigvee_{i=1}^n p_i$.

If $G$ is $T$-full, define the forgetful function $f: \mathcal{L} \to \mathcal{L}_p$ that deletes all the occurrences of $\Box$ and $\Diamond$ in the formula $A$ and outputs $A^f$. As an example $f(\Box (\Diamond p \wedge \Diamond q \to \Box r))$ is $p \wedge q \to r$. Now, define $\Sigma'$ as $\Sigma^f=\{A^f \mid A \in \Sigma\}$. Again, the process of computing $\Sigma'$ takes polynomial time in the size of $\Sigma$ and hence in the size of the input sequent $\Sigma, \{\neg q_j\}_{j=1}^m  \Rightarrow \bigvee_{i=1}^n p_i$. Now, for the sake of contradiction, suppose $\Sigma' \Rightarrow \bigvee_{i=1}^n p_i \vee \bigvee_{j=1}^m q_j$ is not classically valid. Then, there exists a classical model $I$ such that $I \vDash \Sigma'$ and $I \nvDash p_i$, and $I \nvDash q_j$, for any $1 \leq i \leq n$ and $1 \leq j \leq m$. Consider the reflexive node frame $\mathcal{K}_r=(\{w\}, =, \{(w, w)\}, \varnothing)$. Define the valuation function $V$ by $V(r)=\{w\}$ if $I(r)=1$ and $V(r)=\varnothing$ if $I(r)=0$, for any atom $r$. Similar to the previous case, for any modality-free formula $\phi$, we have $w \vDash \phi$ if and only if $I(\phi)=1$, and hence $w \vDash \Sigma'$. Moreover, since the model only consists of one reflexive world $w$, for any formula $\psi$ we have 
\begin{center}
$w \vDash \psi$ \; if{f}\; $w \vDash \Box^k \psi$\; if{f}\; $w \vDash \Diamond^l \psi$, \; for any $k, l \geq 0$. $\quad (*)$
\end{center}
Therefore, it is clear that $w$ thinks that any formula $\phi$ is equivalent to $\phi^f$. Hence, $w \vDash \Sigma$. 
Now, notice that $w \nvDash p_i$ and $w \nvDash q_j$, for any $1 \leq i \leq n$ and $1 \leq j \leq m$. Since $G$ is valid in the model $(\mathcal{K}_r, V)$, we reach a contradiction. Hence, $\Sigma' \Rightarrow \bigvee_{i=1}^n p_i \vee \bigvee_{j=1}^m q_j $ is classically valid. For Condition \eqref{label2}, we first provide a feasible algorithm that reads a modal Horn formula $A$ and outputs a $G$-proof $\pi_A$ for $A \Rightarrow A^f$. For that purpose, consider the meta-sequents $\Pi, \Box r \Rightarrow r$ and $\Pi, r \Rightarrow \Diamond r$, which can be also read as two rules with no premises. As they have the form mentioned in Definition \ref{Def:RulesR}, they are in $\mathfrak{R}$. Now, note that both of these rules are provable in $G$, as $G$ is $T$-full.
As $G$ is strong, by Corollary \ref{EquivProvableAndFeasiblyProvable}, $G$ feasibly proves both the meta-sequents. Therefore, there are feasible functions $g$ and $h$ such that $g(C)$ and $h(C)$ are the $G$-proofs of $ \Box C \Rightarrow C$ and $ C \Rightarrow \Diamond C$, respectively.
Now, let us explain the algorithm that computes $\pi_A$ from $A$, by using a recursion on the structure of $A$. Call the time of the algorithm $T(A)$. If $A$ is an atom or $\bot$, then the algorithm outputs $A \Rightarrow A$ as an instance of the identity axiom. In this case, $T(A) \leq O(1)$. If $A= \Box B$, consider the proof
\begin{center}
\begin{tabular}{c c c}
 \AxiomC{$g(B)$}
 \noLine\UnaryInfC{$\Box B \Rightarrow B$}
 \AxiomC{$\pi_B$}
   \noLine\UnaryInfC{$B \Rightarrow B^f$}
 \BinaryInfC{$\Box B \Rightarrow B^f$}
 \DisplayProof
\end{tabular}
\end{center}
where $\pi_B$ is provided by the recursive step. Note that, as the time to compute $g(B)$ is polynomial in $|B|$, we have $T(\Box B) \leq T(B)+|B|^{O(1)}$. If $A=\bigwedge_{i=1}^{k} \Diamond^{n_i}r_i \to B$, where $B$ is a modal Horn formula, by the recursive step, $G\vdash^{\pi_B} B \Rightarrow B^f$. The algorithm first provides a proof for $r_i \Rightarrow \Diamond^{n_i}r_i$ for each $1\leq i\leq k$, by using $n_i-1$ many cuts on the proofs $h(\Diamond^m r_i)$, for any $0 \leq m \leq n_i-1$. As $|\Diamond^m r_i| \leq |A|$ and $n_i \leq |A|$, this part takes $|A|^{O(1)}$ steps. Using the rule $(L \wedge)$, the algorithm continues to provide a proof for $\bigwedge_{i=1}^k r_i \Rightarrow \Diamond^{n_i} r_i$ for each $1\leq i\leq k$, and then, using $(R \wedge)$, a proof for $\bigwedge_{i=1}^k r_i \Rightarrow \bigwedge_{i=1}^k \Diamond^{n_i} r_i$. The time this part takes is trivially $|A|^{O(1)}$. Finally, using the proof for $\bigwedge_{i=1}^k r_i \Rightarrow \bigwedge_{i=1}^k \Diamond^{n_i} r_i$ and $\pi_B$, the algorithm provides a proof for $\bigwedge_{i=1}^k \Diamond^{n_i} r_i \to B \Rightarrow \bigwedge_{i=1}^k r_i \to B^f$. Using the bounds mentioned for the time of each part, we can easily see that $T(A) \leq T(B)+|A|^{O(1)}$. Using all these cases, it is clear that the algorithm provide a proof $\pi_A$ for $A \Rightarrow A^f$ in time polynomial in $|A|$. Therefore, using $\pi_A$'s, for all $A \in \Sigma$ and some weakening rules followed by some applications of $(R\wedge)$, we can provide a $G$-proof for $\Sigma \Rightarrow \bigwedge \Sigma'$, feasibility in the size of $\Sigma$ and hence in the size of the input sequent $\Sigma, \{\neg q_j\}_{j=1}^{m} \Rightarrow \{p_i\}_{i=1}^n$.
\end{proof}

\subsection{The Main Theorem}\label{SubsectionMain}
In this subsection, we put together the two ingredients we developed in this section to prove the main result of the paper. First, let us generalize the feasible admissibility of Visser's rules that we mentioned before to also allow the Harrop formulas in the antecedent of the sequents.

\begin{dfn}\label{DfnFeasibleDPHarropProperty}
A sequent calculus $G$ has the \emph{feasible Visser-Harrop property}, if there exists a feasible function $f$ that reads a $G$-proof $\pi$ of \[\Gamma,  \{ A_i \to B_i \}_{i \in I} \Rightarrow C \vee D,\] where $\Gamma$ is a multiset of Harrop formulas, and $I$ is a (possibly empty) finite index set, and outputs a $G$-proof for one of the following sequents:
\small \begin{center}
\begin{tabular}{c c c}
$\Gamma,  \{ A_i \to B_i \}_{i \in I} \Rightarrow C\;$ or & $\Gamma,  \{ A_i \to B_i \}_{i \in I} \Rightarrow D\;$ or
& $\Gamma,  \{ A_i \to B_i \}_{i \in I} \Rightarrow A_i$,\\
\end{tabular}
\end{center}
\normalsize for some $i \in I$. In the case that $\Gamma=I=\emptyset$, we call it the \emph{feasible disjunction property} of $G$. The \emph{Visser-Harrop property} or \emph{disjunction property} is defined in the same way if we drop the feasibility condition.
\end{dfn}

\begin{thm} \label{FeasibleHarropForIK}
Let $G=\mathbf{CK}+\mathcal{C}$ be either a $T$-free or a $T$-full sequent calculus, where $\mathcal{C}$ is a finite set of constructive formulas. Then, $G$ has the feasible Visser-Harrop property.
\end{thm}

\begin{proof}
Suppose a proof $\pi$ of the sequent $ \Gamma, \{A_i \to B_i\}_{i \in I} \Rightarrow C \vee D$ is given, where $\Gamma$ is a multiset of Harrop formulas. By Theorem \ref{MainTheorem}, feasibly in $\pi$ we can get a multiset $\Sigma_{\pi}$ and a $G$-proof $\sigma_{\pi}$, where the former consists of modal Horn formulas constructed only from angled atoms such that 
\begin{center}
    $G\vdash \Sigma_{\pi}, \Gamma^t, \{(A_i \to B_i)^t\}_{i \in I} \Rightarrow (C \vee D)^t$ \quad \text{and}\quad $G\vdash^{\sigma_{\pi}} (\, \Rightarrow \bigwedge \Sigma_{\pi}^s)$ .
\end{center}
By Definition \ref{Translation}, $\mathbf{CK} \vdash (C \vee D)^t \Rightarrow C^t \vee D^t$. Using the $\mathbf{CK}$-provable sequents $A_i^t \to \bot \Rightarrow A_i^t \to B_i^t$ for each $i \in I$, and hence $\langle A_i \to B_i \rangle, A_i^t \to \bot \Rightarrow (A_i \to B_i)^t$, we have 
\begin{center}
$G\vdash \Sigma_{\pi}, \Gamma^t, \{\langle A_i \to B_i \rangle\}_{i \in I}, \{ \neg A_i^t\}_{i \in I} \Rightarrow C^t \vee D^t$ ,
\end{center}
which by Lemma \ref{TranslationAndAtoms} implies 
\begin{center}
$G\vdash \Sigma_{\pi}, \Gamma^t, \{\langle A_i \to B_i \rangle\}_{i \in I}, \{ \neg \langle A_i \rangle\}_{i \in I}  \Rightarrow \langle C \rangle \vee \langle D \rangle$ . 
\end{center}
By Lemma \ref{LemHarrop}, for any $\gamma \in \Gamma$, feasibly in $\gamma$ and hence in $\pi$, we can find a multiset $\Lambda_\gamma$ and a $\mathbf{CK}$-proof $\sigma_{\gamma}$, where $\Lambda_\gamma$ consists of modal Horn formulas built from angled atoms such that $\mathbf{CK} \vdash \Lambda_\gamma \Rightarrow \gamma^t$ and $\mathbf{CK} \vdash^{\sigma_{\gamma}} \bigwedge \Lambda_\gamma^s \Leftrightarrow \gamma$. Take $\Lambda_\Gamma=\bigcup_{\gamma \in \Gamma} \Lambda_\gamma$ and $\sigma_{\Gamma}$ constructed from $\sigma_\gamma$'s such that $\mathbf{CK} \vdash \Lambda_\Gamma \Rightarrow \bigwedge \Gamma^t$ and $\mathbf{CK} \vdash^{\sigma_{\Gamma}} \bigwedge \Lambda_\Gamma^s \Leftrightarrow \bigwedge \Gamma$. It is easy to see that $\Lambda_\Gamma$ and $\sigma_\Gamma$ can be obtained feasibly from $\Lambda_\gamma$ and $\sigma_\gamma$ for each $\gamma \in \Gamma$ and hence feasibly in $\pi$. Therefore, $G$ proves the sequent
\begin{center}
$S=\Sigma_{\pi}, \Lambda_\Gamma, \{\langle A_i \to B_i \rangle\}_{i \in I}, \{ \neg \langle A_i \rangle\}_{i \in I}  \Rightarrow \langle C \rangle \vee \langle D \rangle$. 
\end{center}
By Lemma \ref{Reduction}, feasibly in $S$ and hence in $\pi$, we can provide a multiset $\Omega$ and a $G$-proof $\alpha$, where $\Omega$ consists of implicational Horn formulas, and $\Omega \Rightarrow  \bigvee_{i \in I} \langle A_i \rangle \vee \langle C \rangle \vee \langle D \rangle$ is classically valid
and
$G\vdash^{\alpha} \Sigma_{\pi}, \Lambda_{\Gamma}, \{\langle A_i \to B_i \rangle\}_{i \in I} \Rightarrow \bigwedge \Omega$. Take $U= (\Omega \Rightarrow  \bigvee_{i \in I} \langle A_i \rangle \vee \langle C \rangle \vee \langle D \rangle)$. We can use Theorem \ref{UnitPropagation} to feasibly in $U$ and hence in $\pi$ find $\tau$ and $1 \leq i \leq n$, such that 
\begin{center}
    $\mathbf{LJ} \vdash^{\tau} \Omega \Rightarrow \langle A_{i} \rangle$ \;\text{or}\; $\mathbf{LJ} \vdash^{\tau} \Omega \Rightarrow \langle C \rangle$ \;\text{or}\; $\mathbf{LJ} \vdash^{\tau} \Omega \Rightarrow \langle D \rangle$.
\end{center}
We only address the case where $\mathbf{LJ} \vdash^{\tau} \Omega \Rightarrow \langle C \rangle$. The rest are similar. Since $G$ extends $\mathbf{LJ}$, we also have $G\vdash^{\tau} \Omega \Rightarrow \langle C \rangle$. Using the fact that $G\vdash^{\alpha} \Sigma_{\pi}, \Lambda_{\Gamma}, \{\langle A_i \to B_i \rangle\}_{i \in I} \Rightarrow \bigwedge \Omega$, we get a proof $\beta$ feasible in $\pi$ such that $G\vdash^{\beta} \Sigma_{\pi} , \Lambda_{\Gamma},  \{\langle A_i \to B_i \rangle\}_{i \in I} \Rightarrow \langle C \rangle$. The sequent will be provable for any substitution, specially the standard substitution. Notice that the process of substitution is feasible. Hence, using the $\mathbf{CK}$-proof $\sigma_{\Gamma}$ for $\bigwedge \Lambda_{\Gamma}^s \Leftrightarrow \bigwedge \Gamma$ and the $G$-proof $\sigma_{\pi}$ for $\Rightarrow \bigwedge \Sigma_{\pi}^s$, both feasible in $\pi$, we reach a proof for $\Gamma,  \{ A_i \to B_i\}_{i \in I} \Rightarrow C$, feasibly in $\pi$.
\end{proof}

The previous theorem proves the main result for $T$-free or $T$-full calculi in the form $\mathbf{CK}+\mathcal{C}$
, where $\mathcal{C}$ is a finite set of constructive formulas. The next corollary generalizes the result to its ultimate form:
\begin{cor}\label{thmKolli}
Let $G$ be either a $T$-free or a $T$-full constructive sequent calculus. Then, $G$ has the feasible Visser-Harrop property.
\end{cor}
\begin{proof}
It is an immediate consequence of Theorem \ref{thm: p-simulation of admissibly strong} and Theorem \ref{FeasibleHarropForIK}.
\end{proof}

Corollary \ref{thmKolli} has two types of applications. First, in its positive form, it proves the feasible Visser-Harrop property for any known $T$-free or $T$-full constructive sequent calculus. As a consequence, it also shows that their corresponding logics have the Visser-Harrop property and hence admit Visser's rules. 
Notice that as the constructive rules have a general form and $T$-freeness and $T$-fullness are quite weak conditions, there are many calculi to which Corollary \ref{thmKolli} is applicable. It is also worth mentioning that even in the cases where the feasibility of Visser's rules is not of interest, Corollary \ref{thmKolli} and the machinery around it are still useful. Usually, to provide a proof-theoretic proof for admissibility of a rule in a logic, one must design a well-behaved proof system for the logic in which the cut rule is admissible. This is unfortunately not possible for many logical systems. However, Corollary \ref{thmKolli} deals with the calculi with the explicit cut rule in them without any need to eliminate it. As a consequence, if a logic is presented by a non-well-behaved proof system (e.g., with many initial sequents) in which we cannot eliminate the cut, then Corollary \ref{thmKolli} is still applicable as long as the axioms and the rules of the system are constructive.

\begin{cor} \label{ConcreteApplication}
(Positive application) Let $\mathcal{A}$ and $\mathcal{B}$ be two finite sets of the axioms in Table \ref{tableAxiom} such that $\mathcal{A}$ does not include any of the axioms $T_a$, $T_b$, $D_a$, $D_b$, $D$, $den_{n,a}$ and $den_{n,b}$, for $n=0$, and $ga_{klmn}$, for $k=m=0$. Then, the sequent calculi $\mathbf{CK}+\mathcal{A}, \mathbf{IK}+\mathcal{A}, \mathbf{CK}+\mathcal{B}\cup\{T_a, T_b\}$, and $\mathbf{IK}+\mathcal{B}\cup\{T_a, T_b\}$  enjoy the feasible Visser-Harrop property and hence feasible disjunction property. As a consequence, the logic of any of these systems has Visser-Harrop property and hence admits all Visser's rules. 
\end{cor}
\begin{proof}
By Example \ref{ExamOfTfreeTfullSequent}, and the fact that $\mathbf{CK}$, $\mathbf{IK}$, and  all the axioms in Table \ref{tableAxiom} are constructive, the claim is a consequence of Corollary \ref{thmKolli}.
\end{proof}

\begin{cor}
The sequent calculi $\mathbf{CK}X$ and $\mathbf{IK}X$, for any $X \subseteq \{T, B, 4, 5\}$, specially  $\mathbf{CS4}$, $\mathbf{CS5}$, $\mathbf{IS4}$ and  $\mathbf{IS5}$ (also known as $\mathbf{MIPC}$) 
enjoy the feasible Visser-Harrop property and hence the feasible disjunction property. As a consequence, the logic of any of these systems has Visser-Harrop property and hence admits all Visser's rules. 
\end{cor}
\begin{proof}
The claim is a direct consequence of Corollary \ref{ConcreteApplication}.
\end{proof}

As the negative application of Corollary \ref{thmKolli}, we have:

\begin{cor}(Negative application) \label{NegAppMain}
Let $L$ be a $T$-free or a $T$-full logic. If there is at least one Visser's rule that is not admissible in $L$, then $L$ does not have a constructive sequent calculus.
\end{cor}
\begin{proof}
Let $G$ be a constructive sequent calculus for $L$. By Lemma \ref{LemAdStrong}, all the rules and axioms of $\mathbf{CK}$ are admissible in $L$. Define $H$ as $G+\mathbf{CK}$. Therefore, $H$ is also a sequent calculus for $L$. The system $H$ is clearly strong and constructive. Moreover, by Remark \ref{RemarkOnTransportabilityOfTfree}, $H$ is either $T$-free or $T$-full, as $L$ is $T$-free or $T$-full. Therefore, by Corollary \ref{thmKolli}, $H$ and hence $L$ admits all Visser's rules which is a contradiction.
\end{proof}
A good source of intuitionistic modal logics that do not admit all Visser's rules are the modal versions of the proper intermediate extensions of $\mathsf{IPC}$. First, let us recall the following propositional characterization of the logics in which all Visser's rules are admissible:
\begin{thm}
\label{ThmRose}\cite{IemhoffAdAdAdAd}
$\mathsf{IPC}$ is the only intermediate logic that admits all Visser's rules.
\end{thm}

Now, we use Corollary \ref{NegAppMain} on a vast range of modal intermediate logics.

\begin{cor} \label{NegAppSpecialMain}
Let $L \neq \mathsf{IPC}$ be an intermediate logic and $\mathcal{A}$ and $\mathcal{B}$ be two finite sets of axioms in Table \ref{tableAxiom} such that $\mathcal{A}$ does not include any of the axioms $T_a$, $T_b$, $D_a$, $D_b$, $D$, $den_{n,a}$ and $den_{n,b}$, for $n=0$, and $ga_{klmn}$ for $k=m=0$. Then, none of the logics $L\mathsf{CK}+\mathcal{A}$, $L\mathsf{IK}+\mathcal{A}$, $L\mathsf{CK}+\mathcal{B} \cup \{T_a, T_b\}$ and $L\mathsf{IK}+\mathcal{B} \cup \{T_a, T_b\}$ have a constructive calculus.
\end{cor}
\begin{proof}
We only prove the cases $L\mathsf{CK}+\mathcal{A}$ and $L\mathsf{CK}+\mathcal{B} \cup \{T_a, T_b\}$. The other two are similar. First, note that $L\mathsf{CK}+\mathcal{A}$ and $L\mathsf{CK}+\mathcal{B} \cup \{T_a, T_b\}$ are conservative over $L$: let $f$ be the forgetful translation that deletes all the occurrences of $\Box$ and $\Diamond$ in a modal formula. To prove the claim, it suffices to read a proof $\pi$ of a propositional formula $A$ in each of these logics and apply $f$ to $\pi$. It is easy to see that $f(\pi)$ is a proof in $L$. The only non-trivial part is showing that $\mathsf{IPC} \vdash B^f$, for any axiom $B$ in Table \ref{tableAxiom} which is easy by the form of the axioms. This completes the proof of the conservativity. By Theorem \ref{ThmRose}, there is a Visser's rule not admissible in $L$. Therefore, it is not admissible in $L\mathsf{CK}+\mathcal{A}$ and $L\mathsf{CK}+\mathcal{B} \cup \{T_a, T_b\}$, either. Finally, note that $L\mathsf{CK}+\mathcal{A}$ is $T$-free and $L\mathsf{CK}+\mathcal{B} \cup \{T_a, T_b\}$ is $T$-full, by Example \ref{ExampleOfTfreeTfullLogic}. Hence, by Corollary \ref{NegAppMain}, we get the result.
\end{proof}

Let us mention that the logics where Visser's rules are not admissible are not limited to the modal versions of the intermediate logics. There are also logics with non-trivial modal disjunctions, such as the logic $\mathsf{CK}+\Box p \vee \Box \neg \Box p$, that lack the disjunction property and hence do not admit all Visser's rules. 

\section{Fragments}\label{SectionFragments}
In this section, we will prove the analogue of Corollary \ref{thmKolli} for the fragments $\mathcal{L_{\Box}}, \mathcal{L_{\Diamond}},$ and $\mathcal{L}_p$. Our main technique is reducing the claim to Corollary \ref{thmKolli}, by changing the language to the full language $\mathcal{L}$ in an appropriate manner. To explain how, we need the following definition:

\begin{dfn}\label{dfn: feasible conservative}
Let $G$ and $H$ be two sequent calculi over the languages $\mathcal{L}_1$ and $\mathcal{L}_2$, respectively, such that $\mathcal{L}_1 \subseteq \mathcal{L}_2$ and any proof in $G$ is also a proof in $H$. We say that $H$ is \emph{feasibly conservative} over $G$ if there exists a feasible function $f$ that reads a proof $\pi$, such that $H \vdash^{\pi} S$ implies $G \vdash^{f(\pi)} S$, for any $H$-proof $\pi$ and any sequent $S$ over the language $\mathcal{L}_1$. 
\end{dfn}
To prove Corollary \ref{thmKolli} for the fragments, we first extend the given calculus $G$ defined over a fragment of $\mathcal{L}$ to a calculus $H$ over the extended language $\mathcal{L}$ in a way that $H$ is feasibly conservative over $G$ and if $G$ is constructive, $T$-free or $T$-full, so is $H$. This way,  we move from $G$ to $H$ to apply Corollary \ref{thmKolli} and as $H$ is feasibly conservative over $G$, we can come back to the original calculus $G$. We employ this strategy in the next three subsections.

\subsection{$\Diamond$-free Fragment} \label{SubsectionDiamond}
Define the forgetful function $f_{c}: \mathcal{L} \to \mathcal{L}_{\Box}$, for any $c \in \{i, r\}$ as follows: $f_c(p)=p$, for any atom $p$ (including $\bot$ and $\top$); $f_c(A \circ B)= f_c(A) \circ f_c(B)$, for $\circ \in \{\wedge, \vee, \to\}$; $f_c(\Box A)=\Box f_c(A)$; and $f_{r}(\Diamond A)=f_r(A)$ and $f_{i}(\Diamond A)=\bot$. The functions $f_{r}$ and $f_{i}$ are clearly feasible and can be extended to multisets, sequents and proofs in a natural way. Let $G$ be a sequent calculus over $\mathcal{L}_{\Box}$. Define $G_{i}=\G+\{K_{\Diamond}\}$ and $\G_{r}=\G+\{K_{\Diamond}, T_b\}$ over $\mathcal{L}$. The following lemma connects these systems to $G$ via the corresponding translations.

\begin{lem}\label{lemG*1}
Let $G$ be a strong sequent calculus over $\mathcal{L}_{\Box}$. Then:
\begin{description}
\item[$(i)$]
There is a feasible algorithm that reads a $\G_{i}$-proof of a sequent $S$ and 
provides a $G$-proof of $f_i(S)$. Hence, $\G_i$ is feasibly conservative over $G$.
\item[$(ii)$]
If $G$ proves $T_a$, then there exists a feasible algorithm that reads a $\G_{r}$-proof of a sequent $S$ and provides a $G$-proof for $f_r(S)$. Consequently, $\G_r$ is feasibly conservative over $G$.
\end{description}
\end{lem}

\begin{proof}
We only prove $(ii)$, the part $(i)$ is similar. For $(ii)$, we provide an algorithm $g$ that reads a $\G_r$-proof $\pi$ of $S$ over $\mathcal{L}$ and returns a $G$-proof of $f_{r}(S)$ over $\mathcal{L}_\Box$. Denoting the time to compute $g(\pi)$ by $T_g(\pi)$, we also show that $T_g(\pi) \leq |\pi|^{O(1)}$. To define $g$, we use recursion on the structure of $\pi$.
If $S$ is an instance of an axiom in $\G_{r}$, then it is either an instance of an axiom in $G$ or an instance of $T_b$. In the former case, assume that $S$ is the instance $\Gamma(\overline{\phi}) \Rightarrow \Delta(\overline{\phi})$ of the axiom $\Gamma(\overline{p}) \Rightarrow \Delta(\overline{p})$ in $G$. Since, $\Gamma(\overline{p}) \cup \Delta(\overline{p}) \subseteq \mathcal{L}_{\Box}$ and $f_r$ preserves every connective in $\mathcal{L}_{\Box}$, we have $f_r(S)=\Gamma(\overline{f_r(\phi})) \Rightarrow \Delta(\overline{f_r(\phi}))$. Therefore, $f_r(S)$ is an instance of the same axiom of $G$. Hence, it is enough to define $g(\pi)$ as the axiom $f_r(S)$. Note that as $f_r$ is a polynomial time computable function, in this case, $T_g(\pi)$ is polynomial in $|S|$ and hence polynomial in $|\pi|$. If $S$ is an instance of $T_b$, then it is of the form $S=(\, \Rightarrow A \to \Diamond A)$ and $f_r(S)=(\, \Rightarrow f_r(A)\to f_r(A))$. 
Now, consider the meta-sequents $(\Pi \Rightarrow p \to p)$, which can be also read as a rule with no premises. As it has the form mentioned in Definition \ref{Def:RulesR}, it is in $\mathfrak{R}$. Now, note that this rule is provable in $G$, as $G$ is strong. Then, as $(\Pi \Rightarrow p \to p)$ is a rule in $\mathfrak{R}$, by Corollary \ref{EquivProvableAndFeasiblyProvable}, $G$ feasibly proves $(\Pi \Rightarrow p \to p)$. Thus, there is a feasible function $h$ such that 
$h(T)$ is a $G$-proof of $T=(\, \Rightarrow B \to B)$, for any $B \in \mathcal{L}_{\Box}$. Therefore,
$h(f_r(S))$ is a $G$-proof of $(\, \Rightarrow f_r(A) \to f_r(A))$. Define $g(\pi)=h(f_r(S))$. Note that in this case, as both $h$ and $f_r$ are feasible, $T_g(\pi)$ is polynomial in $|S|$ and hence polynomial in $|\pi|$.

If the last rule in $\pi$ is in $G$ with the subproofs $\pi_1, \ldots, \pi_n$ for the premises $S_1, \ldots, S_n$, then by recursion, we have $G \vdash^{g(\pi_i)} f_r(S_i)$, for any $1 \leq i \leq n$. Similar to the case of the axioms, as $f_r$ commutes with the rules of $G$, applying the same rule in $G$ to $g(\pi_i)$'s will result in $f_r(S)$. Therefore, it is enough to define $g(\pi)$ as the application of the rule on $g(\pi_i)$'s. Here, $T_g(\pi)$ is bounded by $\sum_{i=1}^n T_g(\pi_i)$ plus the additional step of implementing the last rule. The latter takes at most $O(|f_r(S)|)$ many steps, which is polynomial in $|S|$ and hence in $|\pi|$. Therefore, $T_g(\pi) \leq \sum_{i=1}^n T_g(\pi_i)+|\pi|^{O(1)}$.

If the last rule of $\pi$ is the rule $(K_{\Diamond})$, 
then the premise of the rule is in the form $S'=(\Gamma, A \Rightarrow B)$ and $S$ is of the form $\Box \Gamma, \Diamond A \Rightarrow \Diamond B$. Note that $f_{r}(S')=(f_r(\Gamma), f_r(A) \Rightarrow f_r(B))$ and $f_r(S)=(\Box f_r(\Gamma), f_r(A) \Rightarrow f_r(B))$. Let $\pi'$ be the subproof of $\pi$ with the conclusion $S'$. By recursion, $G \vdash^{g(\pi')} f_r(S')$.
As $G$ is strong and proves the axiom $T_a$, it also proves the rule 
\begin{center}
\begin{tabular}{c c c}
 \AxiomC{$\Pi, p \Rightarrow q$}
\UnaryInfC{$\Pi, \Box p \Rightarrow q$}
 \DisplayProof
 \end{tabular}
\end{center}
Thus, as this rule is in $\mathfrak{R}$, by Corollary \ref{EquivProvableAndFeasiblyProvable}, $G$ feasibly proves this rule. Therefore, there is a feasible function $h$ such that $h(R', R)$ is a $G$-proof of $R=(\Sigma, \Box C \Rightarrow D)$ from $R'=(\Sigma, C \Rightarrow D)$, for any $\Sigma \cup \{C,D\} \subseteq \mathcal{L}_\Box$. Using $h$ for $\parallel \Gamma \parallel$ many times, we reach a $G$-proof of $f_r(S)=(\Box f_r(\Gamma), f_r(A) \Rightarrow f_r(B))$ from $f_r(S')=(f_r(\Gamma), f_r(A) \Rightarrow f_r(B))$. Adding this proof to the end of $g(\pi')$, we get a $G$-proof, called $g(\pi)$, for $f_r(S)$. In this case, as $f_r$ and $h$ are feasible, it is easy to see that $T_g(\pi) \leq T_g(\pi')+|\pi|^{O(1)}$.  This completes the recursive construction of $g(\pi)$.  
Finally, to show $T_g(\pi) \leq |\pi|^{O(1)}$, note that in the case of the axioms $T_g(\pi) \leq |\pi|^{O(1)}$ and in the case where $\pi$ is an application of a rule on $\pi_i$'s, $T_g(\pi) \leq \sum_{i=1}^n T_g(\pi_i)+|\pi|^{O(1)}$. Employing these two upper bounds, by induction on the structure of $\pi$ we can easily show that $T_g(\pi) \leq |\pi|^{O(1)}$.
\end{proof}

\begin{thm} \label{FeasibleHarropForDiamonFree}
Let $G$ be a $T$-free or a $T$-full constructive sequent calculus over the language $\mathcal{L}_{\Box}$.
Then, $G$ has the feasible Visser-Harrop property.
\end{thm}
\begin{proof}
First, by Corollary \ref{thm: p-simulation of admissibly strong}, $G$ is pd-equivalent to $\mathbf{CK}_{\Box}+\mathcal{C}$, for a finite set $\mathcal{C}$ of constructive formulas. Therefore, by Remark \ref{RemarkOnTransportabilityOfTfree}, it is enough to prove the claim for $G=\mathbf{CK}_{\Box}+\mathcal{C}$.  Hence, $G$ has all the rules of $\mathbf{CK}_{\Box}$ as its primitive rules and $G_i$ and $G_r$ are both strong over $\mathcal{L}$. Second, we show that if $G$ is $T$-free (resp. $T$-full) over $\mathcal{L}_{\Box}$, then $\G_{i}$ (resp. $\G_{r}$) is $T$-free (resp. $T$-full) over $\mathcal{L}$. We only prove the $T$-full case. The other is similar.
As $G$ is $T$-full, it is valid in the reflexive node frame, $\mathcal{K}_r$. To show the same property for $G_{r}$, assume that $S$ is provable in $\G_{r}$. As $G$ is $T$-full, it is strong over $\mathcal{L}_{\Box}$ and proves $T_a$. Thus, by Lemma \ref{lemG*1}, $f_{r}(S)$ is provable in $G$. Therefore, $f_{r}(S)$ is valid in $\mathcal{K}_r$. However, $\mathcal{K}_r$ reads $\Diamond A$ as $A$, for any $A \in \mathcal{L}$. Hence, for any $B \in \mathcal{L}$, the formula $B$ is valid in $\mathcal{K}_r$ if and only if $f_{r}(B)$ is valid there. Therefore, $S$ is also valid in $\mathcal{K}_r$.\\
Now, we showed that if $G$ is $T$-free (resp. $T$-full) over $\mathcal{L}_\Box$,
then $\G_{i}$ (resp. $\G_{r}$) is $T$-free (resp. $T$-full) over $\mathcal{L}$ and as both $\G_{i}$ and $\G_{r}$ are clearly constructive, by Corollary \ref{thmKolli}, $\G_{i}$ (resp. $\G_{r}$) has the feasible Visser-Harrop property. In the following, we show that $G$ also has the feasible Visser-Harrop property. Assume $G$ is $T$-free and $\pi$ is a $G$-proof of $\Gamma, \{ A_j \to B_j \}_{j \in J} \Rightarrow C \vee D$, where $\Gamma$ is a multiset of Harrop formulas. As  $\G_{i}$ extends $G$, the proof $\pi$ is also a proof in the calculus $\G_{i}$. Using the feasible Visser-Harrop property for $\G_{i}$, we can feasibly extract a $\G_{i}$-proof for either 
\small \begin{center}
\begin{tabular}{c c c}
$\Gamma, \{ A_j \to B_j \}_{j \in J} \Rightarrow C\;$ or & $\Gamma,  \{ A_j \to B_j \}_{j \in J} \Rightarrow D\;$ or
& $\Gamma, \{ A_j \to B_j \}_{j \in J} \Rightarrow A_j$,\\
\end{tabular}
\end{center}
\normalsize for some $j \in J$. By Lemma \ref{lemG*1}, as $G_{i}$ is feasibly conservative over $G$ and the three sequents are over $\mathcal{L}_{\Box}$, we feasibly reach a $G$-proof for one of them.
\end{proof}

\begin{cor} \label{ConcreteApplicationDiamonFree}
(Positive application) Let $\mathcal{A}$ be a finite set of $\Diamond$-free axioms in Table \ref{tableAxiom} such that $\mathcal{A}$ does not include $D_a$. Then, the sequent calculus $\mathbf{CK}_{\Box}+\mathcal{A}$, especially $\mathbf{CK}_\Box X$, for any $X \subseteq \{T, 4\}$, enjoys the feasible Visser-Harrop property and hence the feasible disjunction property. Consequently, the logic of any of these systems has the Visser-Harrop property.
\end{cor}

\begin{proof}
By Example \ref{ExamOfTfreeTfullSequent}, we know that $\mathbf{CK}_{\Box}+\mathcal{A}$ is either $T$-free or $T$-full over $\mathcal{L}_\Box$. As it is clearly constructive, using Theorem \ref{FeasibleHarropForDiamonFree}, we get the result. The claim for $\mathbf{CK}_\Box X$ is a direct consequence of the first part.
\end{proof}

\begin{cor}(Negative application)
Let $L$ be a $T$-free or a $T$-full logic over $\mathcal{L}_{\Box}$. If there is at least one Visser's rule that is not admissible in $L$, then $L$ does not have a constructive sequent calculus.
\end{cor}
\begin{proof}
Similar to the proof of Corollary \ref{NegAppMain}, using $\mathbf{CK}_{\Box}$ instead of $\mathbf{CK}$.
\end{proof}

\begin{cor}
Let $L \neq \mathsf{IPC}$ be an intermediate logic and $\mathcal{A}$ be a finite set of $\Diamond$-free axioms in Table \ref{tableAxiom} such that $\mathcal{A}$ does not include $D_a$. Then, $L\mathsf{CK}_{\Box}+\mathcal{A}$ does not have a constructive sequent calculus.
\end{cor}

\begin{proof}
The proof is similar to the proof of Corollary \ref{NegAppSpecialMain}.
\end{proof}

\subsection{$\Box$-free Fragment} \label{SubsectionBox}
Define the forgetful function $g: \mathcal{L}\to \mathcal{L}_{\Diamond}$ as follows: $g(p)=p$, for any atom $p$ (including $\bot$ and $\top$); $g(A \circ B)= g(A) \circ g(B)$, for $\circ \in \{\wedge, \vee, \to\}$; $g(\Diamond A)=\Diamond g(A)$ and $g(\Box A)=g(A)$. Clearly, $g$ is polynomial time computable. Let $G$ be a calculus over $\mathcal{L}_{\Diamond}$. Define $\G_{i}=G+ \{K_{\Box}, K_{\Diamond}\}$ and $\G_{r}=G+ \{K_{\Box}, K_{\Diamond}, T_a\}$ over $\mathcal{L}$. The following connects these systems to $G$ via the translation $g$.
\begin{lem}\label{lemG*2}
Let $G$ be a strong sequent calculus over $\mathcal{L}_{\Diamond}$. Then, there exists a feasible algorithm that reads a $\G_{r}$-proof of a sequent $S$ and provides a $G$-proof for $g(S)$. Consequently, $G_i$ and $\G_r$ are feasibly conservative over $G$.
\end{lem}
\begin{proof}
The proof is similar to the proof of Lemma \ref{lemG*1}. 
We provide an algorithm $f$ that reads a $\G_{r}$-proof $\pi$ of $S$ and returns a $G$-proof of $g(S)$. Denoting the time to compute $f(\pi)$ by $T_f(\pi)$, we will also show that $T_f(\pi) \leq |\pi|^{O(1)}$. To define $f$, we use recursion on the structure of $\pi$. If $S$ is an instance of an axiom of $G$, as $g$ commutes with all the connectives in $\mathcal{L}_{\Diamond}$, similar to the argument in the proof of Lemma \ref{lemG*1}, it is easy to see that $g(S)$ is also an instance of the same axiom. Hence, it is enough to define $f(\pi)=g(S)$. Notice $T_f(\pi) \leq |\pi|^{O(1)}$, as $g$ is a feasible function and $|S| \leq |\pi|$. If $S$ is the instance $(\, \Rightarrow \Box A \to A)$ of the axiom $T_a$, then $g(S)=(\, \Rightarrow g(A) \to g(A))$. Again, similar to the proof of Lemma \ref{lemG*1}, there is a feasible function $h$ such that 
$h(g(S))$ is a $G$-proof of $(\, \Rightarrow g(A) \to g(A))$. Define $f(\pi)=h(g(S))$. As both $g$ and $h$ are feasible, $T_f(\pi)$ is polynomial in $|S|$ and hence in $|\pi|$.

The case that the last rule in $\pi$ is in $G$ is similar to the same case in the proof of Lemma \ref{lemG*1}. If the last rule in $\pi$ is $(K_{\Box})$, the proof is easy. The only interesting case is when the last rule in $\pi$ is $(K_{\Diamond})$. Then, $S$ is of the form $\Box \Gamma, \Diamond A \Rightarrow \Diamond B$ and the premise of the rule is of the form $S'=(\Gamma, A \Rightarrow B)$ with the proof $\pi'$. By recursion, $G \vdash^{f(\pi')} g(S')$. As $G$ is strong, it feasibly proves $(\Diamond L)$ by Definition \ref{dfnAdmissiblyStrong}. Therefore, there is a feasible algorithm $h$ such that $h(g(S'), g(S))$ is a $G$-proof of the rule
\begin{center}
\begin{tabular}{c c c}
 \AxiomC{$g(\Gamma), g(A) \Rightarrow g(B)$}
\UnaryInfC{$ g(\Gamma), \Diamond g(A) \Rightarrow \Diamond g(B)$}
 \DisplayProof
 \end{tabular}
\end{center}
Add $h(g(S'), g(S))$ to the end of $f(\pi')$ and call this proof $f(\pi)$. It is clear that $f(\pi)$ is a $G$-proof for $g(S)$. Note that as both $g$ and $h$ are polynomial time computable, computing $h(g(S'), g(S))$ takes $(|S|+|S'|)^{O(1)}$ many steps. As $|S|, |S'| \leq |\pi|$, then the time of the additional part is $|\pi|^{O(1)}$. Hence, $T_f(\pi) \leq T_f(\pi')+|\pi|^{O(1)}$. This completes the recursive construction of $f(\pi)$. Finally, using the upper bound on $T_f(\pi)$, established in each case, proving the feasibility of $f$ is clear and similar to Lemma \ref{lemG*1}.

For the last part, the feasible conservativity of $G_r$ over $G$ is clear from the first part. For $G_i$, the claim is clear as its rules are contained in $G_r$.
\end{proof}


\begin{thm} \label{FeasibleHarropForBoxFree}
Let $G$ be either a $T$-free or a $T$-full constructive sequent calculus over the language $\mathcal{L}_{\Diamond}$. Then, $G$ has the feasible Visser-Harrop property.
\end{thm}
\begin{proof}
First, by Corollary \ref{thm: p-simulation of admissibly strong}, $G$ is pd-equivalent to $\mathbf{BLL}+\mathcal{C}$, for a finite set $\mathcal{C}$ of constructive formulas. Therefore, using Remark \ref{RemarkOnTransportabilityOfTfree}, it is enough to prove the claim for $G=\mathbf{BLL}+\mathcal{C}$.  Hence, $G$ has all the rules of $\mathbf{LJ}$ as its primitive rules and $G_i$ and $G_r$ are strong over $\mathcal{L}$. Now, to prove the claim, we first show that if $G$ is $T$-full (resp. $T$-free) over $\mathcal{L}_{\Diamond}$, then $G_{r}$ (resp. $G_i$) is $T$-full (resp. $T$-free) over $\mathcal{L}$. 
For $T$-fullness, as $G$ is $T$-full, the axiom $T_b$ is provable in it. Therefore, $T_a$ and $T_b$ are both provable in $G_{r}$. Moreover, for the validity in the reflexive node frame, $\mathcal{K}_r$, assume $G_{r} \vdash S$. Then, as $G$ is $T$-full, it is strong over $\mathcal{L}_{\Diamond}$. Hence, by Lemma \ref{lemG*2}, $g(S)$ is provable in $G$. Since $G$ is valid in $\mathcal{K}_r$, the sequent $g(S)$ is valid in $\mathcal{K}_r$. However, $\Box A$ and $A$ are equivalent in $\mathcal{K}_r$, for any $A \in \mathcal{L}$. Thus, $S$ and $g(S)$ are equivalent in $\mathcal{K}_r$. Therefore, $S$ is also valid in $\mathcal{K}_r$, and by Definition \ref{Def: T-free T-full calculus}, $G_r$ is $T$-full over $\mathcal{L}$.

For $T$-freeness, we have to show that if $G$ is $T$-free over $\mathcal{L}_{\Diamond}$, then $G_i$ is $T$-free over $\mathcal{L}$. As $G_i$ is strong over $\mathcal{L}$, we only have to show that if $G_i \vdash S$, then $S$ is valid in the irreflexive node frame, $\mathcal{K}_i$. First, notice that as $G=\mathbf{BLL}+\mathcal{C}$, we have $G_i=\mathbf{CK}+\{\Diamond L\}+\mathcal{C}$. Now, we prove the claim by induction on the structure of the proof of $S$ in $G_i$. Let $S$ be an instance of an axiom in $\mathcal{C}$. Then, as all the axioms of $\mathcal{C}$ are provable in $G$ and $G$ is $T$-free over $\mathcal{L}_{\Diamond}$, then $S$ is valid in $\mathcal{K}_i$. Now, if the last rule in the proof of $S$ is an axiom or a rule in $\mathbf{CK}+\{\Diamond L\}$,  there is nothing to prove as they respect the validity in $\mathcal{K}_i$.

We showed that if $G$ is $T$-free (resp. $T$-full) over $\mathcal{L}_\Diamond$, then $\G_{i}$ (resp. $\G_{r}$) is $T$-free (resp. $T$-full) over $\mathcal{L}$. As both $\G_{i}$ and $\G_{r}$ are clearly constructive, $\G_{i}$ (resp. $\G_{r}$) has the feasible Visser-Harrop property  by Corollary \ref{thmKolli}. Similar to the proof of Theorem \ref{FeasibleHarropForDiamonFree}, as $G_i$ and $\G_r$ are feasible conservative over $G$ by Lemma \ref{lemG*2}, it is easy to prove the feasible Visser-Harrop property for $G$.
\end{proof}

\begin{cor} \label{ConcreteApplicationBoxFree}
(Positive application) Let $\mathcal{A}$ be a finite set of $\Box$-free axioms in Table \ref{tableAxiom} such that $\mathcal{A}$ does not include $D_b$. Then, the calculus $\mathbf{BLL}+\mathcal{A}$, specially the sequent calculus $\mathbf{PLL}$ for the propositional lax logic, enjoys the feasible Visser-Harrop property and hence the feasible disjunction property. Consequently, the logic of any of these calculi has the Visser-Harrop property.
\end{cor}

\begin{proof}
By Example \ref{ExamOfTfreeTfullSequent}, $\mathbf{BLL}+\mathcal{A}$ is either $T$-free or $T$-full. As it is clearly constructive, by Theorem \ref{FeasibleHarropForBoxFree} we get the result.
\end{proof}

\begin{cor}(Negative application)
Let $L$ be either a $T$-free or a $T$-full logic over $\mathcal{L}_{\Diamond}$.  If there is at least one Visser's rule that is not admissible in $L$, then $L$ does not have a constructive sequent calculus.
\end{cor}
\begin{proof}
The proof is similar to the proof of Corollary \ref{NegAppMain}.
\end{proof}

\begin{cor}
Let $L\neq \mathsf{IPC}$ be an intermediate logic and $\mathcal{A}$ be a finite set of $\Box$-free axioms in Table \ref{tableAxiom} such that $\mathcal{A}$ does not include $D_b$. Then, the logic $L\mathsf{BLL}+\mathcal{A}$ does not have a constructive sequent calculus.
\end{cor}

\begin{proof}
The proof is similar to the proof of Corollary \ref{NegAppSpecialMain}.
\end{proof}

\subsection{Propositional Fragment} \label{SubsectionPropositional}
Define the forgetful function $h:\mathcal{L} \to \mathcal{L}_p$ as $h(p)=p$, for any atom $p$ (including $\bot$ and $\top$), $h(A \circ B)= h(A) \circ h(B)$, for $\circ \in \{\wedge, \vee, \to\}$, $h(\Box A)=\top$, and $h(\Diamond A)=\bot$. Let $G$ be a strong sequent calculus over $\mathcal{L}_{p}$. Define $\G_m$ over $\mathcal{L}$ as $G+ \{K_{\Box}, K_{\Diamond}\}$. The following connects $G_m$ to $G$ via the translation $h$.
\begin{lem}\label{lemG*3}
If $G$ is a strong sequent calculus over $\mathcal{L}_{p}$, then there exists a feasible algorithm that reads a $G_m$-proof of a sequent $S$ over $\mathcal{L}$ and outputs a $G$-proof of $h(S)$. Consequently, $\G_m$ is feasibly conservative over $G$.
\end{lem}
\begin{proof}
The proof is similar to the proof of Lemma \ref{lemG*1}.  
\end{proof}

\begin{thm} \label{FeasibleHarropForProp}
Let $G$ be a strong constructive sequent calculus over $\mathcal{L}_{p}$. Then, $G$ has the feasible Visser-Harrop property.
\end{thm}
\begin{proof}
If $G$ is inconsistent, it has the feasible Visser-Harrop property. Hence, assume $G$ is consistent. As the first step, we prove that if $G \vdash S=(\Gamma \Rightarrow \Delta)$, then $S$ is classically valid. Assume otherwise. 
Then, there is a substitution $\sigma$ (mapping atoms to $\bot$ and $\top$) such that $\sigma(A)$ is classically valid, where $A=\neg [\bigwedge \Gamma \to \bigvee \Delta]$. As a consequence of Glivenko's theorem, as $\sigma(A)$ is a negative propositional formula, its classical validity implies its intuitionistic provability. Hence, $\mathbf{LJ} \vdash \, \Rightarrow \sigma(A)$ which implies $\mathbf{LJ} \vdash \sigma(\bigwedge \Gamma) \to \sigma(\bigvee \Delta) \Rightarrow \bot$. As $G$ is strong, it proves all the rules of $\mathbf{LJ}$. Hence, $G \vdash \sigma(\bigwedge \Gamma) \to \sigma(\bigvee \Delta) \Rightarrow \bot$. As $G \vdash \Gamma \Rightarrow \Delta$, we have $G \vdash \sigma(\Gamma) \Rightarrow \sigma(\Delta)$, by substitution. Therefore, $G \vdash \, \Rightarrow \sigma(\bigwedge \Gamma) \to \sigma(\bigvee \Delta)$, as all the rules of $\mathbf{LJ}$ are provable in $G$. Hence, by the cut rule, also provable in $G$, we have $G \vdash \, \Rightarrow \bot$ which is impossible. Hence, $\Gamma \Rightarrow \Delta$ is classically valid.

Now, using what we showed, we prove that $G_m$ is $T$-free over $\mathcal{L}$. Again, by Corollary \ref{thm: p-simulation of admissibly strong}, $G$ is pd-equivalent to $\mathbf{LJ}+\mathcal{C}$, for a finite set $\mathcal{C}$ of constructive formulas. Therefore, using Remark \ref{RemarkOnTransportabilityOfTfree}, it is enough to prove the claim for $G=\mathbf{LJ}+\mathcal{C}$.  Therefore, $G$ has all the rules of $\mathbf{LJ}$ as its primitive rules and hence $G_m$ is strong over $\mathcal{L}$. For the other condition, if $S$ is provable in $G_m$, then by Lemma \ref{lemG*3}, $h(S)$ is provable in $G$. Thus, by the first part of the present proof, $h(S)$ is classically valid and hence valid in the irreflexive node frame $\mathcal{K}_i$. However, $\mathcal{K}_i$ reads $\Diamond B$ as $\bot$ and $\Box B$ as $\top$, for any $B \in \mathcal{L}$. Hence, for any $C \in \mathcal{L}$, the formula $C$ is valid in $\mathcal{K}_i$ if and only if $h(C)$ is valid there. Thus, $S$ is also valid $\mathcal{K}_i$. Therefore, $G_m$ is $T$-free over $\mathcal{L}$. 

Finally, as $G_m$ is clearly constructive, by Corollary \ref{thmKolli}, $G_m$ has the feasible Visser-Harrop property. By Lemma \ref{lemG*3}, it is easy to derive the feasible Visser-Harrop property for $G$.
\end{proof}

\begin{cor} \label{ConcreteApplicationProp}
(Positive application) $\LJ$ has the feasible Visser-Harrop property and hence feasible disjunction property.
\end{cor}

\begin{proof}
By Definition \ref{dfnAdmissiblyStrong}, $\LJ$ is strong over $\mathcal{L}_p$. As it is clearly constructive, by Theorem \ref{FeasibleHarropForProp}, we get the result.
\end{proof}

\begin{cor}(Negative application) \label{NegAppProp}
Let $L \supsetneq \mathsf{IPC}$ be a logic over $\mathcal{L}_p$. Then, $L$ does not have a constructive sequent calculus.
\end{cor}
\begin{proof}
The proof is similar to the proof of Corollary \ref{NegAppSpecialMain}.
\end{proof}

Characterizing $\mathsf{IPC}$ by the form of its sequent calculus, Corollary \ref{NegAppProp} shows that it is the only intermediate logic with a constructive sequent calculus. Moreover, as the constructive rules follow the constructive heuristics, one may read Corollary \ref{NegAppProp} as a justification that $\mathsf{IPC}$ is the only intermediate logic that is constructively acceptable.

\section{Conclusion and Future Work}
Over the modal language $\mathcal{L}=\{\wedge, \vee, \to, \top, \bot, \Box, \Diamond\}$ and its fragments, we introduced a family of sequent-style rules called the constructive rules. The main motivation was to capture the constructively valid axioms and rules over the language $\mathcal{L}$. We managed to accomplish this goal by allowing formulas in which disjunction, diamond and implication appear in a restricted form. Then, we proved that for any sequent calculus $G$ consisting of these constructive rules and possibly the rules $(K_{\Box})$ and $(K_\Diamond)$, if $G$ is either $T$-free or $T$-full, then $G$ has the feasible Visser-Harrop property, which is a generalization of the feasible version of the admissibility of Visser's rules. Here, $T$-freeness (resp. $T$-fullness) of either a logic or a proof system is a mild technical condition that essentially states that the logic or the system is strong enough while it is valid in an irreflexive (resp. a reflexive) one node Kripke frame.
Using this machinery, we first showed that the sequent calculi for various intuitionistic modal logics enjoy the feasible Visser-Harrop property. The generality of our constructive rules, then, was witnessed by the fact that the main result is applicable to the sequent calculi for several well-known intuitionistic modal logics. Second, we used the theorem to show that if a $T$-free or a $T$-full logic does not admit Visser's rules, it cannot have a sequent calculus consisting of constructive rules and the rules $(K_{\Box})$ and $(K_{\Diamond})$. 

For the future work, it is important to emphasize that the machinery provided here is quite general and is not restricted to the modal language. Consequently, the next natural step is to generalize the constructive rules from the modal language to more complex languages, specially the first-order language. Doing so, the technique then, can be used as a mathematical tool to prove a low complexity version of the disjunction and the existence property in constructive theories. Moreover, on the philosophical level, it can also provide a general form for the constructively acceptable rules in a more complex settings of arithmetical and set-theoretical languages.
As another possible expansion of the present study, it is also interesting to see how the form of constructive rules can be relaxed to capture the disjunction property rather than the full Visser's rules. As there are many intermediate logics
with the disjunction property, such an investigation can be interesting.

\bibliographystyle{plain}
\bibliography{newbib}

\end{document}